\documentclass[10pt,a4paper]{scrartcl}

\usepackage{amssymb}
\usepackage{amsmath}
\usepackage{amsfonts}
\usepackage{bbm}
\usepackage{amsthm}
\usepackage{mathrsfs}
\usepackage[hidelinks]{hyperref}\usepackage{color}
\usepackage[margin=2cm]{geometry}
\usepackage[all,cmtip]{xy}
\usepackage[utf8]{inputenc}
\usepackage{graphicx}
\usepackage{varwidth}

\let\emph\undefined\newcommand{\emph}[1]{\textsl{#1}}
\usepackage{upgreek}
\usepackage{rotating}
\usepackage{tikz}
\usetikzlibrary{matrix,arrows,decorations.pathmorphing,shapes.geometric}
\usepackage{tikz-cd}
\usepackage{needspace}
\newcommand{\spaceplease}{\needspace{5\baselineskip}}
\usetikzlibrary{decorations.markings}
     
\usetikzlibrary{backgrounds}

\tikzstyle{tikzfig}=[baseline=-0.25em,scale=0.5]

\pgfkeys{/tikz/tikzit fill/.initial=0}
\pgfkeys{/tikz/tikzit draw/.initial=0}
\pgfkeys{/tikz/tikzit shape/.initial=0}
\pgfkeys{/tikz/tikzit category/.initial=0}

\pgfdeclarelayer{edgelayer}
\pgfdeclarelayer{nodelayer}
\pgfsetlayers{background,edgelayer,nodelayer,main}

\tikzstyle{none}=[inner sep=0mm]

\newcommand{\tikzfig}[1]{%
	{\tikzstyle{every picture}=[tikzfig]
		\IfFileExists{#1.tikz}
		{\input{#1.tikz}}
		{%
			\IfFileExists{./figures/#1.tikz}
			{\input{./figures/#1.tikz}}
			{\tikz[baseline=-0.5em]{\node[draw=red,font=\color{red},fill=red!10!white] {\textit{#1}};}}%
	}}%
}

\tikzstyle{every loop}=[]

\usepackage{mathtools}
\mathtoolsset{showonlyrefs}
\usepackage{epstopdf}

\newtheoremstyle{mytheorem}
  {\topsep}
  {\topsep}
  {\slshape}
  {0pt}
  {\bfseries}
  {.}
  { }
  {\thmname{#1}\thmnumber{ #2}\thmnote{ {\normalfont\slshape(#3)}}}

  \newtheoremstyle{mydefinition}
    {\topsep}
    {\topsep}
    {\normalfont}
    {0pt}
    {\bfseries}
    {.}
    { }
    {\thmname{#1}\thmnumber{ #2}\thmnote{ {\normalfont\slshape(#3)}}}

\theoremstyle{mytheorem}
\newtheorem{theorem}{Theorem}[section]

\makeatletter
\newtheorem*{rep@theorem}{\rep@title}
\newcommand{\newreptheorem}[2]{%
	\newenvironment{rep#1}[1]{%
		\def\rep@title{#2 \ref{##1}} \slshape%
		\begin{rep@theorem}}%
		{\end{rep@theorem}}}
\makeatother

\newreptheorem{theorem}{Theorem}
\newreptheorem{corollary}{Corollary}
\newtheorem{lemma}[theorem]{Lemma}
\newtheorem{proposition}[theorem]{Proposition}
\newtheorem{corollary}[theorem]{Corollary}
\theoremstyle{mydefinition}
\newtheorem{definition}[theorem]{Definition}
\newtheorem{exx}[theorem]{Example}
\newenvironment{example}
{\pushQED{\qed}\exx}
{\popQED\endexx}
\newtheorem{remm}[theorem]{Remark}
\newenvironment{remark}
{\pushQED{\qed}\remm}
{\popQED\endremm}
\numberwithin{equation}{section}
\usepackage{enumitem}

\newenvironment{xenumerate}{\begin{enumerate}[topsep=2pt,parsep=2pt,partopsep=2pt,itemsep=0pt,label={\normalfont(\arabic*)}]\itemsep0pt}{\end{enumerate}}
\newenvironment{pnum}{\begin{enumerate}[topsep=2pt,parsep=2pt,partopsep=2pt,itemsep=0pt,label={(\roman{*})}]}{\end{enumerate}}

\DeclareMathSymbol{\Phiit}{\mathalpha}{letters}{"08}\let\Phi\undefined\newcommand{\Phi}{\Phiit}
\DeclareMathSymbol{\Psiit}{\mathalpha}{letters}{"09}\let\Psi\undefined\newcommand{\Psi}{\Psiit}
\DeclareMathSymbol{\Sigmait}{\mathalpha}{letters}{"06}\let\Sigma\undefined\newcommand{\Sigma}{\Sigmait}
\DeclareMathSymbol{\Xiit}{\mathalpha}{letters}{"04}
\DeclareMathSymbol{\Lambdait}{\mathalpha}{letters}{"03}\let\Lambda\undefined\newcommand{\Lambda}{\Lambdait}
\DeclareMathSymbol{\Piit}{\mathalpha}{letters}{"05}\let\Pi\undefined\newcommand{\Pi}{\Piit}
\DeclareMathSymbol{\Gammait}{\mathalpha}{letters}{"00}\let\Gamma\undefined\newcommand{\Gamma}{\Gammait}
\DeclareMathSymbol{\Omegait}{\mathalpha}{letters}{"0A}\let\Omega\undefined\newcommand{\Omega}{\Omegait}
\DeclareMathSymbol{\Upsilonit}{\mathalpha}{letters}{"07}\let\Upsilon\undefined\newcommand{\Upsilon}{\Upilonit}
\DeclareMathSymbol{\Thetait}{\mathalpha}{letters}{"02}\let\Theta\undefined\newcommand{\Theta}{\Thetait}

\def\Hom{\mathrm{Hom}}

\def\id{\mathrm{id}}
\def\SL{\operatorname{SL}}
\let\to\undefined\newcommand{\to}{\longrightarrow}
\let\mapsto\undefined\newcommand{\mapsto}{\longmapsto}
\newcommand{\catf}[1]{\mathsf{#1}}
\newcommand{\Proj}{\operatorname{\catf{Proj}}}
\newcommand{\Mod}{\catf{Mod}}
\newcommand{\Map}{\catf{Map}}

\def\op{\mathrm{op}}

\newcommand{\la}[1]{\xleftarrow{\   #1    \ }}
\newcommand{\ra}[1]{\xrightarrow{\   #1    \ }}
\def\Ch{\catf{Ch}_k}
\def\Grpd{\catf{Grpd}}
\def\Cat{\catf{Cat}}

\let\F\undefined
\let\S\undefined
\newcommand{\F}{\text{F}}
\newcommand{\S}{\text{S}}
\newcommand{\B}{\text{B}}
\newcommand{\Z}{\text{Z}}
\newcommand{\barF}{\bar{\text{F}}}
\newcommand{\barS}{\bar{\text{S}}}
\newcommand{\mfc}{\mathfrak{F}_{\cat{C}}}
\newcommand{\mfg}{\mathfrak{F}_{\Mod_k D(G)}}
\newcommand{\mfcm}{\mathfrak{F}_{\cat{C}}^{\catf{m}}}

\newcommand{\hocolim}{\operatorname{hocolim}}

\newcommand{\hocolimsub}[1]{\underset{#1}{\operatorname{hocolim}}\,}
\def\PBun{\catf{PBun}}

\newcommand{\K}{\catf{K}}
\newcommand{\Vect}{\catf{Vect}}
\newcommand{\block}{\catf{B}}
\newcommand{\ordblock}{\catf{b}}
\newcommand{\cof}{\mathsf{Q}}

\newcommand{\im}{\operatorname{i}}

\newcommand{\cat}[1]{\mathcal{#1}}

\newcommand{\Ext}{\operatorname{Ext}}
\newcommand{\flint}{\int_{\text{\normalfont f}\mathbb{L}}}

\newcommand{\Surf}{\catf{Surf}}

\newcommand{\mSurf}{\catf{mSurf}}
\let\C\undefined\newcommand{\C}{\catf{C}}\newcommand{\Chat}{\widehat{\catf{C}}}
\newcommand{\m}{\catf{m}}

\newcommand{\Surfc}{\catf{Surf}^\catf{c}}

\newcommand{\DW}{\mathsf{Z}_G^\mathbb{L}}

\def\cut{\operatorname{cut}}

\newcommand{\FM}{\catf{M}}
\newcommand{\M}{\catf{M}}
\newcommand{\CMhat}{\cat{C}\text{-}\!\!\int \Mhat}
\newcommand{\CMC}{\cat{C}\text{-}\!\!\int \M^\C}
\newcommand{\CM}{\cat{C}\text{-}\!\!\int \M}
\newcommand{\FMhat}{\widehat{\catf{M}}}
\newcommand{\Mhat}{\widehat{\catf{M}}}

\newcommand{\lint}{\int_\mathbb{L}}

\newcommand{\Aut}{\operatorname{Aut}}

\definecolor{Blue}  {rgb} {0.282352,0.239215,0.803921}
\definecolor{Green} {rgb} {0.133333,0.545098,0.133333}
\definecolor{Red}   {rgb} {0.803921,0.000000,0.000000}
\definecolor{Violet}{rgb} {0.580392,0.000000,0.827450}

\newcommand{\myskip}{\\[0.5ex]}

\setkomafont{caption}{\small\slshape}

\setkomafont{subparagraph}{\normalfont\slshape}
\setkomafont{section}{\rmfamily\Large}
\setkomafont{subsection}{\rmfamily}

\usepackage{tocloft}

\newtheorem*{theorem*}{Theorem}
\newtheorem*{corollary*}{Corollary}

\begin{document}
\vspace*{-10mm}
	\begin{flushright}
		\small
		{\sffamily [ZMP-HH/20-12]} \\
		\textsf{Hamburger Beiträge zur Mathematik Nr.~835}\\
		\textsf{April 2020}
	\end{flushright}
	
\vspace{8mm}
	
	\begin{center}
		\textbf{\large{Homotopy Coherent Mapping Class Group Actions and Excision  \\[0.5ex] \large for Hochschild Complexes of Modular Categories}}\\
		\vspace{6mm}
		{\large Christoph Schweigert and  Lukas Woike }
		
		\vspace{3mm}
		
	\normalsize
		{\slshape Fachbereich Mathematik\\ Universit\"at Hamburg\\
			Bereich Algebra und Zahlentheorie\\
			Bundesstra\ss e 55\\  D\,--\,20\,146\, Hamburg }
	\end{center}
\vspace*{1mm}
	\begin{abstract}\noindent 
		Given any modular category $\mathcal{C}$ over an algebraically closed field $k$, we extract a sequence $(M_g)_{g\geq 0}$ of $\mathcal{C}$-bimodules and show that the Hochschild chain complex $CH(\mathcal{C};M_g)$ of $\mathcal{C}$ with coefficients in $M_g$ carries a canonical homotopy coherent projective action of the mapping class group of the surface of genus $g+1$. The ordinary  Hochschild complex of $\mathcal{C}$  corresponds to $CH(\mathcal{C};M_0)$.
		
		This result is obtained as part of the following more comprehensive topological structure: We construct a symmetric monoidal functor $\mathfrak{F}_{\mathcal{C}}:\mathcal{C}\text{-}\mathsf{Surf}^{\mathsf{c}}\to\mathsf{Ch}_k$ with values in chain complexes over $k$ defined on a symmetric monoidal category of surfaces whose boundary components are labeled with projective objects in $\mathcal{C}$. The functor $\mathfrak{F}_{\mathcal{C}}$ satisfies an excision property which is formulated in terms of homotopy coends. In this sense, any modular category  gives naturally rise to a modular functor with values in chain complexes. In zeroth homology, it recovers Lyubashenko's mapping class group representations.
		
		The chain complexes in our construction are explicitly computable by choosing a marking on the surface, i.e.\ a cut system and a certain embedded graph. For our proof, we replace the connected and simply connected groupoid of cut systems that appears in the Lego-Teichmüller game by a contractible Kan complex.
\end{abstract}

\tableofcontents
\normalsize
\newpage

\section{Introduction and summary}
It is an important insight of the last decades that the deep connections between low-dimensional topology and representation theory
can be profitably used in both directions.

Given a certain type of representation category (typically a monoidal category with plenty of additional structure and often subject to finiteness conditions), one can construct topological invariants: Via a surgery construction, 
 a semisimple modular category gives rise to the Reshetikhin-Turaev invariants \cite{rt1,rt2,turaev}.
 These include the Turaev-Viro invariants \cite{turaevviro}
 that can be obtained from
 a spherical fusion category
 via a state sum construction. 
 These constructions actually extend to three-dimensional topological field theories in the sense of \cite{atiyah}. 
 
By a change of perspective, such constructions can be read backwards in the sense that low-dimensional topology can be used to construct meaningful algebraic quantities from a representation category (or a related algebraic object). As a second and often even more important step, one will then use topology to establish properties of these algebraic quantities. In a lot of cases, such topological manipulations
are not only more conceptual, but also turn out to be easier than purely algebraic manipulations.\myskip

In this article, we present such a topological perspective on modular categories --- with a special emphasis on the non-semisimple case.
On the one hand, it has been known
for more than 25 years that one can construct from a not necessarily semisimple modular category a system of projective mapping class group representations  \cite{lubacmp,luba,lubalex} that can be used to build a \emph{modular functor} with values in vector spaces.
On the other hand, it is clear that non-semisimplicity will result in  a non-trivial homological algebra, an important aspect that from a Hopf algebraic perspective appears e.g.\ in \cite{gk,schau,bichon}. 
In this paper, we unravel within a homotopy coherent framework the interplay of 
the homological algebra of a modular category and low-dimensional topology.
This leads to homotopy coherent projective mapping class group actions and excision
results for certain Hochschild complexes of a modular category. Our methods will allow us to systematically trace
back this structure to a clear topological origin.
\myskip

Let us first recall the notion of a modular category: A \emph{finite category}  is a linear Abelian category (over a fixed field $k$ that we will assume to be algebraically closed throughout this article)
with finite-dimensional morphism spaces, enough projectives, finitely many isomorphism classes of simple objects such that every object has finite length. 
A \emph{finite tensor category} \cite{etinghofostrik} is a tensor category (linear Abelian rigid monoidal category with simple unit) whose underlying linear category is a finite category. A finite tensor category that is also equipped with a braiding and a ribbon structure, is called a \emph{finite ribbon category}.
For a braided finite tensor category $\cat{C}$ with braiding $c$, one defines the \emph{Müger center}
as the full subcategory of $\cat{C}$ spanned by all transparent objects, i.e.\ all objects $X\in\cat{C}$  satisfying $c_{Y,X}c_{X,Y}=\id_{X\otimes Y}$ for every $Y\in\cat{C}$.
The braiding (and then also the braided finite tensor category) is called \emph{non-degenerate} if its Müger center is trivial, i.e.\ spanned by the monoidal unit under finite direct sums.
 A \emph{modular category} is a non-degenerate finite ribbon category.
Modular categories appear as categories of modules over certain Hopf algebras \cite{turaev,egno}, vertex operator algebras \cite{huang} or nets of observable algebras \cite{klm}, see in particular \cite{lo,glo,cgr}   for the non-semisimple case. 
Recall that a modular category (more generally a finite tensor category) is  semisimple if and only if all of its objects are projective.

For a treatment of modular categories in topological terms,
non-semisimplicity is a major challenge: From a semisimple modular category, a once-extended three-dimensional oriented topological field theory can be built via the 
Reshetikhin-Turaev construction \cite{rt1,rt2,turaev}. In the non-semisimple case, such a construction is not available. In fact, once-extended three-dimensional oriented topological field theories are equivalent to semisimple modular categories \cite{BDSPV15} by evaluation on the circle 
(interestingly enough, if one changes the bordism category to the extent that it loses rigidity, some constructions are still possible \cite{gai}).

For this reason, we will work throughout this article with a different kind of topological structure that comprises slightly less
than the notion of a once-extended oriented three-dimensional topological field theory: the notion of \emph{modular functor} \cite{tillmann,baki}, or rather a suitable version thereof. 
Roughly, a modular functor is a consistent system of (projective) mapping class group representations. These are classically valued in vector spaces, but in order to capture the homological algebra of a modular category, we will consider a differential graded version. 

Let us discuss the definition of a modular functor in more detail:
An extended surface $\Sigma$ is a compact oriented 
two-dimensional smooth manifold (possibly with boundary) with the choice of a point on each boundary component and an orientation on each boundary component (which may either agree or disagree with the orientation induced by the surface making this boundary component either \emph{outgoing} or \emph{incoming}).
For a set $\mathfrak{X}$ (to be thought of as label set), we 
define the category $\mathfrak{X}\text{-}\Surfc$ whose objects are extended surfaces with an element in $\mathfrak{X}$ for each boundary component
and whose morphisms are generated by mapping classes and sewings; the superscript $\textsf{c}$ indicates that some relations between mapping classes will just be satisfied up to an additional central generator (this is to allow for (a certain type of) projective actions).
Disjoint union endows $\mathfrak{X}\text{-}\Surfc$ with a symmetric monoidal structure. We refer to Section~\ref{secsurfacecat}
 for the detailed definition of this surface category.

A \emph{modular functor (with values in chain complexes over $k$)} 
is defined as a symmetric monoidal functor $\mathfrak{X}\text{-}\Surfc \to \Ch$ from $\mathfrak{X}\text{-}\Surfc$ to the category of chain complexes over $k$ satisfying an excision property formulated in terms of homotopy coends (Definition~\ref{predefdmf}).
We refer to the values of a modular functor as \emph{conformal blocks}.
The notion of symmetric monoidal functor may of course be relaxed from a strict version to a homotopy coherent version by considering instead of $\mathfrak{X}\text{-}\Surfc$ a suitable resolution.

While considering  modular functors 
with values in chain complexes as a generalization of the classical notion is certainly natural, it is not clear that a non-trivial class of examples exists.
The goal of this article is to prove that modular categories produce such a non-trivial class of examples of differential graded modular functors.
This class of examples will lead to concrete applications to Hochschild complexes of modular categories. \myskip

Let us state the main \emph{topological}
result: To this end, we fix a modular category $\cat{C}$ and consider the surface category (as defined above) for the label set $(\Proj \cat{C})_0$, the set of projective objects of $\cat{C}$. We denote this surface category by $\cat{C}\text{-}\Surfc$.
The projectivity assumption for boundary labels is not essential and is used here to simplify the presentation, see Remark~\ref{remproj}.

\begin{reptheorem}{thmmain}[Main Theorem, Part I]
	Any modular category $\cat{C}$ 
	gives rise in a canonical way to a  modular functor 
	\begin{align}
	\mfc\ :\ \cat{C}\text{-}\Surfc\to\Ch  \label{eqnmfcintro}
	\end{align}
	with values in chain complexes.
\end{reptheorem}
The specific model for $\mfc$ that we provide  will actually be strictly functorial in $\cat{C}\text{-}\Surfc$; in particular, the resulting projective mapping class group actions are strict.
However, below we will transfer these actions along  equivalences (i.e.\ quasi-isomorphisms, see also our conventions on page~\pageref{pageconventions}) to certain Hochschild complexes leading to non-strict actions.

The category $\cat{C}$ (or rather its subcategory $\Proj \cat{C}$) enters the definition of $\cat{C}\text{-}\Surfc$ just through its object \emph{set}, but it is actually recovered as a \emph{linear category} by evaluation of \eqref{eqnmfcintro} on the cylinder; for more details we refer to Section~\ref{secdefdmf}. 

As an example, we explicitly describe the modular functor for modules over the Drinfeld double of a finite group $G$ in positive characteristic as chains on groupoids of $G$-bundles over surfaces (Example~\ref{exdouble}).\myskip

The second part of the main result is concerned with the concrete computation of the  modular functor $\mfc$ on a given extended surface $\Sigma$ with projective boundary label $\underline{X}$ (of course, $\Sigma$ can be closed and hence $\underline{X}$ the empty collection):  We  choose
an auxiliary datum, namely a marking $\Gamma$ on $\Sigma$ (roughly: a cut system and an embedded graph).
By a  prescription using the combinatorial data provided by the marking, the morphism spaces of $\cat{C}$ and homotopy coends
we  define a chain complex $\block_\cat{C}^{\Sigma,\Gamma}(\underline{X})$, the so-called \emph{marked block}
for   $(\Sigma,\Gamma,\underline{X})$; see Section~\ref{derivedconformalblockssec} for details.
For the example of the closed torus and a sufficiently simple marking, this complex is given by the (normalized) chains on the simplicial vector space
\begin{center}
	{\footnotesize
		\begin{equation}
		\begin{tikzcd}
		\dots \ar[r, shift left=6]  \ar[r, shift left=2]
		\ar[r, shift right=6]  \ar[r, shift right=2]
		& 	\displaystyle \bigoplus_{  \substack{X_0,X_1 ,X_2 \\ \in \Proj\cat{C}}}  \cat{C}(X_1,X_0)\otimes\cat{C}(X_2,X_1)\otimes  \cat{C}(X_0,X_2)
		\ar[l, shift left=4]  \ar[l]
		\ar[l, shift right=4]  
		\ar[r, shift left=4] \ar[r, shift right=4] \ar[r] & \displaystyle \bigoplus_{X_0,X_1 \in \Proj\cat{C}}  \cat{C}(X_1,X_0)\otimes \cat{C}(X_0,X_1) \ar[r, shift left=2] \ar[r, shift right=2]
		\ar[l, shift left=2] \ar[l, shift right=2]
		& \displaystyle \bigoplus_{X_0 \in \Proj \cat{C}} \cat{C}(X_0,X_0)\ ,  \ar[l] \\
		\end{tikzcd}\label{eqnhochschildgenus1}
		\end{equation}}
\end{center}
\normalsize
where $\cat{C}(X,Y)$ is the space of morphisms from $X$ to $Y$,  and the face and degeneracy maps are given by composition in $\cat{C}$ and insertion of identities, respectively, see Example~\ref{exhochschild}.
Hence, it is given by the Hochschild complex   \cite{mcarthy,keller} for the category of projective objects in $\cat{C}$.

\begin{reptheorem}{thmmain}[Main Theorem, Part II]
	After any choice of marking $\Gamma$ for an extended surface $\Sigma$ with projective boundary label $\underline{X}$, there is a canonical equivalence
	\begin{align} \block_\cat{C}^{\Sigma,\Gamma}(\underline{X})\ra{\simeq} \mfc(\Sigma,\underline{X}) \ . \label{eqnequivblock}
	\end{align}
\end{reptheorem}
Note that this equivalence is canonical \emph{after} the choice of the marking; the marking itself is not canonical.\myskip

The modular functor $\mfc:\cat{C}\text{-}\Surfc \to \Ch$ relates 
to classical constructions of modular functors with values in vector spaces:
\begin{itemize}
	\item 
	\emph{Reshetikhin-Turaev construction}. As mentioned above, from a \emph{semisimple} modular category, one can build a once-extended  three-dimension\-al oriented topological field theory 
	via the Reshetikhin-Turaev construction \cite{rt1,rt2,turaev}. It is not surprising that
	the differential graded modular functor $\mfc : \cat{C}\text{-}\Surfc \to \Ch$ for a semisimple modular category does not add anything to the picture: It has non-trivial homology only in degree zero and recovers in zeroth homology the modular functor obtained by restriction of the Reshetikhin-Turaev topological field theory to surfaces.

	\item \emph{Lyubashenko construction.}
	The classification result from \cite{BDSPV15} tells us that from a non-semisimple modular category, we cannot obtain a once-extended oriented three-dimensional topological field theory.
	However, by a remarkable result of Lyubashenko \cite{lubacmp,luba,lubalex} any  modular  category (not necessarily semisimple) still gives rise to a mapping class group representations (in the semisimple case, they agree with  the ones obtained from the Reshetikhin-Turaev construction). A key ingredient for the construction of these representations is the canonical coend $\mathbb{F}=\int^{X\in\cat{C}} X \otimes X^\vee\in\cat{C}$ that is also referred to as \emph{Lyubashenko coend}.
	The   modular functor $\mfc:\cat{C}\text{-}\Surfc \to \Ch$ will recover in zeroth homology
	the linear dual of Lyubashenko's mapping class group representations. However, the  modular functor $\mfc$ will generally have non-trivial higher homologies.
\end{itemize}

We can now provide a topological perspective on Hochschild complexes of a modular category $\cat{C}$ with coefficients in specific bimodules:
For a modular category $\cat{C}$ and $g\ge 0$, 
the evaluation of the modular functor $\mfc$ on a surface  of genus $g$ and with two oppositely oriented boundary components yields a bimodule, i.e.\ a functor 
$M_g : \cat{C}^\op\otimes\cat{C}\to\Ch$.
Up to equivalence, $M_g$ is concentrated in degree zero and given  by $M_g(X,Y)=\cat{C}(X,Y\otimes\mathbb{F}^{\otimes g})$ for $X,Y\in\cat{C}$, where $\cat{C}(-,-)$ denotes the morphism spaces of $\cat{C}$ and $\mathbb{F}=\int^{X\in\cat{C}} X \otimes X^\vee\in\cat{C}$ the canonical coend. We recall in Section~\ref{secbimodules} the definition of the Hochschild chains $CH(\cat{C};M_g)$ of $\cat{C}$ with coefficients in $M_g$ and prove:

\begin{reptheorem}{cormain2}
	For any modular category $\cat{C}$  and $g\ge 0$, the Hochschild chains $CH(\cat{C};M_g)$ with coefficients in the bimodule $M_g$  carry a canonical homotopy coherent projective action of the mapping class group $\Map(\Sigma_{g+1})$ of the closed surface of genus $g+1$.
\end{reptheorem}

The complex $CH(\cat{C};M_0)$ is the `ordinary' Hochschild complex \eqref{eqnhochschildgenus1}, and the homotopy coherent projective action of $\Map(\Sigma_1)=\SL(2,\mathbb{Z})$ was already established in \cite{dva}, see \cite{svea} for a Hopf algebraic analogue of this result on (co)homology level.
The projective mapping class group actions induced on the homologies $H_*(CH(\cat{C};M_g))$ can be related to the projective mapping class group actions on certain Ext groups in \cite{svea2} (Remark~\ref{remsvea}).

Our main result provides the following topological proof  for Theorem~\ref{cormain2}: Using the excision property for marked blocks, we observe that $CH(\cat{C};M_g)$ can be seen as the marked block for $\Sigma_{g+1}$ \emph{and a specific marking}. This makes the Hochschild complexes $CH(\cat{C};M_g)$ canonically equivalent to the  conformal block $\mfc(\Sigma_{g+1})$ thanks to \eqref{eqnequivblock}. 
The  conformal block $\mfc(\Sigma_{g+1})$ carries even a strict projective action of $\Map(\Sigma_{g+1})$. As a consequence, $CH(\cat{C};M_g)$ carries also a projective $\Map(\Sigma_{g+1})$-action through transfer which, in general, will just be homotopy coherent. 
Note that constructing directly a homotopy coherent action on the Hochschild complex $CH(\cat{C};M_g)$, i.e.\
without using the relation to $\mfc(\Sigma_{g+1})$,
would be rather involved (as the treatment of $CH(\cat{C};M_0)$ in \cite{dva} shows). 
The reason for this difficulty is clear: From a topological perspective, the complex $CH(\cat{C};M_g)$ corresponds to a specific marking, and the action of the mapping class group will not preserve this marking! Therefore, it is easier to obtain  the mapping class group action through the complex $\mfc(\Sigma_{g+1})$, which is a genuinely topological quantity. \myskip

Theorem~\ref{cormain2} implies a Hopf algebraic statement:
Let $A$ be a ribbon factorizable Hopf algebra and denote by $A_\text{coadj}^*$ the dual of $A$ equipped with the coadjoint action. Consider now for $g\ge 0$ the $A$-module $A\otimes \left(  A_\text{coadj}^*     \right)^{\otimes g}$ (tensor product in the monoidal category of $A$-modules). By multiplication from the right on the $A$-factor, this becomes an $A$-bimodule.

\begin{repcorollary}{corha}
	Let $A$ be a ribbon factorizable Hopf algebra and $g\ge 0$.
	Then the Hochschild chains of $A$ with coefficients in the $A$-bimodule $A\otimes \left(  A_\text{coadj}^*     \right)^{\otimes g}$ carry a canonical homotopy coherent projective action of the mapping class group $\Map(\Sigma_{g+1})$ of the closed surface of genus $g+1$.
\end{repcorollary}

While the proof of our Main~Theorem~\ref{thmmain} uses Lyubashenko's work on the canonical coend $\mathbb{F}=\int^{X\in\cat{C}} X \otimes X^\vee$ of a modular category $\cat{C}$ 
and 
 the $\S$-transformation, 
 it does not directly build on Lyubashenko's construction of the projective mapping class group representations in \cite{lubacmp}.
 These are based on a presentation of  mapping class groups in terms of generators and relations and seem hard to adapt to a differential graded  framework. 
Instead, we adapt the Lego Teichmüller game developed by Bakalov and Kirillov in \cite{bakifm} based on \cite{hatcherthurston,harer,grothendieck} to our purposes by replacing their connected and simply connected groupoid of markings on an extended surface by a contractible $\infty$-groupoid.

These contractible $\infty$-groupoids replacing the groupoids of markings are obtained by localizing certain categories of colored markings at uncoloring morphisms and form, at a technical level,
the backbone of our construction. They provide a dictionary between low-dimensional topology and the homological algebra of a modular category including the calculus for homotopy coends over finite tensor categories from \cite{dva}. After gluing together the $\infty$-groupoids attached to varying surfaces via the Grothendieck construction, we obtain a model for the surface category and hence can extract explicitly computable mapping class group representations. We refer to this procedure as the \emph{homotopy coherent Lego Teichmüller game}.

In slightly more technical terms, the strategy is the following: We define a \emph{category  $\Mhat(\Sigma)$ of colored markings} on an extended surface $\Sigma$ formed by markings on $\Sigma$ with the additional datum of a subset of distinguished cuts which we call \emph{colored cuts}. We require that there is at least one such colored cut per closed connected component.
We also add \emph{uncolorings}, new non-invertible morphisms that reduce the number of colored cuts.  We then prove the crucial result that the category of colored markings $\Mhat(\Sigma)$ (Theorem~\ref{thmmhatcontractible}) is contractible.
The reason for the significance of the categories of colored markings is Theorem~\ref{thmfunctormhat} 
which states that 
marked blocks can  be naturally extended to functors $\Mhat(\Sigma)\to\Ch$ out of the  category $\Mhat(\Sigma)$ of colored markings on $\Sigma$.
The idea is to send a colored marking to a version of marked blocks which uses homotopy coends for gluing at all colored cuts and ordinary coends at uncolored cuts.
This is motivated by the key observation that the   marked blocks do not change up to equivalence if we replace the \emph{homotopy coends} used for the gluing by \emph{ordinary coends} at all but one cut per closed connected component (Corollary~\ref{coruncoloringmaps}). As a consequence, the functor $\Mhat(\Sigma)\to\Ch$ sends all  uncolorings  to equivalences and hence descends to the $\infty$-groupoid obtained by localizing $\Mhat(\Sigma)$ at all uncolorings.  This construction allows us to reduce some statements about our differential graded marked blocks to statements about 
 marked blocks
 with values in vector spaces.

The functors $\Mhat(\Sigma)\to\Ch$ descend to the category obtained  by gluing colored markings for \emph{different} surfaces together (the gluing is accomplished via the Grothendieck construction).
By a homotopy left Kan extension, we obtain a symmetric monoidal functor defined on labeled surfaces --- this will be our  modular functor. The proof of the equivalence $\block_\cat{C}^{\Sigma,\Gamma}(\underline{X})\ra{\simeq} \mfc(\Sigma,\underline{X})$ from \eqref{eqnequivblock} relies on the contractibility result Theorem~\ref{thmmhatcontractible}. Finally, we use \eqref{eqnequivblock} to conclude excision from a marked version of excision (Proposition~\ref{propexcision}) which is easier to prove. \myskip

It should be mentioned that 
the methods developed in this article could also be helpful to study   modular functors \emph{with values in vector spaces}.
We consistently use the Lego Teichmüller game from \cite{bakifm} and construct the modular functor by gluing together (in a categorical sense) markings for different surfaces via the Grothendieck construction and a  left Kan extension (note that this  is different from the strategy in \cite{jfcs}, see Remark~\ref{remjfcs}). 
The important, but subtle concept of coends in categories of left exact functors between finite categories from \cite{lubalex} is avoided and replaced by techniques which are easier to adapt to a differential graded framework.
One key simplification in comparison to \cite{lubacmp,svea,dva,svea2} (regardless of whether these works cover the vector space valued case or work at chain level or in (co)homology) is that our construction does \emph{not} rely on a concrete presentation of mapping class groups in terms of generators and relations, but is genuinely topological.

\subparagraph{Conventions.}\label{pageconventions} Throughout this text, we will work over an algebraically closed field $k$ which is \emph{not} assumed to have characteristic zero.
 By $\Ch$ we denote the symmetric monoidal category of chain complexes over $k$ equipped with its projective model structure in which weak equivalences (for short: equivalences) are quasi-isomorphisms and fibrations are degree-wise surjections. A (small) category enriched over $\Vect_k$ or $\Ch$ will be called a linear or differential graded category, respectively. Unless otherwise stated, functors between linear and differential graded categories will automatically be assumed to be enriched. 
By a \emph{(canonical) equivalence} between chain complexes we do not necessarily mean  a map in either direction, but also allow a (canonical) zigzag. We will denote equivalences by $\simeq$ while the notation $\cong$ is reserved for isomorphisms (in any category).

 \subparagraph{Acknowledgments.} 
We would like to thank 
	Adrien Brochier, 
	Damien Calaque,
	Jürgen Fuchs,
	David Jordan,
	André Henriques,
	Simon Lentner,
	Svea Nora Mierach,
	Lukas Müller,
	Claudia Scheimbauer,
	 Yorck Sommerhäuser
	 and Nathalie Wahl
	for helpful discussions.

CS and LW are supported by the Deutsche Forschungsgemeinschaft (DFG, German Research
Foundation) within the framework of the RTG 1670
``Mathematics Inspired by String Theory and QFT'' and
under Germany’s Excellence Strategy -- EXC 2121 ``Quantum Universe'' -- 390833306.

\spaceplease
\section{Marked blocks\label{secderivedblocks}}
We start by
giving the definition of marked blocks which should be seen as  auxiliary objects needed for the construction of modular functors. They will later enable us to perform  concrete computations.
Moreover, we will already establish  a version of excision for marked surfaces that may be seen as a preparation  for the excision property that will be a part of the  modular functor.

\subsection{Conventions on surfaces, cut systems and markings\label{secsurfaceterminology}}
Before defining marked blocks, we recall some terminology and conventions on surfaces, cut systems and markings from 
\cite{bakifm,baki,jfcs}:
In the sequel, a \emph{surface} will be an abbreviation for compact oriented two-dimensional smooth manifold with boundary. A \emph{surface with oriented boundary} is a surface such that every boundary component is endowed with an orientation. If this orientation of a specific boundary component coincides with the orientation inherited from the surface, we refer to this component as \emph{outgoing}, otherwise as \emph{incoming}.
An \emph{extended surface} is a surface with oriented boundary and the choice of a marked point on every boundary component.

\subparagraph{Cut systems.}
Every surface can be non-uniquely cut into spheres with several open disks removed.
This is formalized as follows: We define a \emph{cut} on an extended surface $\Sigma$ as an oriented simple closed curve in the interior of $\Sigma$ with the choice of a point on this curve. An isotopy class of a finite family $C$ of disjoint cuts on $\Sigma$ is called a \emph{cut system} if every component of $\Sigma \setminus C$ has genus zero (here the isotopy has to preserve the disjointness of the cuts).
The number of cuts in a cut system $C$ will be denoted by $|C|$.  The manifold with boundary obtained by cutting $\Sigma$ along $C$ will be denoted by $\cut_C \Sigma$ (we must keep in mind that, strictly speaking, $\cut_C \Sigma$ is only well-defined up to diffeomorphism because cut systems are defined as isotopy classes).
If $C$ consists of $|C|$ cuts, then $\cut_C \Sigma$ has $2|C|$ boundary components more than $\Sigma$. The orientation and the marked point on these additional boundary components are inherited from the cutting curve. The relations between different cut systems will be covered in Section~\ref{secfinemarkings}.

\subparagraph{Standard spheres.}
For $n\ge 1$ and a family $\underline{\varepsilon}=(\varepsilon_1,\dots,\varepsilon_n) \in \{\pm 1\}^n$ of signs, we define a particular extended surface, namely the \emph{standard sphere} $\mathbb{S}_{n,\underline{\varepsilon}}^\circ$. The underlying surface is the Riemann sphere $\mathbb{C}\cup \{\infty\}$ with $n$ open disks $D_1,\dots,D_n$ with radius $1/3$ and centers $1,\dots,n$ removed. The orientation of this surface is the standard one, and the boundary circle with center $j$, where $1\le j\le n$, is endowed with the inherited orientation if $\varepsilon_j=+1$ (making this component outgoing), and endowed with the opposite orientation if $\varepsilon_j=-1$ (making this component \emph{incoming}). The marked points lie at $j-\im /3$ for $1\le j\le n$. This extended surface is decorated with a graph $\Gamma_n^\circ$ called the \emph{standard marking} whose vertices are the marked points and the so-called \emph{internal vertex} at $-2 \im$. The edges are the $n$ straight lines between the internal vertex and each marked point. The point $1-\im /3$ is called the \emph{distinguished vertex}, and the edge connecting the internal vertex to the distinguished vertex is called the \emph{distinguished edge}. The standard marking for a sphere with three holes and an example of a boundary orientation is depicted in Figure~\ref{figstdmarking}.

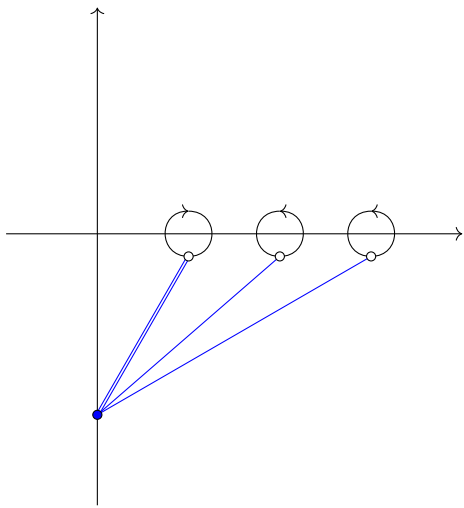
\begin{figure}[h]
	\centering
	\begin{tikzpicture}[scale=0.6, style/.style={circle, inner sep=0pt,minimum size=0mm},
	poin/.style={circle, inner sep=0.2pt,minimum size=0.2mm},decoration={
		markings,
		mark=at position 0.7 with {\arrow{>}}}]
\node [style=none] (14) at (-22, 5) {};
\node [style=none] (15) at (-22, -6) {};
\node [style=none] (16) at (-24, 0) {};
\node [style=none] (17) at (-14, 0) {};
\node [style=none] (18) at (-20, 0.5) {};
\node [style=none] (19) at (-20, -0.5) {};
\node [style=none] (21) at (-18, 0.5) {};
\node [style=none] (22) at (-18, -0.5) {};
\node [style=none] (23) at (-16, 0.5) {};
\node [style=none] (24) at (-16, -0.5) {};
\node [style=none] (25) at (-22, -4) {};
\draw [->] (15.center) to (14.center);
\draw [->] (16.center) to (17.center);
\draw [bend left=90, looseness=1.75] (18.center) to (19.center);
\draw [->, bend left=90, looseness=1.75] (19.center) to (18.center);
\draw [bend left=90, looseness=1.75] (22.center) to (21.center);
\draw [->, bend right=90, looseness=1.75] (22.center) to (21.center);
\draw [bend left=90, looseness=1.75] (24.center) to (23.center);
\draw [->, bend right=90, looseness=1.75] (24.center) to (23.center);
\draw [blue,double] (25.center) to (19.center);
\draw [blue] (25.center) to (22.center);
\draw [blue] (25.center) to (24.center);
 \draw[fill=white] (19) circle (1mm);
  \draw[fill=white] (22) circle (1mm);
   \draw[fill=white] (24) circle (1mm);
   \draw[fill=blue] (25) circle (1mm);
	\end{tikzpicture}
	\caption{Standard marking on a sphere with three holes and sign tuple $\underline{\varepsilon}=(-1,+1,+1)$. The long straight arrows are the coordinate axes of the complex plane.
		On the three boundary components, the marked points are drawn as white dots and the orientation is indicated by an arrow. The distinguished vertex of the marking is a blue dot, and the edges of the marking are drawn in blue. The distinguished edge is drawn as a double line.}
	\label{figstdmarking}
\end{figure}

\subparagraph{Markings.}
Let $\Sigma$ be a connected extended surface of genus zero. 
A \emph{chart} for $\Sigma$ is a diffeomorphism $\Phi : \Sigma \to  \mathbb{S}_{n,\underline{\varepsilon}}^\circ$ for some $n\ge 1$ and $\underline{\varepsilon} \in \{\pm 1\}^n$ which preserves the orientation, the orientation of the boundary and sends marked points to marked points.
We call a graph $\Gamma$ embedded in $\Sigma$ a \emph{$\Phi$-compatible graph} for a chart $\Phi : \Sigma \to  \mathbb{S}_{n,\underline{\varepsilon}}^\circ$ if $\Phi$ sends $\Gamma$ to the standard graph $\Gamma_n^\circ$ on $\mathbb{S}_{n,\underline{\varepsilon}}^\circ$.
Consider two pairs $(\Phi_j,\Gamma_j)$ for $j=0,1$, where $\Phi_j : \Sigma \to  \mathbb{S}_{n,\underline{\varepsilon}}^\circ$ is a chart (both for the same $n$ and $\underline{\varepsilon}$) and $\Gamma_j$ a $\Phi_j$-compatible graph on $\Sigma$.
An isotopy $(\Phi_0,\Gamma_0) \to (\Phi_1,\Gamma_1)$ is an isotopy $\Phi_t$ from $\Phi_0$ to $\Phi_1$ through charts.
 An isotopy class of a pair formed by a chart for $\Sigma$ and a compatible graph is called a
 \emph{marking without cuts} on $\Sigma$.
A \emph{marking} on an arbitrary extended surface $\Sigma$ is a cut system $C$ together
with a marking without cuts on every connected component of $\cut_C \Sigma$.
Note that this equips $\Sigma$ in particular with an isotopy class of graphs $\Gamma$, see Figure~\ref{figmarking} for an example.
Often, we will use the symbol $\Gamma$ for this graph to denote the entire marking, thereby suppressing the underlying cut system and the charts in the notation.

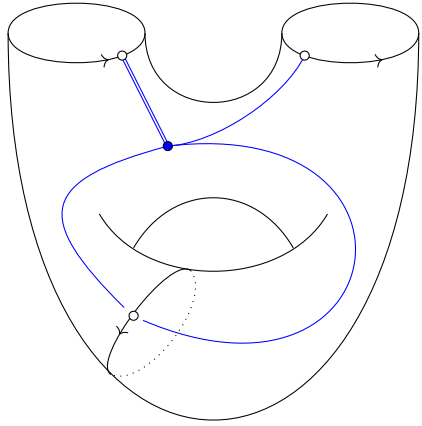
\begin{figure}[h]
	\centering
 \begin{tikzpicture}[scale=0.6, style/.style={circle, inner sep=0pt,minimum size=0mm},
 poin/.style={circle, inner sep=0.2pt,minimum size=0.2mm},decoration={
 	markings,
 	mark=at position 0.7 with {\arrow{>}}}]
 \node [style=none] (0) at (-15, 7) {};
 \node [style=none] (1) at (-12, 7) {};
 \node [style=none] (2) at (-9, 7) {};
 \node [style=none] (3) at (-6, 7) {};
 \node [style=none] (4) at (-12.25, 2.25) {};
 \node [style=none] (5) at (-8.75, 2.25) {};
 \node [style=none] (6) at (-13, 3) {};
 \node [style=none] (7) at (-8, 3) {};
 \node [style=none] (8) at (-12.5, 6.5) {};
 \node [style=none] (9) at (-8.5, 6.5) {};
 \node [style=none] (10) at (-11.5, 4.5) {};
 \node [style=none] (11) at (-12.75, -0.55) {};
 \node [style=none] (12) at (-11, 1.75) {};
 \node  (13)  at (-12.25, 0.75)  {};
 \draw [in=90, out=90, looseness=0.75] (0.center) to (1.center);
 \draw [in=90, out=90, looseness=0.75] (2.center) to (3.center);
 \draw [in=-90, out=-90, looseness=1.75] (1.center) to (2.center);
 \draw [in=-90, out=-90, looseness=3.25] (0.center) to (3.center);
 \draw [bend left=60, looseness=1.25] (4.center) to (5.center);
 \draw [bend right=60] (6.center) to (7.center);
 \draw [postaction={decorate}, bend right=90, looseness=0.75] (0.center) to (1.center);
 \draw [postaction={decorate}, bend right=90, looseness=0.75] (2.center) to (3.center);
 \draw [dotted, bend left=75, looseness=0.75] (12.center) to (11.center);
 \draw [postaction={decorate},bend right=90, looseness=0.50] (12.center) to (11.center);
 \draw [bend left=60, looseness=2.00,blue] (13) to (10);
 \draw  [double,blue] (10) to (8);
 \draw [bend right=30, looseness=0.75,blue] (10) to (9);
 \draw [bend left=105, looseness=4.00,blue] (10) to (13);
 \draw[fill=white] (8) circle (1mm);
 \draw[fill=white] (9) circle (1mm);
 \draw[fill=white] (13) circle (1mm);
 \draw[fill=blue] (10) circle (1mm);
 \end{tikzpicture}
 	\vspace*{-1cm}
 	\caption{A marking on a torus $\mathbb{T}^2$ with two boundary components. The cut system has one cut such that $\cut_C \mathbb{T}^2$ is a sphere with four holes.
 		We use the same drawing convention for the marked points, the orientation and the marking as in Figure~\ref{figstdmarking}.}
 	\label{figmarking}
 	\end{figure}

\subsection{Definition of marked blocks\label{derivedconformalblockssec}}
After these preparations, we may  define the marked blocks for a pivotal $k$-linear monoidal category $\cat{C}$.
The monoidal product and the monoidal unit for the underlying $k$-linear monoidal category will be denoted by $\otimes$ and $I$, respectively. 
Pivotality means that $\cat{C}$ is rigid 
(we denote by $-^\vee$ the duality functor; our conventions for the duality are the ones from \cite{egno})
and equipped
with a monoidal natural isomorphism $\id_\cat{C} \cong -^{\vee \vee}$ from the identity functor on $\cat{C}$ to the double dual functor.
Thanks to the pivotal structure, left and right duals coincide.
 We introduce the notation $X^\varepsilon$ for $\varepsilon \in \{\pm 1\}$ by $X^1:=X$ 
and $X^{-1}:=X^\vee$. 

In order to establish later the crucial results for these marked blocks, we will need $\cat{C}$ to be modular (a notion recalled in Section~\ref{secdmfmain}), but the mere 
definition of marked blocks makes sense in greater generality.

Denote by $\Sigma_0^{p|n-p}$ a surface of genus zero with $n\ge 1$ holes, $p$ of which are incoming. Let $\Sigma_0^{p|n-p}$ be endowed with a marking without cuts, i.e.\ an isotopy class of a chart $\Phi : \Sigma_0^{p|n-p} \to \mathbb{S}_{n,\underline{\varepsilon}}^\circ$ and a graph $\Gamma$ that is mapped by $\Phi$
 to the standard marking $\Gamma_n^\circ$ on the standard sphere $\mathbb{S}_{n,\underline{\varepsilon}}^\circ$. 
 Now let $\underline{X}$ be a labeling of the boundary components of $\Sigma_0^{p|n-p}$ with projective objects in $\cat{C}$, i.e.\ a function from $\pi_0(\partial \Sigma_0^{p|n-p})$ to the projective objects of $\cat{C}$. 
The map $\pi_0(\partial \mathbb{S}_{n,\underline{\varepsilon}}^\circ ) \to \pi_0(\partial \Sigma_0^{p|n-p})$ induced by $\Phi$ provides a numbering $(X_1,\dots,X_n)$ of the objects that are part of the labeling.
We define the vector space
\begin{align}
\block_\cat{C}^{\Sigma_0^{p|n-p},\Gamma} (\underline{X})             :=  \cat{C}\left(I,  X_{1}^{\varepsilon_1}   \otimes\dots\otimes  X_{n}^{\varepsilon_n } \right)    \ , \label{eqndmfoncyl}
\end{align}
where $\cat{C}(-,-)$ is our notation for the morphism spaces of $\cat{C}$,
and observe that this is well-defined, i.e.\ it does not depend on the representative of our marking (recall that in our language a marking is always an isotopy class of certain data as precisely defined above).

In the next step, let $(\Sigma,\Gamma)$ be a connected  marked surface. The surface $\cut_C \Sigma$ obtained by cutting $\Sigma$ along the cut system underlying the     marking yields components $(\Sigma_1,\Gamma_1),\dots,(\Sigma_\ell,  \Gamma_\ell  )$ with each $\Sigma_j$ being a sphere with $m_j\ge 1$ holes. We are choosing here a numbering for the surfaces $\Sigma_j$ and we will do the same for the cuts in a moment; this is done mainly for the readability of this rather technical definition --- we will explain afterwards how the dependence on this order is dealt with. 
If $n_j$ is the number of boundary components of $\Sigma_j$ that do not arise from the cutting, then we can label these boundary components by a family $\underline{X}_j$ of length $n_j$ of projective objects in $\cat{C}$. These boundary labels form a boundary label $\underline{X}$ for $\Sigma$ (and of course, every boundary label for $\Sigma$ arises this way). Additionally, we label the two oppositely oriented boundaries arising from a cut $c_i$ by a projective object $Y_i$, where $1\le i\le |C|$.  Denote by $\underline{Y}_j$ the boundary labels of $\Sigma_j$ arising from cuts. 
Now it makes sense to consider the vector spaces
\begin{align}
\block_\cat{C}^{\Sigma_j,\Gamma_j} (\underline{X}_j,     \underline{Y}_j     )  \quad \text{for} \quad 1\le j\le \ell \ . \label{eqndefviaderivedcoend}
\end{align}
As these vector spaces run over $j$, each $Y_i$ appears precisely twice, for the two boundary components resulting from the cut --- once as covariant and once as contravariant argument. 
Therefore, we may define the chain complex
\begin{align}
\block_\cat{C}^{\Sigma,\Gamma} (\underline{X}):=\lint^{Y_1,\dots,Y_{|C|} \in \Proj \cat{C}} \bigotimes_{j=1}^\ell \block_\cat{C}^{\Sigma_j,\Gamma_j} (\underline{X}_j,     \underline{Y}_j     )      \label{eqnderivedconfblock}
\end{align}
via an iterated homotopy coend \cite[Section~2]{dva}.
Recall that for a functor $G:(\Proj \cat{C})^\op \otimes \Proj \cat{C}\to \Ch$,
the homotopy coend $\lint^{X\in\Proj\cat{C}} G(X,X)$ is given by the realization of the simplicial chain complex
\begin{center}
	{\footnotesize
		\begin{equation}
			\begin{tikzcd}
				\dots \ar[r, shift left=6]  \ar[r, shift left=2]
				\ar[r, shift right=6]  \ar[r, shift right=2]
				& 	\displaystyle \bigoplus_{  \substack{X_0,X_1 ,X_2 \\ \in \Proj\cat{C}}}  \cat{C}(X_1,X_0)\otimes\cat{C}(X_2,X_1)\otimes  G(X_0,X_2)
				\ar[l, shift left=4]  \ar[l]
				\ar[l, shift right=4]  
				\ar[r, shift left=4] \ar[r, shift right=4] \ar[r] & \displaystyle \bigoplus_{X_0,X_1 \in \Proj\cat{C}}  \cat{C}(X_1,X_0)\otimes G(X_0,X_1) \ar[r, shift left=2] \ar[r, shift right=2]
				\ar[l, shift left=2] \ar[l, shift right=2]
				& \displaystyle \bigoplus_{X_0 \in \Proj \cat{C}} G(X_0,X_0)\ ,  \ar[l] \\
			\end{tikzcd}
	\end{equation}}
\end{center}
\normalsize
where the face maps use the composition in $\cat{C}$ and the degeneracies insert identities.

In \eqref{eqnderivedconfblock} we have used a numbering for the iterated homotopy coends and the multiple tensor products, which, strictly speaking, is problematic because those numberings are not part of the data of a marking (for instance, the set of cuts was not assumed to be ordered). 
Let us discuss the remedy in detail for the iterated homotopy coends (we comment on the multiple tensor products afterwards):
Denote by $\cat{O}_C$ the action groupoid of the free and transitive action of the permutation group on $|C|$ letters on the bijections from the set of cuts of $C$ to $\{1,\dots,|C|\}$. 
Then the right hand side of \eqref{eqnderivedconfblock}, when evaluated for all possible orderings of cuts, will actually give us a functor $\cat{O}_C \to \Ch$. To this end, a permutation of such an ordering must be sent to an isomorphism compatibly with the composition of permutations. The latter can be done using the Fubini Theorem from \cite[Proposition~2.7]{dva}. The marked block $\block_\cat{C}^{\Sigma,\Gamma} (\underline{X})$ can then be defined as the colimit  of that functor (which is also the homotopy colimit). 
This  gives us the notion of an \emph{unordered} iterated homotopy coend. Of course, if we pick an ordering for the set of cuts, the complex computed for that ordering will be canonically isomorphic to colimit over all orderings.  
In the same way, we can use an unordered tensor product (defined using the symmetric braiding of $\Ch$) to get rid of the numbering of the surfaces that $\Sigma$ is cut into.
In the sequel, however, we will suppress such subtleties in the notation and will understand an expression like   the right hand side of \eqref{eqnderivedconfblock} always as an \emph{unordered} homotopy coend and/or tensor product.

As for \eqref{eqndmfoncyl}, we can observe that \eqref{eqnderivedconfblock} does not depend on any representatives chosen for the marking (in particular, it is not a problem that $\cut_C \Sigma$ is only well-defined up to diffeomorphism). 
The reason for this is that the definitions \eqref{eqndmfoncyl} and \eqref{eqnderivedconfblock} just depend on the combinatorics coming from the incidences of cuts and graphs, i.e.\ their relative location to each other. This important point is already implicit in related constructions in  \cite{baki,jfcs}.\label{pageincidences}

The definition \eqref{eqnderivedconfblock} is extended to non-connected surfaces by sending the disjoint union of connected  marked surfaces  with boundary components labeled by projective objects to the tensor product of the chain complexes defined for the connected case.

\begin{definition}\label{derivedconformalblock}
	We refer to $\block_\cat{C}^{\Sigma,\Gamma} (\underline{X})$ as the \emph{marked block} for the  marked surface $(\Sigma,\Gamma)$, the pivotal linear monoidal category $\cat{C}$ and the projective boundary label $\underline{X}$.
	\end{definition}

Of course,  the marked blocks are functorial in the boundary label. This will play a role later, see Section~\ref{secdefdmf} and in particular Remark~\ref{cylcatfunctorialdeprem}.

The following example shows 
how the marked blocks are designed to produce
algebraic quantities appearing in the homological algebra of $\cat{C}$:

\begin{example}[Relation to Hochschild chains]\label{exhochschild}
	For the closed torus $\mathbb{T}^2$ with the  marking $\Gamma$ with one cut shown in Figure~\ref{figmarkingtorus},         
	 the marked block $\block_\cat{C}^{\mathbb{T}^2,\Gamma}$ is  given by the 
	 homotopy coend
	 \begin{align} \block_\cat{C}^{\mathbb{T}^2,\Gamma} =\lint^{X \in \Proj \cat{C}}  \cat{C}(X,X) \ , \end{align}
	 which by the definition is given by the
	 normalized chains on the simplicial vector space
	\footnotesize
	\begin{equation}
	\begin{tikzcd}
	\dots \ar[r, shift left=6]  \ar[r, shift left=2]
	\ar[r, shift right=6]  \ar[r, shift right=2]
	& 	\displaystyle \bigoplus_{  \substack{X_0,X_1 ,X_2 \\ \in \Proj\cat{C}}}  \cat{C}(X_1,X_0)\otimes\cat{C}(X_2,X_1)\otimes  \cat{C}(X_0,X_2)
	\ar[l, shift left=4]  \ar[l]
	\ar[l, shift right=4]  
	\ar[r, shift left=4] \ar[r, shift right=4] \ar[r] & \displaystyle \bigoplus_{X_0,X_1 \in \Proj\cat{C}}  \cat{C}(X_1,X_0)\otimes \cat{C}(X_0,X_1) \ar[r, shift left=2] \ar[r, shift right=2]
	\ar[l, shift left=2] \ar[l, shift right=2]
	& \displaystyle \bigoplus_{X_0 \in \Proj \cat{C}} \cat{C}(X_0,X_0)\ ,  \ar[l] \\
	\end{tikzcd}
	\end{equation}
	\normalsize
	where the face maps are defined using the composition in $\cat{C}$ and the degeneracies insert identities.
	This is the Hochschild complex of the linear category $\Proj \cat{C}$.
	It is called the Hochschild complex of $\cat{C}$ (instead of $\Proj \cat{C}$) in \cite{dva} because it is implicitly assumed that only the projective objects are used to construct the complex. The latter is motivated by the  \emph{Agreement Principle}  \cite{mcarthy,keller} that says that the Hochschild complex built from the projective objects of a linear category $\cat{C}$ is equivalent  to the usual Hochschild complex of a finite-dimensional algebra $A$ if $\cat{C}$ is the category of finite-dimensional modules over $A$ (see also Section~\ref{secbimodules}).
	\begin{figure}[h]
		\centering
		\begin{tikzpicture}[scale=0.6, style/.style={circle, inner sep=0pt,minimum size=0mm},
		poin/.style={circle, inner sep=0.2pt,minimum size=0.2mm},decoration={
			markings,
			mark=at position 0.7 with {\arrow{>}}}]
\node [style=none] (4) at (-12.25, 2.25) {};
\node [style=none] (5) at (-8.75, 2.25) {};
\node [style=none] (6) at (-13, 3) {};
\node [style=none] (7) at (-8, 3) {};
\node [style=none] (12) at (-10.5, 1.75) {};
\node [style=none] (13) at (-15, 2.5) {};
\node [style=none] (14) at (-6, 2.5) {};
\node [style=none] (15) at (-10.5, -0.14) {};
\node [style=none] (21) at (-11.05, 0.5) {};
\node [style=none] (22) at (-10.5, 4.5) {};
\draw [bend left=60, looseness=1.25] (4.center) to (5.center);
\draw [bend right=60] (6.center) to (7.center);
\draw [bend left=90] (13.center) to (14.center);
\draw [bend right=90] (13.center) to (14.center);
\draw [bend right=90,dotted] (15.center) to (12.center);
\draw [bend left=90,postaction={decorate}] (15.center) to (12.center);
\draw [in=0, out=0, looseness=3.00,blue,double] (22.center) to (21.center);
\draw [in=-180, out=-180, looseness=2.50,blue] (22.center) to (21.center);
 \draw[fill=white] (21) circle (1mm);
\draw[fill=blue] (22) circle (1mm);
		\end{tikzpicture}
		\vspace*{-0.5cm}
		\caption{Marking for the closed torus (same drawing conventions as in Figure~\ref{figstdmarking}).}
		\label{figmarkingtorus}
	\end{figure}
	\end{example}

\subsection{Excision properties of marked blocks}
The locality properties of modular functors are usually phrased in terms of \emph{factorization} and \emph{self-sewing}, see \cite{baki,jfcs}.
For the marked blocks, we will already formulate here a marked version of such a locality property that will follow in a straightforward way from the definitions.
 Factorization and self-sewing can even be packaged into one property that we call \emph{excision}.

In order to formalize the excision property, let $(\Sigma,\Gamma)$ be a  marked surface
and denote by $\Sigma'$ the result of sewing a specific incoming boundary component to a specific outgoing boundary component (more complicated sewing operations can be considered as well, but they can the decomposed into sewings of the form just described).
We write this sewing operation symbolically as an arrow $s : \Sigma \to \Sigma'$ (in Section~\ref{secsurfacecat} we will discuss a category of surfaces in which these sewings are promoted to actual morphisms, but this is not needed for the moment). The  marking $\Gamma$ on $\Sigma$ induces a  marking $\Gamma'$ on $ \Sigma'$.

Consider the marked block $\block^{\Sigma,\Gamma}_\cat{C}(\underline{X},P,Q)$ for $(\Sigma,\Gamma)$ and boundary label $(\underline{X},P,Q)$, where $P$ and $Q$ are the labels for the incoming and outgoing boundary component along which we glue, and $\underline{X}$ is a label for the remaining boundary components. 
Since  $\cat{C}(-,-):\cat{C}^\op\otimes\cat{C}\to\Vect$ is a functor,  it follows from the definition of marked blocks in \eqref{eqnderivedconfblock}
that the assignment $P\otimes Q \mapsto \block^{\Sigma,\Gamma}_\cat{C}(\underline{X},P,Q)$ yields a functor $(\Proj \cat{C})^\op \otimes \Proj \cat{C}\to\Ch$, i.e.\ a differential graded bimodule over $\Proj \cat{C}$. 
We now obtain a canonical isomorphism \begin{align} \block^{\Sigma',\Gamma'}_\cat{C} (\underline{X})\cong \lint^{P\in\Proj \cat{C}}\block^{\Sigma,\Gamma}_\cat{C}(\underline{X},P,P)  \label{eqnfubiniiso} \end{align} 
because both complexes describe the colimit by which the unordered homotopy coend is defined.
In particular, we obtain a \emph{sewing map}
\begin{align}
s^P : \block^{\Sigma,\Gamma}_\cat{C}(\underline{X},P,P) \to \lint^{P\in\Proj \cat{C}}\block^{\Sigma,\Gamma}_\cat{C}(\underline{X},P,P) \cong \block^{ \Sigma', \Gamma'}_\cat{C} (\underline{X}) \ ,            \label{eqnsewingmaps}
\end{align}
which is just the structure map $\block^{\Sigma,\Gamma}_\cat{C}(\underline{X},P,P) \to \lint^{P\in\Proj \cat{C}}\block^{\Sigma,\Gamma}_\cat{C}(\underline{X},P,P)$
composed with the isomorphism \eqref{eqnfubiniiso}.
By definition of the homotopy coend we now obtain the following statement:

\begin{proposition}[Excision with marking]\label{propexcision}
For every pivotal linear monoidal category $\cat{C}$, any  marked surface $(\Sigma,\Gamma)$ and a sewing $s : (\Sigma,\Gamma) \to (\Sigma',\Gamma')$ that glues one incoming boundary component to an outgoing one, the sewing maps \eqref{eqnsewingmaps} induce an equivalence
\begin{align*}
\lint^{P\in\Proj \cat{C}}\block^{\Sigma,\Gamma}_\cat{C}(\underline{X},P,P) \xrightarrow{\ \simeq \ }  \block^{   \Sigma',\Gamma'}_\cat{C} (\underline{X}) 
\end{align*}
of chain complexes
from the homotopy coend over the $\Proj \cat{C}$-bimodule $\block^{\Sigma,\Gamma}_\cat{C}(\underline{X},-,-)$
to $\block^{   \Sigma',\Gamma'}_\cat{C} (\underline{X}) $.
\end{proposition}

	In fact, this map is actually an isomorphism with the concrete models for the homotopy coend that we use.

Of course, the proof of the marked version of excision is immediate. The more non-trivial task of formulating such statements independently of the marking will be addressed later.

\subsection{Homotopy coends over projectives and Lyubashenko's coend\label{secaugmentationfibration}}
In \cite[Section~3.2]{dva}, we have proven a relation between homotopy coends over projective objects and the Lyubashenko coend whose definition we recall in a moment 
(we have also argued there that this is an instance of the general principle to express traces via class functions, but this background is not  needed to understand the statement). After recalling some terminology, we will prove a generalization of the statements made in \cite[Section~3.2]{dva} that we will need in the sequel.

\subparagraph{Finite tensor categories.}
As mentioned in the introduction, a \emph{finite category} is a linear Abelian category 
with finite-dimensional morphism spaces, enough projectives, finitely many isomorphism classes of simple objects such that every object has finite length.
A linear category is finite if and only if it is linearly equivalent to
  the category of finite-dimensional modules over a finite-dimensional algebra (which does not mean that choosing such an equivalence 
 will be necessarily helpful). 
A \emph{tensor category} is a
linear Abelian rigid monoidal category with simple unit. 
A \emph{finite tensor category} \cite{etinghofostrik} is a tensor category whose underlying linear category is a finite category.
From these definitions, one concludes the following:
In a finite tensor category $\cat{C}$,
 any tensor product $X\otimes Y$ is projective if $X$ or $Y$ is projective. Moreover, the tensor product is exact in both arguments. Finally, any finite tensor category is self-injective, i.e.\ the projective objects are precisely the injective ones.

\subparagraph{The Lyubashenko coend.}
For any finite tensor category $\cat{C}$, one may define the coend
\begin{align}
\mathbb{F}:=\int^{X \in \cat{C}} X \otimes X^\vee \label{lyubacoend}
\end{align}
which is called the \emph{canonical coend of $\cat{C}$} or also the \emph{Lyubashenko coend} due to its appearance in \cite{lubacmp,luba,lubalex}. 
This object is the key to the construction of the mapping class group actions in \cite{lubacmp}. 

In \cite[Section~3.2]{dva}, the coend \eqref{lyubacoend} is replaced by a (finite version of a) homotopy coend leading to a differential graded object $\flint^{X \in \Proj \cat{C}} X\otimes X^\vee$ 
in $\cat{C}$ defined as follows:
The objects
\begin{align}
	\bigoplus_{X_0,\dots,X_n \in \Proj\cat{C}}  \cat{C}(X_n,X_{n-1})\otimes\dots\otimes\cat{C}(X_1,X_0) \otimes X_n \otimes X_0^\vee
\end{align}
living in a completion of $\cat{C}$ under infinite direct sums can be used to assemble a simplicial object
\begin{center}
	{	\begin{equation}\label{eqnresF}
			\begin{tikzcd}  
				\ar[r, shift left=4] \ar[r, shift right=4] \ar[r]\dots  & \displaystyle \bigoplus_{X_0,X_1 \in \Proj\cat{C}}  \cat{C}(X_1,X_0)\otimes X_1 \otimes X_0^\vee \ar[r, shift left=2] \ar[r, shift right=2]
				\ar[l, shift left=2] \ar[l, shift right=2]
				& \displaystyle \bigoplus_{X_0 \in \Proj \cat{C}} X_0\otimes X_0^\vee\ .  \ar[l] \\
			\end{tikzcd}
	\end{equation}}
\end{center}
It is shown that the direct sums appearing here can be reduced to finite direct sums running only over a suitable finite collection of projective objects (e.g.\ a projective generator) without changing the simplicial object up to equivalence.
One defines $\flint^{X \in \Proj \cat{C}} X\otimes X^\vee$
(the `f' stands for \emph{finite}) as the differential graded object in $\cat{C}$ obtained by the realization of any such finite reduction of~\eqref{eqnresF} (this does not depend on the concrete reduction up to homotopy equivalence).  
It is also shown that $\flint^{X \in \Proj \cat{C}} X\otimes X^\vee$ is a projective resolution of $\mathbb{F}$. This projective resolution appears in the following important homological algebra result that is the key to understanding marked blocks because it gives us the possibility to express iterated homotopy coends over morphism spaces in a different way:

\begin{proposition}\label{propoitcoend}
	For any pivotal finite tensor category $\cat{C}$, there is a canonical equivalence of chain complexes
	\begin{align}
	\lint^{P_1,\dots,P_g \in \Proj \cat{C}} \cat{C}(X,P_1\otimes P_1^\vee \otimes \dots \otimes P_g \otimes P_g^\vee) \simeq\ \cat{C}\left(X,   \left(\flint^{P\in \Proj \cat{C}}   P\otimes P^\vee\right)  ^{\otimes g}  \right)
	\end{align} for any $X\in\cat{C}$ and $g\ge 0$.
	\end{proposition}

Recall from our conventions that by \emph{(canonical) equivalence} we do not mean necessarily a map in either direction, but also allow a (canonical) zigzag.

\begin{proof}
	For $g=0$, the statement is true by the convention
	that a tensor product over an empty index set is the monoidal unit.
	We now prove the statement by induction on $g\ge 1$. 
	For $g=1$, we find by duality in $\cat{C}$ and \cite[Theorem~3.5]{dva}
	\begin{align}
	\lint^{P \in \Proj \cat{C}} \cat{C}(X,P\otimes P^\vee) &\cong 	\lint^{P \in \Proj \cat{C}} \cat{C}(X\otimes P,P)     \quad \text{(by duality)}  \\&\simeq \  \cat{C} \left(     I, \flint^{P\in \Proj \cat{C}} P \otimes (X\otimes P)^\vee  \right) \quad \text{(by \cite[Theorem~3.5]{dva})}\\&\cong\ \cat{C} \left(     I, \flint^{P\in \Proj \cat{C}} P \otimes P^\vee \otimes X^\vee \right) \quad \text{(since $(X\otimes P)^\vee \cong P^\vee \otimes X^\vee$)} \\&\cong\
	\cat{C} \left(     X, \flint^{P\in \Proj \cat{C}} P \otimes P^\vee  \right)    \quad \text{(by duality)} \ . 
	\end{align}
	This proves the statement for $g=1$. 
	
	In order to complete the induction step $g \to g+1$, we observe
	\begin{align}
	&	\lint^{P_1,\dots,P_{g+1} \in \Proj \cat{C}} \cat{C}(X,P_1\otimes P_1^\vee \otimes \dots \otimes P_{g+1} \otimes P_{g+1}^\vee) \\\cong\ &	\lint^{P_1,\dots,P_{g+1} \in \Proj \cat{C}} \cat{C}(X\otimes P_{g+1}\otimes P_{g+1}^\vee,P_1\otimes P_1^\vee \otimes \dots \otimes P_{g} \otimes P_{g}^\vee)   \quad \text{(by duality)} \\ \simeq &
	\lint^{P_{g+1} \in \Proj \cat{C}} \cat{C} \left(   X\otimes P_{g+1}\otimes P_{g+1}^\vee ,    \left(\flint^{P\in \Proj \cat{C}}   P\otimes P^\vee\right)  ^{\otimes g}    \right) \quad \text{(by the induction hypothesis)} \\
\simeq\ &	\lint^{P_{g+1} \in \Proj \cat{C}} \cat{C} \left(   X\otimes P_{g+1}\otimes P_{g+1}^\vee ,   \mathbb{F}  ^{\otimes g}    \right) \quad \substack{\displaystyle \text{(by \cite[Corollary~3.7]{dva}} \\ \displaystyle \text{ and since $X\otimes P_{g+1}\otimes P_{g+1}^\vee$ is projective)}}\\
\cong\ &   \lint^{P_{g+1} \in \Proj \cat{C}} \cat{C} \left(   \left(\mathbb{F}  ^{\otimes g}\right) ^\vee  \otimes X, P_{g+1}\otimes P_{g+1}^\vee        \right) \quad \text{(by duality)}     
\\
\simeq\ &    \cat{C} \left(   \left(\mathbb{F}  ^{\otimes g}\right) ^\vee  \otimes X, \flint^{P\in \Proj \cat{C}}   P\otimes P^\vee       \right) \quad \text{(by the induction start)}   
\\
\cong\ &    \cat{C} \left(   X, \mathbb{F}  ^{\otimes g} \otimes \flint^{P\in \Proj \cat{C}}   P\otimes P^\vee       \right) \quad \text{(by duality)} 
\\
\simeq\ &    \cat{C} \left(   X, \left(\flint^{P\in \Proj \cat{C}}   P\otimes P^\vee\right)  ^{\otimes (g+1)}      \right) \quad \text{(by \cite[Lemma~3.8]{dva}).}
	\end{align} 
	\end{proof}

\subparagraph{Unimodularity.}
Let $\cat{C}$ be a pivotal finite tensor category. Since $\cat{C}$ is assumed to be over an algebraically closed field, we have by 
\cite[Section~5.3]{mtrace} a so-called
\emph{modified trace} on the tensor ideal of projective objects that is unique up to invertible scalar.
 It gives us canonical 
 and natural isomorphisms $\cat{C}(X,P)\cong \cat{C}(\alpha \otimes P , X)^*$ for any $X\in\cat{C}$ as long as $P$ is projective, where $\alpha$ is the socle of the projective cover of the monoidal unit.  
 If $\cat{C}$ is unimodular, $\alpha$ is isomorphic to the monoidal unit. In this case, we even find isomorphisms \begin{align} \label{eqndualhom} \cat{C}(X,P)^*\cong \cat{C}(P,X)\ , \quad X\in\cat{C}\ , \quad P\in\Proj \cat{C} \ . \end{align} Therefore, we may conclude from Proposition~\ref{propoitcoend}:

\begin{corollary}\label{coritcoend}
	For any unimodular pivotal finite tensor category $\cat{C}$ and $X\in\cat{C}$, there is a canonical equivalence of chain complexes
	\begin{align}
	\lint^{P_1,\dots,P_g \in \Proj \cat{C}} \cat{C}(X,P_1\otimes P_1^\vee \otimes \dots \otimes P_g \otimes P_g^\vee) \simeq
	\cat{C} \left(  \left(  \flint^{P\in\Proj \cat{C}} P\otimes P^\vee  \right)^{\otimes g}  ,X  \right)^*  \ . 
	\end{align}
\end{corollary}

\begin{example}\label{exblock}
	For $p,q\ge 0$ and $g\ge 1$, consider the standard marked sphere $\mathbb{S}_{p+q+2g,\underline{\varepsilon}}^\circ$ such that the first $p$ holes are incoming, the next $q$ holes are outgoing, and the last $2g$ holes are pairs of an outgoing followed by an incoming hole.
	By gluing together these $g$ pairs, we obtain a marking $\Gamma_{g,p,q}^\circ$ on the surface $\Sigma_g^{p|q}$ of genus $g$ with $p$ incoming and $q$ outgoing boundary components. If $\cat{C}$ is a  pivotal finite tensor category and $\underline{X}=(X_1', \dots,X_p',X_1'',\dots,X_q'')$ a projective boundary label for $\Sigma_g^{p|q}$ (the first $p$ labels are for the incoming boundary components, the last $q$ labels for the outgoing ones), then by the definition of marked blocks and duality (recall that left and right duals coincide thanks to pivotality), we have an isomorphism
	\begin{align}
	\block_\cat{C}^{\Sigma_g^{p|q},\Gamma_{g,p,q}^\circ}(\underline{X})\cong \lint^{P_1,\dots,P_g \in \Proj \cat{C}} \cat{C}(X_p'\otimes\dots\otimes X_1', X_1''\otimes\dots\otimes X_q''\otimes  P_1\otimes P_1^\vee \otimes \dots \otimes P_g \otimes P_g^\vee) \ .
	\end{align}
	From Proposition~\ref{propoitcoend}, we conclude
		\begin{align}
	\block_\cat{C}^{\Sigma_g^{p|q},\Gamma_{g,p,q}^\circ}(\underline{X})\simeq 
	\cat{C}\left(X_p'\otimes\dots\otimes X_1',X_1''\otimes\dots\otimes X_q''\otimes   \left(\flint^{P\in \Proj \cat{C}}   P\otimes P^\vee\right)  ^{\otimes g}  \right) \ . \label{eqnexblock1} 
	\end{align} 
	If $\cat{C}$ is unimodular, 
	 we arrive at
	\begin{align}
	\block_\cat{C}^{\Sigma_g^{p|q},\Gamma_{g,p,q}^\circ}(\underline{X})\simeq \cat{C} \left( X_1''\otimes\dots\otimes X_q''\otimes  \left(   \flint^{P\in\Proj \cat{C}} P\otimes P^\vee  \right)^{\otimes g}  ,X_p'\otimes\dots\otimes X_1'  \right)^*  \label{eqnexblock2} 
	\end{align}
thanks to	Corollary~\ref{coritcoend}.
	\end{example}

\subsection{The augmentation fibration\label{secaugmentationfibration}}
We will now relate marked blocks to the `classical' marked blocks with values in vector spaces from \cite{jfcs} which are based on \cite{lubacmp}. First we observe that the 
 marked block $\block_\cat{C}^{\Sigma,\Gamma} (\underline{X})$
 for a marked surface $(\Sigma,\Gamma)$ with projective boundary label $\underline{X}$
comes with a canonical surjection
\begin{align}\block_\cat{C}^{\Sigma,\Gamma} (\underline{X})\to H_0 \block_\cat{C}^{\Sigma,\Gamma} (\underline{X}) \label{augmentationfibration}
\end{align} to its zeroth homology
that we refer to as the \emph{augmentation fibration}.
The zeroth homology of the marked block $\block_\cat{C}^{\Sigma,\Gamma} (\underline{X})$
is given  the same expression as in \eqref{eqnderivedconfblock}, but with the homotopy coend replaced by an ordinary coend. It is important that both types of coends run over the subcategory of projective objects.

\subparagraph{Marked blocks with values in vector spaces.}
In \cite[Section~2.4]{jfcs} marked blocks 
are defined based on \cite{lubacmp,lubalex}
as coends in categories of left-exact functors. For genus zero surfaces with boundary, the definition agrees with~\eqref{eqndmfoncyl}, but without the restriction that the boundary labels need to be projective. Instead of homotopy coends over categories of projective objects that are used for the gluing of marked blocks in~\eqref{eqnderivedconfblock}, left-exact coends are used.
This definition can be made for any pivotal finite tensor category $\cat{C}$ (for the construction of mapping class group actions, more structure is needed).
The value of the these block functors on a given boundary label is a vector space and will denoted by $\ordblock_\cat{C}^{\Sigma,\Gamma} (\underline{X})$; in \cite{jfcs} the notation $\widetilde{\text{Bl}}_{\Sigma,\Gamma}$ is used for the block functors.
We now arrive at the following comparison (for the moment, this is a comparison at the level of chain complexes and vector spaces, respectively --- no additional structure is included):

\begin{proposition}[Augmentation fibration in the case of non-empty boundary]\label{proptrivfibbdy}
	Let $\cat{C}$ be a pivotal finite tensor category.
	If $(\Sigma,\Gamma)$ is a  marked surface with at least one boundary  component on each connected component, then the augmentation fibration \eqref{augmentationfibration}
	is a trivial fibration for every projective boundary label $\underline{X}$, i.e.\ in this case, the marked block has non-trivial homology only in degree zero.
	Moreover, this zeroth homology is canonically isomorphic to the vector spaces valued marked block $\ordblock_\cat{C}^{\Sigma,\Gamma} (\underline{X})$ such that the augmentation fibration takes the form of a trivial fibration
	\begin{align}\block_\cat{C}^{\Sigma,\Gamma} (\underline{X})\ra{\simeq}\ordblock_\cat{C}^{\Sigma,\Gamma} (\underline{X}) \ . \label{eqnaugmentionfibrationeqn}
	\end{align}
\end{proposition}

\begin{proof} 
	Using the key homological algebra result from Proposition~\ref{propoitcoend} that relates 
	homotopy coends over projective objects to morphisms into a resolution of the Lyubashenko coend,
	the result follows from a careful comparison of definitions.
	
	Let us make this precise:
	Without loss of generality, we may assume that $\Sigma$ is a connected surface of genus $g$ with $n\ge 1$ boundary components that are all incoming.
Consider now the diagram (whose ingredients we will explain step by step):

\footnotesize 
	\begin{equation}\label{eqnaugmentation}
\begin{tikzcd}
\block_\cat{C}^{\Sigma,\Gamma}(\underline{X})  \ar[<->]{rr}{} & & \cat{C}\left(\underline{X}^\otimes,  \left(   \flint^{P\in\Proj \cat{C}} P\otimes P^\vee    \right)^{\otimes g}   \right) \ar[swap]{dd}{\substack{\text{induced by}\\  \text{ $ \flint^{P\in\Proj \cat{C}} P\otimes P^\vee\ra{\simeq}\mathbb{F}$  }}}  \\
& & \\
\ordblock_\cat{C}^{\Sigma,\Gamma}(\underline{X}) \ar{rr}{  \text{\cite[(2.12)]{jfcs}}  }   & & \cat{C}(\underline{X}^{\otimes}, \mathbb{F}^{\otimes g} )   
\end{tikzcd} \ra{H_0}	\begin{tikzcd}
H_0\block_\cat{C}^{\Sigma,\Gamma}(\underline{X}) \ar[swap,dashed]{dd}{} \ar[->]{rr}{\cong} & & H_0\cat{C}\left(\underline{X}^\otimes,  \left(   \flint^{P\in\Proj \cat{C}} P\otimes P^\vee    \right)^{\otimes g}   \right) \ar{dd}{\cong}  \\
& & \\
\ordblock_\cat{C}^{\Sigma,\Gamma}(\underline{X})    & & \cat{C}(\underline{X}^{\otimes}, \mathbb{F}^{\otimes g} ) \ar{ll}{ \cong  } 
\end{tikzcd} 
\end{equation} \normalsize
We first describe the left diagram:
The lower horizontal map is the canonical isomorphism \cite[(2.12)]{jfcs} from the vector spaces valued marked block to the morphism space $\cat{C}(\underline{X}^{\otimes}, \mathbb{F}^{\otimes g} )$, where $\underline{X}^{\otimes}=X_{\sigma(1)} \otimes \dots \otimes X_{\sigma(n)}$ for some permutation $\sigma$ on $n$ letters
(that is fixed by our definition of marked blocks, but not relevant for the argument because it is the same regardless of whether we consider blocks with values in vector spaces or chain complexes). 
This isomorphism is a specific sequence coming from the application of the Yoneda Lemma in the form \cite[(2.2)]{jfcs}
and the relation \cite[(2.4)]{jfcs} between the coend in left-exact functors and the Lyubashenko coend. 
The upper horizontal double-headed arrow in the left diagram is a zigzag that is obtained by performing the same sequence of operations on marked blocks, but with the following replacements for
\cite[(2.2)]{jfcs} and \cite[(2.4)]{jfcs}:
\begin{itemize}
	\item Instead of \cite[(2.2)]{jfcs}, we use the canonical equivalence
	\begin{align*}
\lint^{P\in\Proj \cat{C}} \cat{C}(X,P)\otimes\cat{C}(P,Y) \simeq \cat{C}(X,Y) 
\end{align*}
from  \cite[Lemma~4.11]{dva} that holds for objects $X$ and $Y$ in $\cat{C}$ if either $X$ or $Y$ is projective
(in order to apply this, it is crucial that $\underline{X}^{\otimes}$ is projective, which uses $n\ge 1$; for $n=0$, the object $\underline{X}^{\otimes}=I$ would not be necessarily projective --- in fact, it is projective if and only if $\cat{C}$ is semisimple).

\item Instead of \cite[(2.4)]{jfcs}, we use Proposition~\ref{propoitcoend}. 

\end{itemize}
Therefore,  the upper horizontal arrow in the left diagram is a zigzag of equivalences.
Moreover, since $\underline{X}^{\otimes}$ is projective, $\cat{C}(\underline{X}^{\otimes},-)$ is exact, which means that the right vertical map in the left diagram is an equivalence (again, this only holds because of $n\ge 1$). This implies that  the augmentation fibration
$\block_\cat{C}^{\Sigma,\Gamma} (\underline{X})\to H_0 \block_\cat{C}^{\Sigma,\Gamma} (\underline{X})$
is an equivalence.

It remains to exhibit a canonical isomorphism $H_0 \block_\cat{C}^{\Sigma,\Gamma} (\underline{X})\cong \ordblock_\cat{C}^{\Sigma,\Gamma}(\underline{X})$.
To this end, we take the zeroth homology of the left diagram which gives us the right diagram. Note that we invert the lower horizontal map. Now we define the dashed isomorphism such that the right square commutes.	\end{proof}

	\begin{remark}
The triviality of higher homologies for marked blocks
for surfaces with \emph{non-empty} boundary comes here from the fact that we only allow 
projective boundary labels (which we do to give a technically more uniform treatment). 
Non-projective boundary labels can be considered in a straightforward way once the modular functor is constructed as we will explain in Remark~\ref{remproj}. 
	\end{remark}

\spaceplease
\section{The main result}
In this section
 we state our main result that a (not necessarily semisimple) modular category gives rise to a \emph{modular functor with values in chain complexes}
(we will give the precise notion in Definition~\ref{predefdmf} and \ref{defdmf} below)
and that we can compute this modular functor in terms of marked blocks.
The proof of this result will occupy the rest of the article. 

\subsection{The category of surfaces and its central extension\label{secsurfacecat}}
The notion of a  modular functor uses the \emph{category of extended surfaces}
 $\Surf$ and a certain central extension of it, see \cite[Section~3]{gilmermasbaum} and 
 additionally \cite{hatcher,baki,segal,turaev} for more background. 
 
 Let us first define $\Surf$: 
 Objects are extended surfaces as defined in Section~\ref{secsurfaceterminology}. Morphisms are generated by mapping classes 
(isotopy classes of orientation-preserving diffeomorphisms  mapping marked points to marked points)
 and sewings  of surfaces that glue one or several pairs of incoming and outgoing boundary components together. These generators are subject to the obvious relations for the composition of mapping classes and sewings separately and the following mixed relation involving both mapping classes and sewings: If $\phi : \Sigma_0 \to \Sigma_1$ is a mapping class and $s : \Sigma_0 \to  \Sigma_0'$ a sewing, then
$\phi$ induces a sewing $s' : \Sigma_1 \to  \Sigma_1'$ of $\Sigma_1$, and $s$ induces a mapping class $  \phi' :  \Sigma_0'\to\Sigma_1'$\label{refsewnmappingclass}. The morphisms are subject to
\begin{align}
s'\phi= \phi' s \ .  \label{eqnmixedrelation}\end{align}
Disjoint union makes $\Surf$ into a symmetric monoidal category.

\subparagraph{Central extensions.}\label{seccentralextensions}
The group  $\Map(\Sigma):=\Aut_{\Surf}(\Sigma)$ is the mapping class group of the extended surface $\Sigma$.
Hence, functors out of the category $\Surf$ of surfaces describe in particular mapping class group representations. In order to describe certain \emph{projective} mapping class group actions 
 relevant in the theory of modular functors, one may use the central extension  
 $\Surfc$ of $\Surf$ discussed in \cite[Section~3]{gilmermasbaum}. Here $\Surfc$ is a certain category with the same objects as $\Surfc$ which comes with a functor
$P:\Surfc \to \Surf$ inducing for every extended surface $\Sigma$ a short exact sequence
\begin{align*} 0 \to \mathbb{Z} \to \Aut_{\Surfc}(\Sigma) \to \Aut_\Surf (\Sigma) \to 0\end{align*}
of groups.
Roughly, one obtains $\Surfc$ from $\Surf$ by introducing for each extended surface $\Sigma$ an additional generator $\C_\Sigma$ which commutes with all mapping classes and which 
behaves multiplicatively under sewing, i.e.\ for any sewing $s : \Sigma \sqcup \Sigma' \to   \Sigma''$, the relation
\begin{align}
s (\C_\Sigma^{p} \sqcup \C_{\Sigma'}^q)=\C_{\Sigma''}^{p+q}s \quad \text{for}\quad p,q \in \mathbb{Z}\ \label{eqnmultiplicatively}
\end{align}
holds.
For details, we refer to \cite{gilmermasbaum}, moreover \cite[Section~3.2]{jfcs}
and also Remark~\ref{remframinganomaly} and \ref{anomalouscase}, where we discuss the origin of the projectiveness and give a concrete model for the central extension in \eqref{eqnmodelcentralextension} (that one could also take as a definition of $\Surfc$). 
The  functor $P:\Surfc \to \Surf$  sends the central generators of $\Surfc$ to identities.

\subparagraph{Labeled surfaces.}
For any set  $\mathfrak{X}$ (to be thought of as \emph{label set}), we can define the category $\mathfrak{X}\text{-}\Surf$ whose objects are extended surfaces with boundary components  labeled by
elements of $\mathfrak{X}$.
Morphisms are again given by mapping classes (acting in the obvious way on the labels)
and sewings with the restriction that a sewing of an incoming to an outgoing boundary component is only allowed if the labeling objects coincide. The label for the sewn surface is obtained by omitting the labels of the boundary components that are sewn together.  
Disjoint union provides again a symmetric monoidal structure. 
Of course, we also obtain a labeled version $\mathfrak{X}\text{-}\Surfc$
of the central extension $\Surfc$.

\subsection{The notion of a modular functor\label{secdefdmf}}
All  definitions for the notion of a modular functor
that abound in the literature have in common  that a
  modular functor should provide a consistent system of (projective) mapping class group representations  in vector spaces compatible with the gluing of surfaces.

We will now lift this structure to a differential graded framework: Vector spaces will be replaced with chain complexes, the projective mapping class group actions will not necessarily be strict, but possibly up to coherent homotopy, also  the gluing properties have to be formulated homotopy coherently.

Let $\mathfrak{X}$ be a label set and $M : \mathfrak{X}\text{-}\Surfc \to \Ch$  a symmetric monoidal functor
whose value on an extended surface $\Sigma$ with label $\underline{X}=(X_1,\dots,X_n)$ for the $n$ boundary components we denote by $M(\Sigma,\underline{X})$.
For the moment, this functor is assumed to be strict, but see Remark~\ref{remhomotopycoherent}.

\begin{remark}[Cylinder category and functorial dependence on boundary label]\label{cylcatfunctorialdeprem}
	To the cylinder with incoming boundary label $X \in \mathfrak{X}$ and outgoing boundary label $Y \in \mathfrak{X}$,  a symmetric monoidal functor $M : \mathfrak{X}\text{-}\Surfc \to \Ch$ assigns a chain complex $M(\mathbb{S}^1\times [0,1],(X,Y))$. Since the sewing of cylinders is associative, 
	$\mathfrak{X}$ becomes the set of objects for a (not necessarily unital) differential graded category that we  denote by $\mathcal{Z}_M$ and refer to as the \emph{cylinder category of $M$}.
		We can now deduce a functorial dependence of the values of $M$ on cylinder category:
	Let $\Sigma$ be an extended surface with one outgoing boundary component (more boundary components can be treated in the same way). Then $M(\Sigma,-) : \cat{Z}_M\to\Ch$ is a functor on the cylinder category.  In fact, for $X,Y \in \mathfrak{X}$, sewing a cylinder to $\Sigma$ yields a map
	\begin{align}
	\cat{Z}_M(X,Y) \otimes M(\Sigma,X)=M(\mathbb{S}^1 \times [0,1] , (X,Y)) \otimes M(\Sigma,X)\to M(\Sigma \cup_{\mathbb{S}^1} (  \mathbb{S}^1 \times [0,1]  ),Y)\cong M(\Sigma,Y) \ . 
	\end{align}
	\end{remark}

 Let $s : \Sigma \to  \Sigma'$ be a sewing that glues an incoming to an outgoing boundary component. Then for any label $\underline{X}$ of the remaining boundary components, $M(\Sigma,(\underline{X},-,-))$ provides a functor $\cat{Z}_M^\op \otimes \cat{Z}_M \to \Ch$ (by the explanations in Remark~\ref{cylcatfunctorialdeprem}), i.e.\ an $\cat{Z}_M$-bimodule. 
	Evaluation of $M$ on $s$ induces a map
	\begin{align}
	\lint^{P \in \cat{Z}_M} M(\Sigma,(\underline{X},P,P)) \to M( \Sigma', \underline{X}) \ . \label{eqnexcisionmap1}
	\end{align}
	(Note that in \cite{dva} homotopy coends over differential graded categories are defined, but the above homotopy coend runs over \emph{a not necessarily unital} differential graded category. In that case, the homotopy coend can still be defined as the realization of a \emph{semisimplicial} object
	instead of a simplicial one.)
	
	We say that the symmetric monoidal functor $M:\mathfrak{X}\text{-}\Surfc \to \Ch$ satisfies
	\emph{excision} if \eqref{eqnexcisionmap1} is an equivalence.

\begin{definition}\label{predefdmf}
	For a set $\mathfrak{X}$, we call a 
	 symmetric monoidal functor $M:\mathfrak{X}\text{-}\Surfc \to \Ch$ a \emph{modular functor (with values in chain complexes)}
	if 
	  $M$ satisfies excision. 
	\end{definition}

	\begin{remark}\label{remhomotopycoherent}
	As mentioned above, one can relax the definition in the sense  that one requires this functor not to be strict, but just homotopy coherent. 
	Technically, one would accomplish this by passing from $\cat{C}\text{-}\Surfc$ to a resolution, see \cite{riehl18} for the necessary techniques.
	For the concrete constructions in this article, however, this will not be necessary because the modular functor that we will build will be strict. Nonetheless, it will lead to non-strict actions on certain Hochschild complexes because these will be obtained by transfer along an equivalence.
	\end{remark}

	In practice, a modular functor is constructed from a certain linear category (with a lot more structure and properties)
	that is recovered by evaluation of the modular functor on the cylinder leading to a notion of a modular functor \emph{for a given linear category}. We will take this into account in our definitions as follows: For a linear category $\cat{A}$, we set $\cat{A}\text{-}\Surfc := (\Proj \cat{A})_0 \text{-}\Surfc$, where $(\Proj \cat{A})_0$ is the set of objects of the subcategory $\Proj \cat{A} \subset \cat{A}$ of projective objects. 
	
	\begin{definition}\label{defdmf}
	Let $\cat{A}$ be a linear category.
		We call a modular functor $M:\cat{A}\text{-}\Surfc \to \Ch$ a \emph{modular functor for $\cat{A}$} if the cylinder category $\cat{Z}_M$ of $M$ is equivalent to $\Proj \cat{A}$. 
		\end{definition}
	
\begin{remark}[Modular functors with values in vector spaces]\label{remunderived}
	Replacing in the above Definitions~\ref{predefdmf} and~\ref{defdmf} the category of chain complexes
	with the category of vector spaces, we obtain the definition of an  modular functor with values in vector spaces (we should say `a' definition because other definitions might be used in different contexts; our definition is close to \cite{tillmann}). It is then clear that the zeroth homology of a modular functor with values in chain complexes is an modular functor with values in vector spaces.
	\end{remark}

\spaceplease
\subsection{The modular functor of a non-semisimple modular category\label{secdmfmain}}
Before stating our main result that a (not necessarily semisimple) modular category gives rise to a modular functor
with values in chain complexes, we recall some terminology.

\subparagraph{Ribbon categories, non-degeneracy and modularity.}
A \emph{ribbon structure} (also called \emph{ribbon twist}) on a  finite tensor category $\cat{C}$ with braiding $c$ is a natural automorphism of the identity whose components $\theta_X : X \to X$ satisfy the conditions
\begin{align*}
\theta_{X\otimes Y} &= c_{Y,X}c_{X,Y} (\theta_X \otimes \theta_Y) \ , \\
\theta_I &=\id_I \ , \\
\theta_{X^\vee} &=\theta_X^\vee
\end{align*} for all $X,Y \in \cat{C}$. 
The first two conditions say precisely that the braided monoidal structure and the natural automorphism $\theta$
form an algebra over the framed little disk operad \cite{salvatorewahl}; the last condition requires an additional compatibility with duality. 
A \emph{finite ribbon category} is a braided finite tensor category with the choice of a ribbon structure. Recall that any ribbon structure induces a pivotal structure. 
If we are given a braided finite tensor category $\cat{C}$, we can define the \emph{Müger center} as the full subcategory of $\cat{C}$ spanned by the transparent objects, i.e.\ those objects $X\in \cat{C}$ such that $c_{Y,X}c_{X,Y}=\id_{X\otimes Y}$ for all $Y\in\cat{C}$. 
We call the braiding (and then also the braided finite tensor category $\cat{C}$) \emph{non-degenerate} if the Müger center is trivial, i.e.\ generated by the monoidal unit under finite direct sums. 
A \emph{modular  category}
is a finite ribbon category with non-degenerate braiding. Note that, in our terminology, modularity does \emph{not} include semisimplicity (but still finiteness assumptions).
Semisimple modular categories are a standard object in quantum topology and known as the input datum for the Reshetikhin-Turaev construction \cite{rt1,rt2,turaev}.
For non-semisimple modular categories, there were for a long time different definitions
of the non-degeneracy of the braiding
which were  proven to be equivalent by Shimizu 
\cite{shimizumodular} only recently --- the one given above in terms of the Müger center is one of possible definitions of non-degeneracy.\myskip

We may now state our main result:

\begin{theorem}[Main Theorem]\label{thmmain}
	Any modular category $\cat{C}$ gives rise in a canonical way to a modular functor 
	\begin{align}
	\mfc\ :\ \cat{C}\text{-}\Surfc\to\Ch \label{eqnmfcfinal}
	\end{align}
	with values in chain complexes.
	This modular functor can be computed by the marked blocks from Section~\ref{derivedconformalblockssec} as follows: After the choice of a marking $\Gamma$ on $\Sigma$, there is a canonical equivalence
	\begin{align}
\block^{\Sigma,\Gamma}_\cat{C}(\underline{X}) \ra{\simeq}	\mfc(\Sigma,\underline{X})  \  \label{eqnagreementconfblock3}
	\end{align}
	of chain complexes.
\end{theorem}

This covers both parts of the main result as formulated in the introduction. 
The relation of the zeroth homology $H_0\mfc$ of $\mfc$ to classical constructions will be explained in Remark~\ref{rellyu}.
The remaining Sections~\ref{seclego} and~\ref{secconstruc} of this article are devoted to the proof of Theorem~\ref{thmmain}. 
In the remainder of this section, we will discuss implications and additions to Theorem~\ref{thmmain} and also examples.\myskip

Since homotopy coherent actions transfer along equivalences, we obtain the following immediate consequence of \eqref{eqnagreementconfblock3}:

\begin{corollary}\label{cormain0}
	Let  $\cat{C}$ be a modular category 
	and $\Sigma$ an extended surface with  marking $\Gamma$ and projective boundary label $\underline{X}$.
	Then the marked block $\block^{\Sigma,\Gamma}_\cat{C}(\underline{X})$ comes canonically with a homotopy coherent projective action of the mapping class group of $\Sigma$.
\end{corollary}

\begin{remark}[Framing anomaly]\label{remframinganomaly}
	For every modular category $\cat{C}$, one can define an invertible scalar $\zeta \in k^\times$, the 
	\emph{framing anomaly} (or \emph{central charge}) \cite{turaev,gilmermasbaum}.
	A modular category whose central charge is $1 \in k$ will be called \emph{anomaly-free}. 
	The framing anomaly controls the projectivity of the mapping class group actions that are part of the vector space valued modular functor for $\cat{C}$ in the sense that the central generator of a surface with genus $g$ is sent to multiplication with $\zeta^g$.
	Hence, in the anomaly-free case, these mapping class group actions will be linear and not only projective.
	The projectivity of the mapping class group actions obtained from the modular functor with values in chain complexes from Theorem~\ref{thmmain}  is precisely the same. 
	This will be explained in detail in Remark~\ref{anomalouscase}.
	\end{remark}

\begin{remark}[Non-projective boundary labels]\label{remproj}
	In the definition of $\cat{C}\text{-}\Surfc$,
	we only allow projective boundary labels. The choice leads to simplifications in the presentation (it matches better with excision arguments because our homotopy coends always run over subcategories of projectives), but we \emph{can} define  conformal blocks for non-projective labels although this does not amount to a substantial addition (because our modular functor with its present definition already contains all the necessary information): 
	Let $\Sigma$ be an extended surface. 
	For \emph{any} boundary label $\underline{X}$, we denote by $\cof \underline{X}$ the boundary label for $\Sigma$ with differential graded objects in $\cat{C}$ that resolves every outgoing label projectively and every ingoing label injectively (this replacement can be chosen functorially). Since $\cat{C}$ is self-injective (i.e.\ the projective objects are precisely the injective ones), $\cof \underline{X}$ is degree-wise projective.
	Therefore, we can apply $\mfc(\Sigma,-)$ degreewise and obtain 
	the $|\underline{X}|+1$-fold complex 
 $\mfc(\Sigma,\cof \underline{X})$ 
 whose totalization we define to be $\mfc(\Sigma,\underline{X})$. In order to see that this  does not depend on the choice of functorial resolution, let $\Sigma_g^{0|1}$ be a surface of genus $g\ge 1$ (the case $g=0$ follows directly from \cite[Lemma~3.8]{dva})
	 and with one boundary component that is outgoing (more boundary components, some of them possibly incoming, can be treated analogously). 
	 Denote by $\Gamma_{g,0,1}^\circ$ the marking on $\Sigma_g^{0|1}$ discussed in Example~\ref{exblock}. There we found in \eqref{eqnexblock2} a canonical equivalence
	\begin{align} \block^{\Sigma_g^{0|1},\Gamma_{g,0,1}^\circ}_\cat{C}(Y)\simeq \cat{C}  \left(  I,Y \otimes \left(     \flint^{P\in\Proj \cat{C}} P\otimes P^\vee       \right)^{\otimes g}    \right)\ , \quad Y\in \Proj \cat{C}\ .  \label{eqncomputationwithblocks}
	\end{align}
	In order to prove independence of $\mfc(\Sigma_g^{0|1},\cof X)$ of the chosen resolution, it suffices by naturality of the maps \eqref{eqnagreementconfblock3} to observe that the chain complex valued functor
	$ \cat{C}  \left( I, - \otimes \left(     \flint^{P\in\Proj \cat{C}} P\otimes P^\vee       \right)^{\otimes g}    \right)$ preserves equivalences between non-negatively graded complexes of projective objects in $\cat{C}$. But this follows from the exactness of the monoidal product and \cite[Lemma~3.8]{dva}.
	 By functoriality of $\mfc$ in boundary labels (Remark~\ref{cylcatfunctorialdeprem})
	 the complexes will again carry an action of the mapping class group of $\Sigma$.
	 It should also be noted that from \eqref{eqncomputationwithblocks}, \eqref{eqnagreementconfblock3} and \cite[Lemma~3.8]{dva} it follows that
	 \begin{align}
	 \mfc(\Sigma_g^{0|1},I)\simeq \mfc(\Sigma_g) \ , 
	 \end{align}
	 where $\Sigma_g=\Sigma_{g}^{0|0}$ is the closed surface of genus $g$.
	\end{remark}

\subsection{A canonical family of bimodules\label{secbimodules}}
The marked block for the closed torus and a specific marking
is given by the Hochschild complex (Example~\ref{exhochschild}), and it was shown already in \cite{dva} that the Hochschild complex of a modular category carries a homotopy coherent $\SL(2,\mathbb{Z})$-action, see \cite{svea} for a related Hopf algebraic result on (co)homology level.
The fully fledged modular functor of a modular category provides now for us complexes that we may interpret as higher genus analoga of the Hochschild complex together with (homotopy coherent) projective mapping class group actions on them.

\subparagraph{Bimodules and the Agreement Principle.}
For a linear category $\cat{C}$, a $\cat{C}$-bimodule is a functor $M:\cat{C}^\op\otimes\cat{C}\to\Vect_k$. 
We define the Hochschild complex of $\cat{C}$ with coefficients in this bimodule as the homotopy coend $CH(\cat{C};M):=\lint^{X\in\Proj \cat{C}} M(X,X)$. 
This definition is made in a way such that the following holds: If $\cat{C}=\Mod_k A$ is given as the category of finite-dimensional modules over a finite-dimensional algebra $A$, then by the \emph{Agreement Principle},
that goes back to McCarthy and Keller \cite{mcarthy,keller} 
and is stated in a slightly modified form in \cite[Theorem~2.9]{dva},
the canonical embedding $\iota_A : \star // A^\op \to \Proj _k A=\Proj \cat{C}$
from the category with one object with endomorphisms given by $A^\op$ into $\Proj\cat{C}$ induces an equivalence
\begin{align}
CH(A;M(A,A)) \ra{\simeq} CH(\cat{C};M)  \label{eqnagreementprinciple}
\end{align}
of chain complexes,
where $CH(A;M(A,A))$ are the `ordinary' Hochschild chains of $A$ with coefficients in the bimodule $M(A,A)$. \myskip

For any finite tensor category $\cat{C}$, there is a canonical family $\left( M_g  \right)_{g\ge 0}$ of $\cat{C}$-bimodules defined by
\begin{align}
M_g (X,Y)= \cat{C}(X,Y\otimes \mathbb{F}^{\otimes g}) \ , \quad X,Y \in \cat{C}\label{eqndefmg}
\end{align}
where $\mathbb{F}=\int^{X \in \cat{C}} X \otimes X^\vee$ is the canonical coend.

\begin{theorem}\label{cormain2}
For any modular category $\cat{C}$  and $g\ge 0$, the Hochschild chains $CH(\cat{C};M_g)$ with coefficients in the bimodule $M_g$  carry a canonical homotopy coherent projective action of the mapping class group $\Map(\Sigma_{g+1})$ of the closed surface of genus $g+1$.
\end{theorem}

The proof strategy was already outlined in the introduction: We realize that $CH(\cat{C};M_g)$ is the marked block for a $\Sigma_{g+1}$ and a \emph{specific marking} and then apply our main result. 

\begin{proof}
	Denote by $\Gamma_{g,1,1}^\circ$ the marking from Example~\ref{exblock}
	on the surface $\Sigma_g^{1|1}$ with genus $g$ and one incoming boundary component labeled by $X\in\Proj \cat{C}$ and one outgoing boundary component labeled by $Y\in\Proj \cat{C}$.
By \eqref{eqnexblock1} we have a canonical equivalence
	\begin{align}
	\block_\cat{C}^{\Sigma_g^{1|1},\Gamma_{g,1,1}^\circ}(X,Y) \simeq \cat{C}\left(X,Y\otimes \left(  \flint^{P\in\Proj\cat{C}} P\otimes P^\vee \right)^{\otimes g}\right) \  .
	\end{align}
	Since $X$ is projective, we arrive at $\block_\cat{C}^{\Sigma_g^{1|1},\Gamma_{g,1,1}^\circ}(X,Y)\simeq \cat{C}(X,Y\otimes \mathbb{F}^{\otimes g})=M_g(X,Y)$. It is now a direct consequence of excision for marked surfaces (Proposition~\ref{propexcision})  that $\block_\cat{C}^{\Sigma_{g+1}^{0|0},\Gamma_{g+1,0,0}^\circ}\simeq CH(\cat{C};M_g)$. Now the statement follows from Corollary~\ref{cormain0}.
	\end{proof}

Theorem~\ref{cormain2} can be used to obtain homotopy coherent projective mapping class group actions on the Hochschild complexes of ribbon factorizable Hopf algebras with coefficients in specific bimodules: For a ribbon factorizable Hopf algebra $A$, the category $\Mod_k A$ is modular, see e.g.\ \cite[Section~2.3]{svea2}, and the
Lyubashenko coend $\mathbb{F}=\int^{X \in \Mod_k A} X \otimes X^\vee$ in $\Mod_k A$
is isomorphic to the $A$-module $A^*_\text{coadj}$ given by the dual $A^*$ of $A$
and the so-called \emph{coadjoint action}
\begin{align*}
A \otimes A^* \to A^* \ , \quad a \otimes \alpha \mapsto \left(  b \mapsto \alpha\left(      S(a_{(1)})ba_{(2)}    \right)   \right) \ , 
\end{align*}
see \cite[Theorem~7.4.13]{kl}, where $S$ is the antipode of $A$, and we have used the Sweedler notation $\Delta a = a_{(1)}\otimes a_{(2)}$ for the coproduct $\Delta$ of $A$.
Now for $g \ge 0$, we can consider the $A$-module $A \otimes \left(  A^*_\text{coadj}  \right)^{\otimes g}$ that is defined using the monoidal structure on $\Mod_k A$.
By right multiplication on the $A$-factor this becomes an $A$-bimodule. We obtain
\begin{align}
M_g\left(A, A \right)=\Hom_A\left(A, A \otimes \left(  A^*_\text{coadj}  \right)^{\otimes g}\right)=A \otimes \left(  A^*_\text{coadj}  \right)^{\otimes g} \ ,  \end{align}
where  $A \otimes \left(  A^*_\text{coadj}  \right)^{\otimes g}$ has to be seen as $A$-module in the middle of the equation and as $A$-bimodule on the right hand side.
The Agreement Principle \eqref{eqnagreementprinciple} provides us with a canonical equivalence
\begin{align}
CH \left(   A; A \otimes \left(  A^*_\text{coadj}  \right)^{\otimes g}  \right) \ra{\simeq} CH(\Mod_k A;M_g) .
\end{align}
This leads to the following Hopf algebraic version of Theorem~\ref{cormain2}:

\begin{corollary}\label{corha}
For any ribbon factorizable Hopf algebra $A$ and $g\ge 0$, the Hochschild chains of $A$ 
 with coefficients in the $A$-bimodule $A \otimes \left(  A^*_\text{coadj}  \right)^{\otimes g}$
come with a canonical homotopy coherent projective action of the mapping class group of the closed surface of genus $g+1$. 
\end{corollary}

\begin{remark}[]\label{remsvea}
	For a specific marking, the homology of the marked block of a modular category $\cat{C}$ can be computed rather explicitly. 
	We will focus on a closed surface $\Sigma_g$ of genus $g\ge 0$. In the presence of boundaries, we may use Proposition~\ref{proptrivfibbdy}.
First assume $g \ge 1$ (we will comment on the case $g=0$ in a moment).
Since any modular category is unimodular \cite[Proposition~4.5]{eno}, 
the equivalence \eqref{eqnexblock2} gives us an isomorphism
	\begin{align}
	H_* \block_\cat{C}^{\Sigma_g,\Gamma_g^\circ} \cong \Ext \left(   \mathbb{F}^{\otimes g},I  \right)^* \cong \Ext \left(   I,\mathbb{F}^{\otimes g}  \right)^* \ .   \label{eqnisoext} 
	\end{align}
	The second isomorphism uses the self-duality $\mathbb{F}^\vee\cong\mathbb{F}$ of the canonical coend of a modular category coming from the non-degenerate Hopf pairing, see \cite{shimizumodular}.
With the convention $\mathbb{F}^{\otimes 0}=I$,
the statement remains true for $g=0$ for an appropriate marking    (we need to cut the sphere into two disks along its equator).
	
	Of course, there is now also an isomorphism $H_*\mfc(\Sigma_g)\cong \Ext \left(   I,\mathbb{F}^{\otimes g}  \right)^*$, but this is not canonical. 
	  Under \eqref{eqnisoext}, the projective action of 
	  the mapping class group of $\Sigma_g$  on $H_* \block_\cat{C}^{\Sigma_g,\Gamma_g^\circ}$ corresponds to the one constructed in \cite{svea2} on $\Ext(I,\mathbb{F}^{\otimes g})$. We will, however, not spell out the details of this comparison here.
\end{remark}

\begin{example}[Drinfeld doubles in positive characteristic and the Dijkgraaf-Witten modular functor]\label{exdouble}
	For a finite group $G$, the Drinfeld double $D(G)$ is a ribbon factorizable Hopf algebra with underlying vector space
	$k(G) \otimes k[G]$, where $k(G)$ is the space of $k$-valued functions and $k[G]$ is the group algebra of $G$.
	The multiplication in $D(G)$ is given by 
		\begin{align}
	(\delta_a  \otimes b)    (\delta_c \otimes d) =     \delta_a \delta_{bcb^{-1}} \otimes b d \quad \text{for all}\quad a,b,c,d\in G \ , 
	\end{align}
	where $\delta_a$ is the $k$-valued function on $G$ supported in $a \in G$ with value $1\in k$; we refer to 
	\cite[Chapter~IX]{kassel}
	for details.
	The category $\Mod_k D(G)$ of finite-dimensional $D(G)$-modules is modular, and it is non-semisimple if and only if the characteristic of $k$ divides $|G|$.

Whenever $\Mod_k D(G)$ is semisimple, it gives rise to a topological field theory \cite{freedquinn,morton1} that is partly based on \cite{dw} and therefore often called \emph{Dijkgraaf-Witten theory}. Therefore, we call the modular functor $\mfg$ with values in chain complexes that exists by
Theorem~\ref{thmmain}
 the \emph{Dijkgraaf-Witten modular functor}.

	In this example, we prove that the evaluation of $\mfg$ 
	on a closed surface $\Sigma$ is  explicitly given by
	\begin{align}
	\mfg (\Sigma) \simeq N_*(\PBun_G(\Sigma);k) \ , \label{eqndgclosed}
	\end{align}
	where $N_*(\PBun_G(\Sigma);k)$ are the (normalized) $k$-chains on the groupoid $\PBun_G(\Sigma)$ of $G$-bundles over $\Sigma$ (for the closed torus, this is \cite[Proposition~3.3]{dva}). Under the equivalence, the mapping class group action on the left hand side corresponds to the obvious action on the right hand side.
Similar statements hold for surfaces with boundary. 

An in-depth discussion of $\mfg$ is beyond the scope of this article, and we will only sketch
 the proof of \eqref{eqndgclosed}:
	First recall the equivalence $\Mod_k D(G)\simeq \Mod_k G//G$ as linear categories, where $\Mod_k G//G$ is the category of finite-dimensional modules over the loop groupoid $G//G$ of $G$, i.e.\ the action groupoid of $G$ acting on itself by conjugation
	(there is naturally a ribbon structure on $\Mod_k G//G$ such that this equivalence is an equivalence of ribbon categories).
	Also recall $G//G\simeq \PBun_G(\mathbb{S}^1)$.
	If $\Sigma$ is an extended surface with projective boundary label $\underline{X}$ in $\Mod_k D(G)$, then we may see $\Sigma : S_0 \to S_1$ as a bordism from the incoming boundary $S_0$ to the outgoing boundary $S_1$. 
	Thanks to $\Mod_k D(G)\simeq \Mod_k \PBun_G(\mathbb{S}^1)$, the label $\underline{X}$ gives rise to projective objects $X_j \in \Mod_k \PBun_G(S_j)$ for $j=0,1$. 
	
	Restriction of bundle groupoids to the boundary yields a span of groupoids 
	\begin{equation}
	\begin{tikzcd}
	\PBun_G(S_0)&\PBun_G(\Sigma)\ar{l}[swap]{r_0}\ar{r}{r_1}&\PBun_G(S_1)
	\end{tikzcd}
	\end{equation} 
	that induces the pullback functors
		\begin{equation}
	\begin{tikzcd}
	\Ch^{\PBun_G(S_0)} \ar{r}{r_0^*} &\Ch^{\PBun_G(\Sigma)} &\Ch^{\PBun_G(S_1)}\ ,  \ar[swap]{l}{r_1^*}
	\end{tikzcd}
	\end{equation} 
	where $\Ch^{\PBun_G(M)}$ denotes the category of functors from the groupoid $\PBun_G(M)$ of $G$-bundles over a manifold $M$ to $\Ch$.
	We can now define the \emph{derived pull push functor}
	\begin{align}
	\DW (\Sigma) := \mathbb{L} {r_1}_! r_0^*\ :\ \Ch^{	\PBun_G(S_0)} \to   \Ch^{	\PBun_G(S_1)}\ ,   \label{eqnpullpush}
	\end{align}
	where $\mathbb{L} {r_1}_!$ is the homotopy left Kan extension along $r_1$.
	We also define the auxiliary complexes  \begin{align}\mathfrak{G}(\Sigma,\underline{X}):= \langle  X_1^\vee, \DW (\Sigma)X_0^\vee    \rangle \ , \end{align}
	 where $\langle -,-\rangle$ denotes the morphism spaces in $\Mod_k \PBun_G(S_1)$ understood degreewise. 
	Whenever $\Sigma$ is closed, we find $\mathfrak{G}(\Sigma)=N_*(\PBun_G(\Sigma);k)$ by definition.
	Hence, in order to prove \eqref{eqndgclosed}, it remains to prove that $\mathfrak{G}$ is actually equivalent to the modular functor $\mfg$. 
	To this end, we need the following statements:
	\begin{itemize}
		\item For composable bordisms $\Sigma:S_0\to S_1$ and $\Sigma':S_1\to S_2$, the pull push maps \eqref{eqnpullpush} satisfy $\DW(\Sigma'\circ \Sigma)\simeq \DW(\Sigma')\DW(\Sigma)$ as follows from a straightforward analogue of the Beck-Chevalley property for pull push maps \cite{morton1,extofk}
		to the derived setting. 
		This implies that $\mathfrak{G}$ satisfies excision for the gluing of \emph{disjoint} surfaces (we do not make a statement about self-sewing).

		\item On a connected surface of genus zero, $\mathfrak{G}$ agrees with the modular functor, i.e.\ it is given by the morphism spaces of $\Mod_k D(G)\simeq \Mod_k G//G$ and the monoidal product as in \eqref{eqndmfoncyl}. 
		This fact is proven by an explicit computation of the corresponding pull push map \eqref{eqnpullpush}:
		We write the homotopy left Kan extension as a homotopy colimit over the homotopy fiber of $r_1$.
		Since the homotopy fiber of any restriction functor $\PBun_G(\Sigma)\to\PBun_G(S)$ from a connected surface $\Sigma$ to a non-empty collection $S$ of boundary components has discrete homotopy fibers (this follows from the holonomy classification of bundles and the long exact sequence for homotopy groups), the homotopy fibers of $r_1$ are also discrete, and all homotopy colimits needed for the homotopy left Kan extension are just coproducts of vector spaces. 
		This allows us to verify the claim directly.
		
		\end{itemize}
	From these two statements, we can already conclude $\mathfrak{G}(\Sigma)\simeq \mfg(\Sigma)$ for every closed surface and thereby deduce \eqref{eqndgclosed}.
	The proof that, under the equivalence, the $\Map(\Sigma)$-action on $\mfg(\Sigma)$ corresponds to the topological one on $N_*( \PBun_G(\Sigma);k)$ is very similar to the vector space valued case. We will not discuss the details here. 
	\end{example}

\section{The homotopy coherent Lego-Teichmüller game\label{seclego}}
In Section~\ref{secderivedblocks} we have defined chain complexes of vector spaces for a surface and the auxiliary datum of a marking. 
In order to prove the main result, we will have to understand how these quantities depend on the marking.
Before investigating this point in the next section, we need to understand how different markings on a surface are related. 
To this end, we will use and extend work of Bakalov and Kirillov \cite{bakifm} who, based on \cite{hatcherthurston,harer}, define a groupoid of markings on a surface and prove that this groupoid is connected and simply connected. Some key ideas in \cite{bakifm} and also the name \emph{Lego-Teichmüller game} go back to Grothendieck's research proposal \emph{Esquisse d'un Programme} \cite{grothendieck}.
For our purposes, we will need to replace the contractible groupoid of markings on a given extended surface by a contractible $\infty$-groupoid. 
This replacement was motivated in the introduction. Its significance will become clear in the next section.

\subsection{Groupoids of  markings\label{secfinemarkings}}
For an extended surface $\Sigma$, we denote by $\FM(\Sigma)$ the \emph{groupoid of     markings on $\Sigma$} \cite{bakifm}. Its objects are     markings on $\Sigma$ (see Section~\ref{secsurfaceterminology}). The morphisms of $\FM(\Sigma)$ are given in terms of four generators (called \emph{moves}):
\begin{itemize}
	\item $\Z$ (cyclic rotation of the marking),
	\item $\F$ (cut deletion),
	\item $\B$ (braiding),
	\item $\S$ (passing to transversal cut on a torus with one hole).
	\end{itemize}
These four moves are subject to a list of relations and generate $\M(\Sigma)$ under sewing.
We refer to \cite[Section~4]{bakifm} for the detailed definition of the moves
and their relations.
We will recall the necessary aspects directly in the place where we will need them.

\begin{remark}
Instead of arbitrary markings, one can restrict to \emph{fine} markings  (for a fine marking, the underlying cut system contains spheres with at least one and at most three holes). This case is also treated in \cite{bakifm} and turns out to be not substantially different. A full list of the generators and relations for the groupoid of fine markings is given in \cite[Section~2.2]{jfcs}.
\end{remark}

Building on the results of \cite{hatcherthurston,harer}, Bakalov and Kirillov prove the following fundamental result about the groupoids of     markings:

\begin{theorem}[\text{\cite[Theorem~4.24]{bakifm}}]\label{thmbakifm}
	For any extended surface $\Sigma$,
	the groupoid $\FM(\Sigma)$ of     markings on  $\Sigma$ is 
	  connected and simply-connected.
\end{theorem}

We call  a category $\cat{A}$ \emph{contractible} if the topological space $|N\cat{A}|$ obtained by geometric realization of its nerve $N\cat{A}$ is contractible, i.e.\ equivalent to a point.
Put differently, in this case, the $\infty$-groupoid obtained by localizing (in the sense of $\infty$-categories) $\cat{A}$ at \emph{all} morphisms is a contractible Kan complex.
Hence, a groupoid is contractible if and only if it is connected and simply connected.
Therefore,  Theorem~\ref{thmbakifm} can be rephrased by saying that the groupoids of     markings are contractible. For groupoids, this additional terminology is not really necessary, but it will become relevant below for categories which are not groupoids. \myskip

For an extended surface $\Sigma$, Bakalov and Kirillov  also define a groupoid $\C(\Sigma)$ of cut systems on $\Sigma$ \cite[Section~7.1--7.3]{bakifm}. Its objects are cut systems on $\Sigma$ (as defined in Section~\ref{secsurfaceterminology}), and its morphisms are generated by the moves $\barF$ and $\barS$
(corresponding to  the moves $\F$ and $\S$ listed above)
which are explained in Figure~\ref{figbarfbars}.
 These two moves are subject to five relations.

\begin{figure}[h]
	\centering
\begin{tikzpicture}[scale=0.45, style/.style={circle, inner sep=0pt,minimum size=0mm},
poin/.style={circle, inner sep=0.2pt,minimum size=0.2mm},decoration={
	markings,
	mark=at position 0.5 with {\arrow{>}}}]
\node [style=none] (23) at (-28.5, 10) {};
\node [style=none] (24) at (-28.5, 8) {};
\node [style=none] (25) at (-20, 10) {};
\node [style=none] (26) at (-20, 8) {};
\node [style=none] (31) at (-24.25, 10) {};
\node [style=none] (32) at (-24.25, 8) {};
\node [style=none] (33) at (-30, 9) {$\dots$};
\node [style=none] (34) at (-19, 9) {$\dots$};
\node [style=none] (37) at (-24.75, 3.5) {};
\node [style=none] (38) at (-21.25, 3.5) {};
\node [style=none] (39) at (-25.5, 4.25) {};
\node [style=none] (40) at (-20.5, 4.25) {};
\node [style=none] (41) at (-23, 3) {};
\node [style=none] (43) at (-18.5, 3.75) {};
\node [style=none] (44) at (-23, 1) {};
\node [style=none] (45) at (-28.5, 4.75) {};
\node [style=none] (46) at (-28.5, 2.75) {};
\node [style=none] (47) at (-23, 6.25) {};
\node [style=none] (50) at (-30, 3.75) {$\dots$};
\node [style=none] (51) at (-7, 3.5) {};
\node [style=none] (52) at (-3.5, 3.5) {};
\node [style=none] (53) at (-7.75, 4.25) {};
\node [style=none] (54) at (-2.75, 4.25) {};
\node [style=none] (55) at (-5.25, 3) {};
\node [style=none] (56) at (-0.75, 3.75) {};
\node [style=none] (57) at (-5.25, 1) {};
\node [style=none] (58) at (-10.75, 4.75) {};
\node [style=none] (59) at (-10.75, 2.75) {};
\node [style=none] (60) at (-5.25, 6.25) {};
\node [style=none] (61) at (-5.25, 2) {};
\node [style=none] (62) at (-5.25, 5.75) {};
\node [style=none] (63) at (-12.25, 3.75) {$\dots$};
\node [style=none] (64) at (-17, 9) {};
\node [style=none] (65) at (-14, 9) {};
\node [style=none] (66) at (-11, 10) {};
\node [style=none] (67) at (-11, 8) {};
\node [style=none] (68) at (-2.5, 10) {};
\node [style=none] (69) at (-2.5, 8) {};
\node [style=none] (70) at (-6.75, 10) {};
\node [style=none] (71) at (-6.75, 8) {};
\node [style=none] (72) at (-12.5, 9) {$\dots$};
\node [style=none] (73) at (-1.5, 9) {$\dots$};
\node [style=none] (74) at (-15.5, 9.5) {$\barF$};
\node [style=none] (75) at (-17, 3.75) {};
\node [style=none] (76) at (-14, 3.75) {};
\node [style=none] (77) at (-15.5, 4.25) {$\barS$};
\node [style={small_black}] (78) at (-24.75, 8.5) {};
\node [style={small_black}] (79) at (-29, 3.25) {};
\node [style={small_black}] (80) at (-11.25, 3.25) {};
\node [style={small_black}] (81) at (-23.5, 1.5) {};
\node [style={small_black}] (82) at (-7.75, 2.6) {};
\draw (23.center) to (25.center);
\draw (24.center) to (26.center);
\draw [bend right=90] (31.center) to (32.center);
\draw [bend left=90,dotted] (31.center) to (32.center);
\draw [bend left=45] (37.center) to (38.center);
\draw [bend right=60] (39.center) to (40.center);
\draw [bend right=90,dotted] (44.center) to (41.center);
\draw [bend left=90] (44.center) to (41.center);
\draw [bend right=90,dotted] (46.center) to (45.center);
\draw [bend left=90] (46.center) to (45.center);
\draw [in=0, out=-90] (43.center) to (44.center);
\draw [in=0, out=180] (44.center) to (46.center);
\draw [in=0, out=90] (43.center) to (47.center);
\draw [in=0, out=-180] (47.center) to (45.center);
\draw [bend left=45] (51.center) to (52.center);
\draw [bend right=60] (53.center) to (54.center);
\draw [bend right=90,dotted] (59.center) to (58.center);
\draw [bend left=90] (59.center) to (58.center);
\draw [in=0, out=-90] (56.center) to (57.center);
\draw [in=0, out=180] (57.center) to (59.center);
\draw [in=0, out=90] (56.center) to (60.center);
\draw [in=0, out=-180] (60.center) to (58.center);
\draw [in=0, out=0, looseness=3.00] (62.center) to (61.center);
\draw [in=180, out=-180, looseness=3.00] (62.center) to (61.center);
\draw (66.center) to (68.center);
\draw (67.center) to (69.center);
\draw [->] (64.center) to (65.center);
\draw [->] (75.center) to (76.center);
\end{tikzpicture}
\caption{The $\barF$-move can be applied to a cylindrical region with a cut and deletes this cut 
 --- provided, of course, that this still leaves us with a cut system. The dots are supposed to symbolize that the surface may continue to the left and the right of the displayed region. The $\barS$-move can be applied to a region of the shape of torus with one hole and one cut. It replaces the one cut with a transversal one. 
The marked points on the boundary and the orientation of cuts and boundaries 
 are suppressed in the picture. For the $\barF$-move,
 the cut that is being deleted can have any orientation and any position of the marked point.
 For the $\barS$-move, the cut that is being replaced by a transversal one as well as its replacement can have arbitrary orientation and position of the marked point. The other cut (or boundary component if the surface ends there) has some arbitrary orientation and position of the marked point that is not changed by the $\barS$-move. 
}
\label{figbarfbars}
\end{figure}
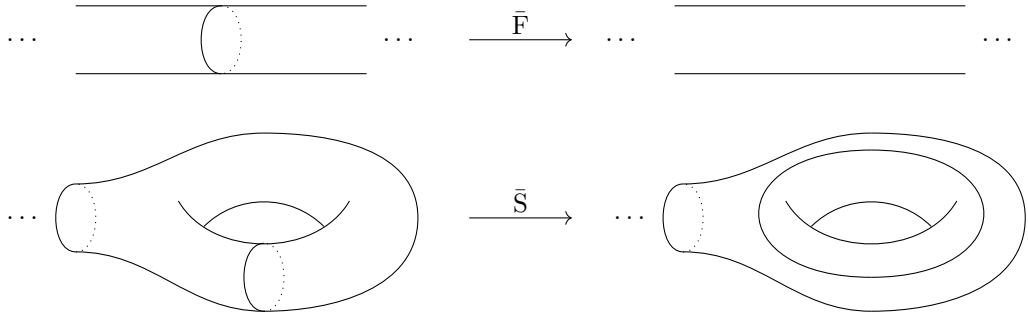

In addition to Theorem~\ref{thmbakifm}, we will also need the following contractibility result (that the proof of Theorem~\ref{thmbakifm} cited above actually relies on):

\begin{theorem}[\text{\cite[Theorem~7.9]{bakifm}}]\label{thmbakifmcut}
	For any extended surface $\Sigma$,
	the groupoid $\C(\Sigma)$ of     cut systems on  $\Sigma$ is 
	connected and simply-connected.
\end{theorem}

Let $\phi : \Sigma \to \Sigma'$ be a mapping class, then $\phi$ sends cut systems on $\Sigma$ to cut systems on $\Sigma'$ and also moves between cut systems on $\Sigma$ to moves between cut systems on $\Sigma'$ (because moves are only defined based on incidences). This way, $\phi$ yields a functor $\C(\phi) : \C(\Sigma)\to\C(\Sigma')$ between groupoids which by Theorem~\ref{thmbakifmcut} is even determined by its object function.
Similarly, any sewing $s : \Sigma \to \Sigma'$ yields a functor $\C(s):\C(\Sigma)\to\C(\Sigma')$ which just regards any pair oppositely oriented  gluing boundaries of $\Sigma$ as a cut in $\Sigma'$. The functors assigned to mapping classes and sewings respect the relations holding in $\Surf$.
These considerations carry over from cuts to markings (for the action of mapping classes on the charts underlying the marking, one precomposes the chart with the inverse of the mapping class).

\begin{proposition}\label{propmark}
	Cut systems and markings on extended surfaces 
	 naturally form symmetric monoidal  functors
	\begin{align} \C : \Surf &\to \Grpd \ , \\
	\M : \Surf &\to \Grpd \ ,
	\end{align}where the monoidal product on $\Surf$ is disjoint union and the monoidal product on $\Grpd$ is the Cartesian product.
	\end{proposition}

\subsection{Describing cut systems and markings via the Grothendieck construction}
As explained in \cite[Section~7.4]{bakifm}, there is a projection functor\begin{align}\pi_\Sigma : \FM(\Sigma)\to\C(\Sigma)\label{eqnprjectionfunctor}\end{align} sending a marking to its underlying cut system. 
The moves $\B$ and $\Z$ are sent to identities while $\F$ and $\S$ are sent to $\barF$ and $\barS$, respectively.

By means of this functor, we can see that the markings over a fixed extended surface form a \emph{category fibered in groupoids} over the cut systems on that surface.
To this end, let us first recall the relevant notions from  \cite[Definition~3.1]{hollander} (this definition is based on \cite{delignemumford}):
A functor $E: \cat{A}\to\cat{B}$ is called a \emph{category fibered in groupoids} (or said to exhibit $\cat{A}$ as a category fibered in groupoids over $\cat{B}$)  
\begin{xenumerate}
	\item if all lifting problems of the form
	\begin{equation}
\begin{tikzcd}
0 \ar[swap]{dd}{} \ar[]{rr}{} & & N\cat{A}  \ar{dd}{NE }   \\
& & \\
\Delta^1 \ar{rr}{ } \ar[dashed]{rruu}{}  & & N\cat{B}  
\end{tikzcd} 
\end{equation}
can be solved \label{cfig1}
\item and if for any diagram $a \la{f} b \ra{g} c$ in $\cat{A}$ and any morphism $h : E(a)\to E(c)$ in $\cat{B}$
making the diagram
	\begin{equation}
\begin{tikzcd}
 & & E(c)     \\
& & \\
E(a) \ar[]{rruu}{h}  & & E(b) \ar{ll}{ E(f)}  \ar{uu}[swap]{E(g) }
\end{tikzcd} 
\end{equation}
commute, 
there is a unique $h' : a \to c$ with $E( h')=h$.  \label{cfig2}

\end{xenumerate}
(Note that point~\ref{cfig2} uses conventions dual to those in \cite[Definition~3.1]{hollander} because the categories fibered in groupoids needed in this article  correspond to category-valued cosheaves instead of sheaves as in \cite{hollander}.)

\begin{lemma}\label{lemmafmfiberedingrpds}
For any extended surface $\Sigma$, the canonical functor $\pi_\Sigma : \FM(\Sigma) \to \C(\Sigma)$ exhibits $\FM(\Sigma)$ as a category fibered in groupoids over $\C(\Sigma)$.  
\end{lemma}

\begin{proof}
It follows from the definition of $\FM(\Sigma)$ that the lifting problem 
	\begin{equation}
	 \begin{tikzcd}
	 0 \ar[swap]{dd}{} \ar[]{rr}{\Gamma} & & N\FM(\Sigma)  \ar{dd}{N\pi_\Sigma }   \\
	 & & \\
	 \Delta^1 \ar{rr}{\mu } \ar[dashed]{rruu}{\widetilde \mu}  & & N\C(\Sigma) \ .  
	 \end{tikzcd} 
	 \end{equation}
can be solved if $\mu$ is one of the moves $\barF$ or $\barS$ (because by definition these lift to $\F$ and $\S$, respectively, for any given start value). From this, we deduce that the lifting problem can be solved when $\mu$ is an arbitrary morphism in $\C(\Sigma)$, which amounts to property~\ref{cfig1} above.
Contractibility of the groupoids $\FM(\Sigma)$ and $\C(\Sigma)$ gives us property~\ref{cfig2}.
\end{proof}

For a cut system $C$ on $\Sigma$, denote by
 \begin{align} \m_\Sigma (C):=  \pi_\Sigma^{-1}(C)  \label{eqndefmsigma}   \end{align} the fiber of the projection functor $\pi_\Sigma$ from \eqref{eqnprjectionfunctor} over $C$. 
 From  Theorem~\ref{thmbakifm} and~\ref{thmbakifmcut} we can deduce the equivalence
 \begin{align}
 \m_\Sigma (C)\simeq \star \label{eqnmsigmacontrac}
 \end{align}
 of categories.

 In the sequel, it will be important to reconstruct $\M(\Sigma)$ from the fibers \eqref{eqndefmsigma}.
 To this end, we will use the \emph{Grothendieck construction}, a  classical construction in category theory, see e.g.\
 \cite[Section~I.5]{maclanemoerdijk}, that we will also use for various other constructions later: For a functor $F: \cat{A} \to \Cat$ from a category $\cat{A}$ to  the category of categories, its Grothendieck construction $\int F$ is defined as the category of pairs $(a,x)$, where $a\in \cat{A}$ and $x\in F(a)$. A morphism $(a,x) \to (a',x')$ is a pair $(f,\alpha)$ of a morphism $f: a \to a'$ in $\cat{A}$ and a morphism $\alpha : F(f)x\to x'$ in $F(a')$. The category $\int F$ comes with a natural functor $\int F\to \cat{A}$ sending an object $(a,x) \in \int F$ to $a$ and a morphism $(f,\alpha):(a,x)\to(a',x')$ in $\int F$ to $f$. 
 For later purposes, we record the useful fact that the Grothendieck construction $\int F$ of a functor $F:\cat{A}\to\Cat$ can be described as certain type of lax colimit of the functor $F$ (referred to as \emph{oplax colimit} in the most common convention);
 this statement can be found e.g.\ in \cite[Theorem~10.2.3]{johnsonyau} or in \cite{lax} within a more general framework.\label{pagereflax}
 One should think of $\int F$ as the result of gluing together in a categorical way the categories $F(a)$ for $a\in \cat{A}$ according to a prescription determined by $\cat{A}$.
 \myskip

 Now from Lemma~\ref{lemmafmfiberedingrpds} and \cite[Theorem~1.2]{hollander} (or rather the dual version) it follows that the fibers \eqref{eqndefmsigma} form a (pseudo-)functor
 \begin{align} \m_\Sigma : \C(\Sigma) \to \Grpd \label{eqnsmallmfunctor}
 \end{align}
 whose Grothendieck construction $\int \m_\Sigma$ comes with a canonical equivalence
 \begin{align}
 \int\m_\Sigma \xrightarrow{\ \simeq \ }
 \FM(\Sigma)\label{markingseqnviagrothendieck}
 \end{align}
 induced by the inclusions $\m_\Sigma(C)\to \FM(\Sigma)$.
This leads to a key observation that we will need later: We can write the groupoid of markings on a fixed surface as the result of categorically gluing together the markings over  varying cut systems, where the groupoid of cuts systems gives us the gluing prescription.

 \subsection{The categories of colored cuts and colored markings\label{seccolcutmarking}}
We will now introduce a new category of cut systems and markings on a fixed extended surface whose objects are cut systems and markings, respectively, equipped with additional data, namely a coloring.
The usefulness of this definition will become apparent in the next section, where it will allow us to relate marked blocks for different markings.

 For a connected extended surface $\Sigma$ with $n$ boundary components, a \emph{colored cut system} $U$ on $\Sigma$ is a pair $C$ of a cut system on $\Sigma$ and a subset $S$ of the set of cuts of $C$ such that
 \begin{align} |S|+n\ge 1 \ .   \label{eqncoloredcondition} \end{align}
 A cut that lies in the distinguished subset $S$ of all the cuts will be referred to as a  \emph{colored cut}. A cut which is not colored will be referred to as \emph{uncolored cut}. 
 A colored cut system on a non-connected extended surface is defined as a colored cut system on every connected component. 
 
 The colored cut systems on an extended surface $\Sigma$ form a category $\Chat(\Sigma)$
 in the following way: Objects are colored cut systems on $\Sigma$.
 The morphisms are generated by two types of moves:
 
 \begin{enumerate}
 	
 	\item[(U)] For any colored cut systems $U =(C,S)$ and any proper subset $S'\subset S$ such that $(C,S')$ is still a colored cut system (meaning that \eqref{eqncoloredcondition} must be satisfied), there is a non-invertible morphism \begin{align} U =(C,S) \xrightarrow{\  S'\subset S  \ } (C,S') \ , \end{align} called \emph{uncoloring}. In other words, there is a morphism that implements forgetting a coloring of subfamily of cuts if enough colored cuts are left to ensure that requirement \eqref{eqncoloredcondition} is met.
 	
 	\item[(AM)] Between colored cut systems, we have \emph{admissible moves}, i.e.\ a move between the underlying cut systems, as defined for $\C(\Sigma)$, such that this move does not affect the colored cuts. 
 	A more formal  definition of an admissible move may be given as follows: \label{notationadmissiblemoves}
 	For a colored cut system $U$ on $\Sigma$ with underlying cut system $C$, denote by $\Sigma_U$ the surface obtained from $\Sigma$ by cutting at all \emph{colored} cuts. Then $C$ induces a cut system $C_U$ on $\Sigma_U$. By Proposition~\ref{propmark} the re-sewing $s_U : \Sigma_U \to \Sigma$ gives rise to a functor \begin{align} \label{eqnsewingfunctorA} \C(s_U) : \C(\Sigma_U) \to \C(\Sigma)\end{align} sending $C_U$ to $C$. 
 	With this notation, we define an \emph{admissible move} $U\to V$ between colored cut systems $U$ and $V$ on $\Sigma$ as follows: Such an admissible move only exists when $\Sigma_U = \Sigma_V$, and in that case, it is defined as a move $\mu : C\to C'$ 
 	between the underlying cut systems that is the image of a move $C_U \to C_V$ of cut systems on $\Sigma_U=\Sigma_V$ under \eqref{eqnsewingfunctorA}. 
 	
 	\end{enumerate}
 \spaceplease
 We impose the following relations:
 
 \begin{enumerate}\label{relfmhat}
 	\item[(RU)] Suppose $U =(C,S)$ is a colored cut system and $S'' \subset S' \subset S$ proper inclusions of subsets such that $(\Gamma,S')$ and $(\Gamma,S'')$ are still colored cut systems, then the composition\label{relfmhatru}
 	\begin{align}
 	U =(\Gamma,S) \xrightarrow{\  S'\subset S  \ } (\Gamma,S') \xrightarrow{\  S''\subset S'  \ } (\Gamma,S'')
 	\end{align} of uncolorings corresponding to $S''\subset S'$ and $S'\subset S$, respectively, is equal to the uncoloring
 		\begin{align}
 	U =(\Gamma,S) \xrightarrow{\  S''\subset S  \ } (\Gamma,S'')
 	\end{align}
 	corresponding to $S''\subset S$. 
 	
 	\item[(RM)] For the composition of admissible moves, the relations for the underlying moves  that hold in the groupoid $\C(\Sigma)$ of cut systems are inherited.\label{relfmhatrm}

 	\item[(C)] Uncolorings and admissible moves commute in the obvious way (provided, of course, that both compositions exist).\label{relfmhatrc}
 	
 	\end{enumerate}
 
 \begin{definition}
 	For an extended surface $\Sigma$, we call the category $\Chat(\Sigma)$
  the \emph{category  of colored cut systems $\Sigma$}.
  \end{definition}
 The category $\Chat(\Sigma)$ comes with a canonical functor \begin{align} \label{eqncanfunctorpi} Q_\Sigma : \Chat(\Sigma) \to \C(\Sigma) \end{align} which forgets the coloring, sends uncolorings to identities and admissible moves to the underlying moves. Note that (RM) ensures that this defines really a functor.

   We can generalize Proposition~\ref{propmark}, in which we stated that cut systems 
   can be \emph{functorially} assigned to surfaces, to \emph{colored} cuts: Indeed, a morphism $f: \Sigma \to \Sigma'$ induces a functor $\Chat(f):\Chat(\Sigma)\to\Chat(\Sigma')$
 which satisfies $Q_{\Sigma'}\Chat(f)=\C(f)Q_\Sigma$ and sends colored cuts to colored cuts and uncolored cuts to uncolored cuts. For the definition of $\Chat(s):\Chat(\Sigma)\to\Chat(\Sigma')$ for a sewing $s : \Sigma \to \Sigma'$, we additionally have to prescribe that boundary components which are glued together through the sewing $s$ give rise to a \emph{colored} cut (assigning an \emph{uncolored} cut instead would generally violate \eqref{eqncoloredcondition}).
 
 \begin{proposition}\label{propnatdec0}
 	Colored 	cut systems on extended surfaces  naturally form a symmetric monoidal functor
 	\begin{align} \Chat : \Surf \to \Cat  \ . 
 	\end{align} 
 \end{proposition}

 The important fact about the categories of colored cut systems is that they are still contractible despite the newly introduced uncoloring morphisms:

 \begin{theorem}\label{thmcolorlego}
 	For any extended surface $\Sigma$, the category $\Chat(\Sigma)$ of colored cut systems on $\Sigma$ is contractible.
 	\end{theorem}
 
 \begin{proof}
 \begin{pnum}
 	
 	\item We consider the subcategory $\catf{U}(\Sigma) \subset \Chat(\Sigma)$ with the same objects, but whose morphisms are generated only by uncolorings and define $\catf{L}(\Sigma)$ as the $\infty$-categorical localization of the $\infty$-category $N\Chat(\Sigma)$ at $\catf{U}(\Sigma)$. 
 	Now the canonical map $  N\Chat(\Sigma)\to \catf{L}(\Sigma)$ induces an equivalence
 	\begin{align}  |N\Chat(\Sigma)| \xrightarrow{\ \simeq \ } |\catf{L}(\Sigma)| \ . 
 	\end{align}
 	Moreover, $\catf{L}(\Sigma)$ is a Kan complex, i.e.\ an $\infty$-groupoid since $\catf{U}(\Sigma)$ contains all the non-invertible morphisms in $\Chat(\Sigma)$. 
 	For concreteness, we choose in this proof the Dwyer-Kan model \cite{dwyerkan} for  localization.\label{Kanproof1} 
 	
 	\item Since the functor $Q_\Sigma$ from \eqref{eqncanfunctorpi} sends all morphisms in $\catf{U}(\Sigma)$ to identities, it induces a simplicial map
 	$\omega _\Sigma : \catf{L}(\Sigma) \to N\C(\Sigma)$.
 	In this step, we prove that $\omega _\Sigma$ is a Kan fibration, i.e.\
 	$\omega _\Sigma$ allows for  solutions to the lifting problems
 	\begin{equation} \label{eqnliftingproblemKan}
 	\begin{tikzcd}
 	\Lambda_k^n \ar[swap]{dd}{\iota} \ar[]{rr}{\xi} & & \catf{L}(\Sigma)  \ar{dd}{\omega _\Sigma}   \\
 	& & \\
 	\Delta^n \ar{rr}{\sigma } \ar[dashed]{rruu}{\widetilde \sigma}  & & N\C(\Sigma) \ , 
 	\end{tikzcd} 
 	\end{equation}
 	where $\iota : \Lambda_k^n \to \Delta^n$ for $n\ge 0$ and $0\le k\le n$ are the horn the inclusions. 
 	\begin{itemize}
 		\item 
 	For $n=0$, this just means that $\omega _\Sigma$ is surjective on 0-simplices, which we can easily observe to be true.

 	\item For $n=1$,
 	we need to solve the lifting problem
 		\begin{equation}\label{eqnliftingproblemn1}
 	\begin{tikzcd}
 	0 \ar[swap]{dd}{\iota} \ar[]{rr}{U} & & \catf{L}(\Sigma)  \ar{dd}{\omega _\Sigma}   \\
 	& & \\
 	\Delta^1 \ar{rr}{\mu } \ar[dashed]{rruu}{\widetilde \mu}  & & N\C(\Sigma) \   
 	\end{tikzcd} 
 	\end{equation}
 	(the lifting problem for the inclusion $1\to \Delta^1$ can be solved by passing to inverse morphisms).
 	As one easily sees, it suffices to prove this lifting property if $\mu$ is a move, i.e.\ $\barF$ or $\barS$, or an inverse of these moves.
 	We will now prove that such a lift indeed exists (we illustrate the strategy by means of an example in Figure~\ref{figliftS}):
 Let $\mu : C\to C'$ be a move in $\C(\Sigma)$. On each connected component of $\Sigma$, the cut system $C$ has a boundary component or a cut that will not be affected by $\mu$ in the sense of (AM) on page~\pageref{notationadmissiblemoves}
 	  (this can be observed for both $\barF$ and $\barS$). This implies that for any $U $ with $Q_\Sigma(U)=C$, we find a colored cut system $V$ connected to $U$ by a zigzag of uncolorings (generally, the uncolorings themselves will not be enough; zigzags of them are needed)
 	  such that $\mu$ does not affect the colored cuts of $V$.
 	  As a consequence, $\mu$
 	  induces an admissible move $\widehat \mu : V\to W$. In particular, $Q_\Sigma(\widehat \mu)=\mu$.
 	  Then the zigzag
 	  \begin{align}
 	 U\xleftrightarrow{ \ \ \ }     V \xrightarrow{\   \widehat{\mu}   \ }W
  	  \end{align}
 	 in $\Chat(\Sigma)$ is mapped to $\mu$ under $Q_\Sigma$ because $Q_\Sigma$ sends uncolorings to identities, and it gives us a 1-morphism $\widetilde \mu : U \to W$ in $\catf{L}(\Sigma)$ with $\omega _\Sigma(\widetilde \mu)=\mu$. 
 	 This shows that the $\widetilde \mu$ solves the lifting problem. The same argument applies to the inverses of moves.

 	 	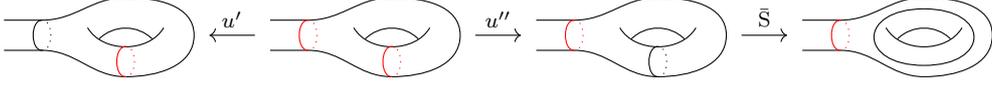
\begin{figure}[h]
 	 	\centering
 	 	\begin{tikzpicture}[scale=0.20, style/.style={circle, inner sep=0pt,minimum size=0mm},
 	 	poin/.style={circle, inner sep=0.2pt,minimum size=0.2mm},decoration={
 	 		markings,
 	 		mark=at position 0.5 with {\arrow{>}}}]
 	 	\node [style=none] (37) at (-24.75, 3.5) {};
 	 	\node [style=none] (38) at (-21.25, 3.5) {};
 	 	\node [style=none] (39) at (-25.5, 4.25) {};
 	 	\node [style=none] (40) at (-20.5, 4.25) {};
 	 	\node [style=none] (41) at (-23, 3) {};
 	 	\node [style=none] (43) at (-18.5, 3.75) {};
 	 	\node [style=none] (44) at (-23, 1) {};
 	 	\node [style=none] (45) at (-28.5, 4.75) {};
 	 	\node [style=none] (46) at (-28.5, 2.75) {};
 	 	\node [style=none] (47) at (-23, 6.25) {};
 	 	\node [style=none] (75) at (-17.5, 3.75) {};
 	 	\node [style=none] (76) at (-14.5, 3.75) {};
 	 	\node [style=none] (77) at (-16, 4.8) {\footnotesize$u'$};
 	 	\node [style=none] (78) at (-31, 4.75) {};
 	 	\node [style=none] (79) at (-31, 2.75) {};
 	 	\node [style=none] (80) at (-23, 3) {};
 	 	\node [style=none] (81) at (-23, 1) {};
 	 	\node [style=none] (82) at (-7.25, 3.5) {};
 	 	\node [style=none] (83) at (-3.75, 3.5) {};
 	 	\node [style=none] (84) at (-8, 4.25) {};
 	 	\node [style=none] (85) at (-3, 4.25) {};
 	 	\node [style=none] (86) at (-5.5, 3) {};
 	 	\node [style=none] (87) at (-1, 3.75) {};
 	 	\node [style=none] (88) at (-5.5, 1) {};
 	 	\node [style=none] (89) at (-11, 4.75) {};
 	 	\node [style=none] (90) at (-11, 2.75) {};
 	 	\node [style=none] (91) at (-5.5, 6.25) {};
 	 	\node [style=none] (92) at (0, 3.75) {};
 	 	\node [style=none] (93) at (3, 3.75) {};
 	 	\node [style=none] (94) at (1.5, 4.8) {\footnotesize$u''$};
 	 	\node [style=none] (95) at (-13.5, 4.75) {};
 	 	\node [style=none] (96) at (-13.5, 2.75) {};
 	 	\node [style=none] (97) at (10.25, 3.5) {};
 	 	\node [style=none] (98) at (13.75, 3.5) {};
 	 	\node [style=none] (99) at (9.5, 4.25) {};
 	 	\node [style=none] (100) at (14.5, 4.25) {};
 	 	\node [style=none] (101) at (12, 3) {};
 	 	\node [style=none] (102) at (16.5, 3.75) {};
 	 	\node [style=none] (103) at (12, 1) {};
 	 	\node [style=none] (104) at (6.5, 4.75) {};
 	 	\node [style=none] (105) at (6.5, 2.75) {};
 	 	\node [style=none] (106) at (12, 6.25) {};
 	 	\node [style=none] (107) at (17.5, 3.75) {};
 	 	\node [style=none] (108) at (20.5, 3.75) {};
 	 	\node [style=none] (109) at (19, 4.9) {\footnotesize$\barS$};
 	 	\node [style=none] (110) at (4, 4.75) {};
 	 	\node [style=none] (111) at (4, 2.75) {};
 	 	\node [style=none] (112) at (27.75, 3.5) {};
 	 	\node [style=none] (113) at (31.25, 3.5) {};
 	 	\node [style=none] (114) at (27, 4.25) {};
 	 	\node [style=none] (115) at (32, 4.25) {};
 	 	\node [style=none] (116) at (29.5, 3) {};
 	 	\node [style=none] (117) at (34, 3.75) {};
 	 	\node [style=none] (118) at (29.5, 1) {};
 	 	\node [style=none] (119) at (24, 4.75) {};
 	 	\node [style=none] (120) at (24, 2.75) {};
 	 	\node [style=none] (121) at (29.5, 6.25) {};
 	 	\node [style=none] (125) at (21.5, 4.75) {};
 	 	\node [style=none] (126) at (21.5, 2.75) {};
 	 	\node [style=none] (127) at (-11, 4.75) {};
 	 	\node [style=none] (128) at (-11, 2.75) {};
 	 	\node [style=none] (129) at (-5.5, 3) {};
 	 	\node [style=none] (130) at (-5.5, 1) {};
 	 	\node [style=none] (131) at (-5.5, 3) {};
 	 	\node [style=none] (132) at (-5.5, 1) {};
 	 	\node [style=none] (133) at (6.5, 4.75) {};
 	 	\node [style=none] (134) at (6.5, 2.75) {};
 	 	\node [style=none] (135) at (6.5, 4.75) {};
 	 	\node [style=none] (136) at (6.5, 2.75) {};
 	 	\node [style=none] (137) at (24, 4.75) {};
 	 	\node [style=none] (138) at (24, 2.75) {};
 	 	\node [style=none] (139) at (24, 4.75) {};
 	 	\node [style=none] (140) at (24, 2.75) {};
 	 	\node [style=none] (149) at (29.5, 1.75) {};
 	 	\node [style=none] (150) at (29.5, 5.5) {};
 	 	\draw [bend left=45] (37.center) to (38.center);
 	 	\draw [bend right=60] (39.center) to (40.center);
 	 	\draw [bend right=90,dotted] (46.center) to (45.center);
 	 	\draw [in=180, out=180] (46.center) to (45.center);
 	 	\draw [in=0, out=-90] (43.center) to (44.center);
 	 	\draw [in=0, out=180] (44.center) to (46.center);
 	 	\draw [in=0, out=90] (43.center) to (47.center);
 	 	\draw [in=0, out=-180] (47.center) to (45.center);
 	 	\draw [style=rightarrow,<-] (75.center) to (76.center);
 	 	\draw (45.center) to (78.center);
 	 	\draw (46.center) to (79.center);
 	 	\draw [red, in=0, out=0,dotted] (80.center) to (81.center);
 	 	\draw [red, bend left=270] (80.center) to (81.center);
 	 	\draw [bend left=45] (82.center) to (83.center);
 	 	\draw [bend right=60] (84.center) to (85.center);
 	 	\draw [in=0, out=-90] (87.center) to (88.center);
 	 	\draw [in=0, out=180] (88.center) to (90.center);
 	 	\draw [in=0, out=90] (87.center) to (91.center);
 	 	\draw [in=0, out=-180] (91.center) to (89.center);
 	 	\draw [style=rightarrow,->] (92.center) to (93.center);
 	 	\draw (89.center) to (95.center);
 	 	\draw (90.center) to (96.center);
 	 	\draw [bend left=45] (97.center) to (98.center);
 	 	\draw [bend right=60] (99.center) to (100.center);
 	 	\draw [bend right=90,dotted] (103.center) to (101.center);
 	 	\draw [bend left=90] (103.center) to (101.center);
 	 	\draw [in=0, out=-90] (102.center) to (103.center);
 	 	\draw [in=0, out=180] (103.center) to (105.center);
 	 	\draw [in=0, out=90] (102.center) to (106.center);
 	 	\draw [in=0, out=-180] (106.center) to (104.center);
 	 	\draw [style=rightarrow,->] (107.center) to (108.center);
 	 	\draw (104.center) to (110.center);
 	 	\draw (105.center) to (111.center);
 	 	\draw [bend left=45] (112.center) to (113.center);
 	 	\draw [bend right=60] (114.center) to (115.center);
 	 	\draw [in=0, out=-90] (117.center) to (118.center);
 	 	\draw [in=0, out=180] (118.center) to (120.center);
 	 	\draw [in=0, out=90] (117.center) to (121.center);
 	 	\draw [in=0, out=-180] (121.center) to (119.center);
 	 	\draw (119.center) to (125.center);
 	 	\draw (120.center) to (126.center);
 	 	\draw [red, in=0, out=0,dotted] (127.center) to (128.center);
 	 	\draw [red, bend left=270] (127.center) to (128.center);
 	 	\draw [red, in=0, out=0,dotted] (131.center) to (132.center);
 	 	\draw [red, bend left=270] (131.center) to (132.center);
 	 	\draw [red, in=0, out=0,dotted] (135.center) to (136.center);
 	 	\draw [red, bend left=270] (135.center) to (136.center);
 	 	\draw [red, in=0, out=0,dotted] (139.center) to (140.center);
 	 	\draw [red, bend left=270] (139.center) to (140.center);
 	 	\draw [in=0, out=0, looseness=3.00] (150.center) to (149.center);
 	 	\draw [in=180, out=-180, looseness=3.00] (150.center) to (149.center);
 	 	\end{tikzpicture}
 	 		\caption{Example for the construction of the lifts needed in \eqref{eqnliftingproblemn1} in the case of the $\barS$-move.
 	 			The 
 	 		$\barS$-move can be applied to a subsurface of the shape of a torus with one boundary component and one cut. We will assume that the surface actually continues at this boundary component, so it will play the role of a cut (if the surface ends there, the situation would be easier as we will explain after covering the present situation). Suppose now we want to lift $\barS$ to $\catf{L}(\Sigma)$ with the start value given by the first picture (colored cuts are drawn in red; orientation of cuts and marked points are suppressed as in Figure~\ref{figbarfbars}).
 	 		If we start with this colored cut system, then $\barS$ will not induce an admissible move because the cut that is replaced with the transversal cut by the $\barS$-move is colored. 
 	 		But by a {zigzag} of uncolorings (denoted by $u'$ and $u''$)
 	 		we arrive at the third colored cut system from the left. Now $\barS$ will induce an admissible move. This way, we obtain the desired lift of $\barS$.
 	 		As just mentioned, if the hole of the torus belongs to a boundary component, the situation simplifies: We forget all colors (which is allowed thanks to the presence of a boundary component) and lift $\barS$ directly. 	 	}
 	 	\label{figliftS}
 	 \end{figure}

 	 \item In order to solve the lifting problem~\eqref{eqnliftingproblemKan} for $n\ge 2$, we first observe that 
 	 the horn $\xi : \Lambda_k^n \to \catf{L}(\Sigma)$ admits a filler $\widetilde \sigma : \Delta^n \to \catf{L}(\Sigma)$, i.e.\ $\widetilde \sigma \iota = \xi$, because $\catf{L}(\Sigma)$ is a Kan complex by \ref{Kanproof1}. 
 	 Moreover,
 	 \begin{align} \omega _\Sigma \widetilde \sigma \iota = \omega _\Sigma \xi = \sigma \iota \ . 
 	 \end{align}
 	 In other words,
 	 both $\omega _\Sigma \widetilde \sigma$ and $\sigma$ fill the horn $\omega _\Sigma \xi : \Lambda_k^n \to N\C(\Sigma)$ in the nerve of the  1-groupoid $\C(\Sigma)$. Hence, they are equal, and we can conclude that the lifting problem~\eqref{eqnliftingproblemKan} can be solved for $n\ge 2$.

 	 \end{itemize} 
  This finishes the proof that $\omega _\Sigma$ is 
  a Kan fibration. \label{Kanproof2}

 \item 	 
 The fiber of $\omega _\Sigma : \catf{L}(\Sigma) \to N\C(\Sigma)$ over $C \in \C(\Sigma)$ can be identified with the $\infty$-localization of the nerve $NQ_\Sigma^{-1}(C)$ of the fiber $Q_\Sigma^{-1}(C)$ of $Q_\Sigma : \Chat(\Sigma) \to \C(\Sigma)$ over $C$ at all uncolorings in that fiber
 (this follows from the fact that the localization that led from $\Chat(\Sigma)$ to $\catf{L}(\Sigma)$ just happens in the fibers of $Q_\Sigma$).
 Therefore, we conclude  $|\omega _\Sigma^{-1}(C)|\simeq 	|NQ_\Sigma^{-1}(\Sigma)|$.
 But $Q_\Sigma^{-1}(C)$ has an initial object, namely the one obtained by coloring \emph{all} cuts of $C$. As a result, $Q_\Sigma^{-1}(C)$ is a contractible category. This implies $|\omega _\Sigma^{-1}(C)|\simeq \star$. Combining this with \ref{Kanproof2} and the long exact sequence for homotopy groups, we conclude that
 \begin{align}
 |\omega _\Sigma| : |\catf{L}(\Sigma)| \xrightarrow{\ \simeq \ }    |N\C(\Sigma)|
 \end{align}
 is an equivalence.\label{Kanproof3}

 	\end{pnum}
 From \ref{Kanproof1} and \ref{Kanproof3} we obtain
 \begin{align}
 |N\Chat(\Sigma)|\simeq  |N\C(\Sigma)| \ . 
 \end{align}
 Now the Theorem follows from Theorem~\ref{thmbakifmcut} that asserted that $\C(\Sigma)$ is contractible. 
 \end{proof}

 In this subsection we have, so far, replaced the groupoid of cut systems of an extended surface with a category of colored cut systems and proven that the latter category is still contractible.
 It remains to replace the groupoid of  markings with a colored analogue. 
 To this end, we will use the presentation \eqref{markingseqnviagrothendieck} of $\M(\Sigma)$ as Grothendieck construction:
 
 \begin{definition}\label{defcoloredfinemarking}
 	For an extended surface $\Sigma$,
 	we define the 
 	\emph{category $\FMhat(\Sigma)$ of colored markings on $\Sigma$}
 	as the Grothendieck construction 
 	\begin{align}
 \FMhat(\Sigma):= \int \left(    	\Chat(\Sigma) \xrightarrow{\ Q_\Sigma\ } \C(\Sigma) \xrightarrow{\  \m_\Sigma  \ } \Grpd \right)  \ , 
 	\end{align}
 	where $Q_\Sigma$ and $\m_\Sigma$  appeared in \eqref{eqncanfunctorpi} and \eqref{eqnsmallmfunctor}, respectively. 
 	\end{definition}
 
Proposition~\ref{propnatdec0} carries over to colored markings:
  \begin{proposition}\label{propnatdec}
 	Colored 	markings  on extended surfaces  naturally form a symmetric monoidal functor
 	\begin{align}
 	\Mhat : \Surf \to \Cat \ . 
 	\end{align} 
 \end{proposition}

 Fortunately, the contractibility of $\FMhat(\Sigma)$ will follow directly from  Theorem~\ref{thmcolorlego}
 and \emph{Thomason's Theorem} \cite[Theorem~1.2]{thomason} that states that for a functor $F:\cat{A}\to\Cat$, the natural map
 \begin{align}
 \hocolimsub{a\in \cat{A}} NF(a) \ra{\simeq} N \int F \label{eqnthomasonthm}
 \end{align} is an equivalence.
 
 \begin{theorem}\label{thmmhatcontractible}
 	For an extended surface $\Sigma$,
 the	category $\FMhat(\Sigma)$ of colored markings on $\Sigma$
 	is contractible.
 	\end{theorem}

 \begin{proof}
 	By Thomason's Theorem \eqref{eqnthomasonthm}
 	we obtain
 	\begin{align}
 	N\FMhat(\Sigma) = N \int \left(   	\Chat(\Sigma) \xrightarrow{\ Q_\Sigma\ } \C(\Sigma) \xrightarrow{\  \m_\Sigma  \ } \Grpd  \right) \simeq \hocolimsub{U\in\Chat(\Sigma)} N\m_\Sigma Q_\Sigma(U) \ . \label{eqnmhatcontractible}
 	\end{align}
 	It follows from \eqref{eqnmsigmacontrac} that $N\m_\Sigma Q_\Sigma(U)$ is equivalent to a point. As a consequence, the right hand side of \eqref{eqnmhatcontractible} is equivalent to the homotopy colimit of the constant diagram with value $\star$ over $\Chat(\Sigma)$, but the latter is given by $N \Chat(\Sigma)$ which have already proven to be contractible in Theorem~\ref{thmcolorlego}.
 	\end{proof}

 \subparagraph{The $\infty$-groupoid of colored markings.}
 The contractibility result from Theorem~\ref{thmmhatcontractible} is a substantial part of the effort needed for the construction of the modular functor in the next section.
 Phrased differently, it tells us that the $\infty$-groupoid $\K(\Sigma)$ obtained by localizing the category $\Mhat(\Sigma)$ of colored markings on an extended surface at all uncolorings is contractible.

\spaceplease 
\section{Construction of the modular functor\label{secconstruc}}
Having defined a  category 
 of colored markings on an extended surface,
we may now finally construct  the modular functor with values in chain complexes.
To this end, 
recall that we have defined in Section~\ref{derivedconformalblockssec} marked blocks that do not only depend on the surface and the boundary label, but also on the auxiliary datum of a \emph{marking}. In the vector space valued situation, the standard procedure is to extend the definition to morphisms of markings, i.e.\ to construct  a functor out of $\M(\Sigma)$ for each surface $\Sigma$. This amounts to relating the structure that is present on a modular category with the moves between different  markings.
Unfortunately, for our differential graded marked blocks, 
this does not seem to be possible directly. The key problem is that Lyubashenko's $\S$-transformation that is used for the definition of vector space valued marked blocks on the $\S$-move does not seem to generalize directly to differential graded marked blocks. Similar, albeit  less severe problems exist for the $\F$-move.

Our categories $\Mhat(\Sigma)$ of \emph{colored} markings precisely solve this problem: 
We show that the marked blocks  extend to functors $\Mhat(\Sigma)\to\Ch$ on categories of colored markings (Theorem~\ref{thmfunctormhat}) such that all uncolorings are sent to equivalences, i.e.\ the functors descend to the $\infty$-localization $\K(\Sigma)$ of $\Mhat(\Sigma)$ at all uncolorings.
 The idea for the definition of the functor $\Mhat(\Sigma)\to\Ch$ is to work with a version of marked blocks which are glued via homotopy coends at colored cuts and via ordinary coends at uncolored ones. 
We prove in Proposition~\ref{propcolorconfblocks} that this is equivalent to the marked blocks we had originally defined
in  Section~\ref{derivedconformalblockssec}. 
Having established this `mixed' definition of a marked block for any colored marking, we will, roughly, apply the moves $\F$ and $\S$ only to those parts of the marked block which coincide with the classical marked block with values in vector spaces using the `classical definitions' there.
All the rest of the information is contained in the uncoloring maps. 
The details of this construction will be discussed in the proof of Theorem~\ref{thmfunctormhat}.

 The functors $\Mhat(\Sigma)\to\Ch$ descend to the category obtained by gluing all categories of colored markings for different surfaces together via the Grothendieck construction (Proposition~\ref{propnatgdc}). The final remaining step in the construction of the modular functor will then be a homotopy left Kan extension (Section~\ref{secleftkan}).

For presentation purposes, we will first treat anomaly-free modular categories and then comment on the anomalous case in Remark~\ref{anomalouscase}.	 This makes sense because the projectivity of the mapping class group actions in the anomalous case will be of the same type for the linear and differential graded setting. 	
		
		\spaceplease
\subsection{Extension of the definition  marked blocks to colored markings}
For a given pivotal finite tensor category $\cat{C}$,
let $\underline{X}$ be a family of projective boundary labels for an extended surface $\Sigma$;
of course,
we also allow the case that $\Sigma$ is closed.
If $\Gamma$ is a  marking on $\Sigma$, then we have defined in Section~\ref{derivedconformalblockssec} a marked block $\block_\cat{C}^{\Sigma,\Gamma}(\underline{X})$ depending on $\Gamma$. This chain complex was defined via an iterated homotopy coend over $\Proj \cat{C}$ with one homotopy coend for each cut in $\Gamma$. 

Let now $\Lambda$ be a \emph{colored} marking on $\Sigma$. 
By Definition~\ref{defcoloredfinemarking} a colored marking $\Lambda$ is a pair $(U,\Gamma)$
 of a colored cut system $U$ and a marking $\Gamma$ which both have the same underlying cut system. In other words, $\Lambda$ arises from $\Gamma$ by declaring some cuts to be colored cuts in a way prescribed by $U$.
Recall that on each connected component of $\Sigma$, the number of colored cuts plus the number  of boundary components has to be at least one; this is required by \eqref{eqncoloredcondition}.
 
We now define $\mfcm(\Sigma,\underline{X},\Lambda)$ to be the chain complex that we obtain from $\block_\cat{C}^{\Sigma,\Gamma}(\underline{X})$ by replacing all homotopy coends corresponding to uncolored cuts by ordinary coends while the homotopy coends corresponding to the colored cuts remain unaffected. 

In formulae, this is expressed as follows: Let $\Sigma_\Lambda$ be the extended surface obtained by cutting $\Sigma$ along all colored cuts of $\Lambda$. Then the marking $\Gamma$ underlying $\Lambda$ gives rise to a marking $\Gamma_\Lambda$ on $\Sigma_\Lambda$ (we are recalling here notation already established on page~\pageref{notationadmissiblemoves}).
If $\Lambda$ has $q$ colored cuts, we arrive at the isomorphism
\begin{align}
\mfcm(\Sigma,\underline{X},\Lambda) \cong \lint^{P_1,\dots,P_q \in \Proj \cat{C}} \ordblock_\cat{C}^{\Sigma_\Lambda,\Gamma_\Lambda}    (\underline{X},P_1,P_1,\dots,P_q,P_q) \ . \label{eqndefinitionalequality} 
\end{align}
This equality just expresses in formulae
the definition of $\mfcm(\Sigma,\underline{X},\Lambda)$ that was just given in words. The only non-trivial fact used here is  that replacing homotopy coends by ordinary coends leads to vector space valued marked blocks $\ordblock_\cat{C}$ (a fact that was explained in Section~\ref{secaugmentationfibration}). 
In \eqref{eqndefinitionalequality} we see the  `mixed' definition of marked blocks mentioned in the introduction of this section made precise: Homotopy coends are used for gluing at colored cuts, ordinary coends at uncolored cuts. This allows us to express  parts of this chain complex by means of the vector spaces $\ordblock_\cat{C}$ thanks to the results of Section~\ref{secaugmentationfibration}.

\begin{proposition}\label{propcolorconfblocks}
	Let $\cat{C}$ be a pivotal finite tensor category and $\underline{X}$ a projective boundary label for an extended surface $\Sigma$ with colored marking $\Lambda$ with underlying marking $\Gamma$.
		Then the canonical map from homotopy coends to ordinary coends, applied to all uncolored cuts of $\Gamma$, induces a trivial fibration
	\begin{align}
	 \varepsilon_\Lambda : \block_\cat{C}^{\Sigma,\Gamma}(\underline{X})   \xrightarrow{\ \simeq \ }    \mfcm(\Sigma,\underline{X},\Lambda) \ . \label{eqncanmapcolmark}
	\end{align}   
	\end{proposition}

\begin{proof}
	Using \eqref{eqndefinitionalequality} and the symbols introduced there, the map in question is the map
	\begin{align}
 \lint^{P_1,\dots,P_q \in \Proj \cat{C}} \block_\cat{C}^{\Sigma_\Lambda,\Gamma_\Lambda}    (\underline{X},P_1,P_1,\dots,P_q,P_q)    \to  \lint^{P_1,\dots,P_q \in \Proj \cat{C}} \ordblock_\cat{C}^{\Sigma_\Lambda,\Gamma_\Lambda}    (\underline{X},P_1,P_1,\dots,P_q,P_q) 
	\end{align}\normalsize
	induced by the augmentation fibration 
	$ \block_\cat{C}^{\Sigma_\Lambda,\Gamma_\Lambda}    (\underline{X},P_1,P_1,\dots,P_q,P_q) \to \ordblock_\cat{C}^{\Sigma_\Lambda,\Gamma_\Lambda}    (\underline{X},P_1,P_1,\dots,P_q,P_q)
$
	from Proposition~\ref{proptrivfibbdy} which is a trivial fibration because the definition of colored markings ensures that $\Sigma_\Lambda$ has at least one boundary component in every connected component. This proves the assertion.
	\end{proof}
	
	The proof showed that the map \eqref{eqncanmapcolmark} is actually induced by the augmentation fibration from Section~\ref{secaugmentationfibration}, but only applied to a selected subfamily of the cuts prescribed by the coloring. Therefore, we will refer to the trivial fibration \eqref{eqncanmapcolmark} as a \emph{partial augmentation fibration}.

\begin{corollary}[Uncoloring maps]\label{coruncoloringmaps}
	Let $\cat{C}$ be a pivotal finite tensor category and $\underline{X}$ a projective boundary label for an extended surface $\Sigma$ with colored marking $\Lambda$. 
	Any uncoloring $\Lambda \to \Omega$ induces a trivial fibration
	\begin{align}
	\mfcm(\Sigma,\underline{X},\Lambda)   \xrightarrow{\ \simeq \ }    \mfcm(\Sigma,\underline{X},\Omega) \ . 
	\end{align} 
	induced by the canonical map from homotopy coends to ordinary coends applied to all cuts that become uncolored through the uncoloring.
	We refer to this map as uncoloring map.  
\end{corollary}

\begin{proof}
	The uncoloring maps fit into the commutative triangle
	\begin{equation}
	\begin{tikzcd}
  &  \block_\cat{C}^{\Sigma,\Gamma}(\underline{X}) \ar[swap]{ldd}{\varepsilon_\Lambda}  \ar[]{rdd}{\varepsilon_\Omega} &    \\
	& & \\
		\mfcm(\Sigma,\underline{X},\Lambda)  \ar{rr}{}  & &  \mfcm(\Sigma,\underline{X},\Omega) \   
	\end{tikzcd} 
	\end{equation}
	featuring the partial augmentation fibrations from Proposition~\ref{propcolorconfblocks}.
	Therefore, the uncoloring map needs to be an epimorphism. By the 2-out-3 property it is also an equivalence.
	\end{proof}

Step by step, we will now define the marked blocks \eqref{eqndefinitionalequality} on the \emph{morphisms} of the category of colored markings. 
The strategy is based on a trade-off:  Using the more complicated version of the Lego-Teichmüller game based on \emph{colored cuts} will enable us to import a lot of critical definitions from the vector space valued case \cite{jfcs} that in turn uses Lyubashenko's work \cite{lubacmp,luba,lubalex}
(this relatively technical construction really seems to be necessary because, as mentioned above, the `usual' maps assigned to moves between cut systems will not directly lift to differential graded marked blocks). Once we have done that, i.e.\ after Theorem~\ref{thmfunctormhat} below, we will use the Grothendieck construction and a homotopy left Kan extension to arrive at the desired differential graded modular functor.\myskip

We begin with those morphisms of colored markings that do not change the underlying cut system. This is the comparatively easy part:

\begin{lemma}\label{lemmamovesonfibers}
	Let $\cat{C}$ be a finite ribbon category and $\Sigma$ an extended surface with projective boundary label $\underline{X}$ in $\cat{C}$.
	For a colored cut system $U$ on $\Sigma$, the assignment $\m_\Sigma Q_\Sigma (U) \ni \Gamma \mapsto \mfcm(\Sigma,\underline{X},     (U,\Gamma)         )$ extends to a functor $L_U:\m_\Sigma Q_\Sigma (U) \to \Ch$. 
	\end{lemma}

\begin{proof}
		The functor $Q_\Sigma : \Chat(\Sigma)\to\C(\Sigma)$ from \eqref{eqncanfunctorpi} just forgets the coloring and sends $U$ to a cut system $C:=Q_\Sigma(U)$.
	The groupoid $\m_\Sigma Q_\Sigma (U)=\m_\Sigma(C)$ is by its definition in \eqref{eqndefmsigma} the groupoid of markings on the cut system $C$.
 Morphisms are just the morphisms of markings on $C$ that leave the cut system $C$ unaffected. Explicitly, the objects of this groupoid are markings on the genus zero surfaces that we obtain from cutting $\Sigma$ at all cuts of $C$. The morphisms are, separately for each of these genus zero surfaces, generated by the Z-move and the B-move \cite[Section~4.1]{bakifm} subject to their relations given in \cite[Section~4.7]{bakifm}.   
	From  \eqref{eqndefinitionalequality} it follows that we can define  the desired functor $L_U:\m_\Sigma Q_\Sigma (U)=\m_\Sigma(C)\to\Ch$ on these moves just as for vector space valued marked blocks (because under the homotopy coend only vector space valued marked blocks $\ordblock_\cat{C}$ appear, see~\eqref{eqndefinitionalequality}). More precisely, the Z-move is sent to the Z-isomorphism induced by the pivotal structure \cite[Definition~3.5~(i)]{jfcs}, and the B-move is sent to the B-isomorphism induced by the braiding \cite[Definition~3.5~(ii)]{jfcs}. In \cite{jfcs}, these definitions were made for fine markings, but they carry over to markings which are not necessarily fine. The statement that the $\Z$-isomorphism and the $\B$-isomorphism satisfy the needed relations in~\cite[Lemma~3.8]{jfcs}. 
	\end{proof}

In order to define marked blocks on the entire category 
of colored markings on a given extended surface,
we will use modularity:

\begin{theorem}\label{thmfunctormhat}
	Let $\cat{C}$ be an anomaly-free modular category and $\Sigma$ an extended surface with projective boundary label $\underline{X}$ in $\cat{C}$.
	Then the functors $\m_\Sigma Q_\Sigma (U) \to \Ch$ for $U\in \Chat(\Sigma)$ from Lemma~\ref{lemmamovesonfibers}
	induce a functor
	\begin{align}\mfcm(\Sigma,\underline{X},-)
	: \Mhat(\Sigma) \to \Ch    \label{eqnintfunctors}
	\end{align}
	that sends all morphisms to equivalences, i.e.\ it descends to the $\infty$-localization $\K(\Sigma)$ of $\Mhat(\Sigma)$ at all uncolorings.
	\end{theorem}

\begin{proof}
The category $\Mhat(\Sigma)$ was defined as a Grothendieck construction (Definition~\ref{defcoloredfinemarking}).
By the definition of the Grothendieck construction, a functor $\Mhat(\Sigma) \to \Ch$ amounts to functors $\m_\Sigma Q_\Sigma (U) \to \Ch$ for $U\in \Chat(\Sigma)$ plus a consistent set of natural transformations (that we will elaborate on in a moment).
As the needed functors $\m_\Sigma Q_\Sigma (U) \to \Ch$, we take the functors $L_U$ from Lemma~\ref{lemmamovesonfibers}.
Additionally, for any morphism $f: U \to V$ in $\Chat(\Sigma)$, we need a natural transformation
$\alpha_f$ filling the triangle
		\begin{equation}
	\begin{tikzcd}
	\m_\Sigma Q_\Sigma (U) \ar{rrd}{  L_U} \ar[swap]{dd}{\m_\Sigma Q_\Sigma(f)}   \ar[Rightarrow,shorten <= 0.5cm, shorten >= 0.8cm]{ddr}{\alpha_f}        & &      \\
	& & \Ch \ .  \\
	\m_\Sigma Q_\Sigma (V) \ar[swap]{rru}{   L_V     }  & \phantom{x} & 
	\end{tikzcd} 
	\end{equation}
	These transformations need to be compatible with the composition of morphisms in $\Chat(\Sigma)$. 
	Instead of defining $\alpha_f$ for arbitrary morphisms, we can of course also define it on generating morphisms, namely uncolorings (U) and admissible moves (AM), see page~\pageref{notationadmissiblemoves}, and verify that the relations (RU), (RM) and (C) are satisfied.

	On generators, we make the following definitions:
	\begin{enumerate}
		\item[(U)] The uncolorings are sent to the uncoloring maps from Corollary~\ref{coruncoloringmaps}.
		
		\item[(AM)] The definition on the admissible moves induced by the $\barF$-move and the $\barS$-move in $\C(\Sigma)$ is accomplished as follows:
		\begin{enumerate}
			\item[($\barF$)]     An $\barF$-move of cut systems give rise to an admissible move $\barF: U \to V$ of colored cut systems if and only if the deleted cut is not colored. In order to obtain for $\Gamma \in \m_\Sigma Q_\Sigma (U)$ the needed isomorphism
			\begin{align}
			\alpha_{\barF} : \mfcm(\Sigma,\underline{X},  (U,\Gamma)   ) \to 
			\mfcm(\Sigma,\underline{X},  (V,\m_\Sigma(\barF) \Gamma) )\label{alphaFeqn} 
			\ , 
			\end{align}
			we can now use the F-isomorphism from \cite[Definition~3.5~(iii)]{jfcs} which uses the (ordinary) Yoneda Lemma for the morphism spaces in $\cat{C}$.

		\item[($\barS$)] The $\barS$-move of a cut system can be applied to a subsurface of $\Sigma$ of the shape of a torus with one boundary component and one cut. The move replaces  this cut by a transversal one (as depicted in Figure~\ref{figbarfbars}). It gives rise  to an admissible move $\barS: U\to V$ of colored cut systems if and only if the cut that is being replaced is not colored. 
		In order to obtain for $\Gamma \in \m_\Sigma Q_\Sigma (U)$ the needed isomorphism
		\begin{align}
		\alpha_{\barS} : \mfcm(\Sigma,\underline{X},   (U,\Gamma)    ) \to 	\mfcm(\Sigma,\underline{X},  (V,\m_\Sigma(\barS) \Gamma) ) \label{alphaSeqn}
		\end{align}
		we can now use the S-isomorphism from \cite[Definition~3.5~(v)]{jfcs} which makes use of the $\S$-transformation for the canonical coend \cite{luba}. Note that modularity enters in this step because it ensures that one can define the $\S$-transformation.
		\end{enumerate}
		\end{enumerate}
	It is seen as follows that the relations are satisfied:
	It can be easily observed that (RU) and (C) are satisfied. The relations (RM) being satisfied is  a statement about vector space valued marked blocks for the anomaly-free case (similarly to Lemma~\ref{lemmamovesonfibers}), which we will demonstrate for one the five relations in \cite[Section~7.3]{bakifm} for the definition of $\C(\Sigma)$, namely the compatibility of $\barF$ and $\barS$. 
	The induced relations for $\barF$ and $\barS$, seen as \emph{admissible} moves in $\Chat(\Sigma)$, arise by coloring (parts of) the cut systems that appear in this relation while respecting of course the definition of $\Chat(\Sigma)$. An example of such a coloring is shown in Figure~\ref{figrelSF}, and we list all other possible colorings in the caption, but they can all be treated as the one which is shown in the picture.

		\begin{figure}[h]
		\centering
		\begin{tikzpicture}[scale=0.3, style/.style={circle, inner sep=0pt,minimum size=0mm},
		poin/.style={circle, inner sep=0.2pt,minimum size=0.2mm},decoration={
			markings,
			mark=at position 0.5 with {\arrow{>}}}]
		\node [style=none] (0) at (6.75, 2.25) {};
		\node [style=none] (1) at (10.25, 2.25) {};
		\node [style=none] (2) at (6, 3) {};
		\node [style=none] (3) at (11, 3) {};
		\node [style=none] (4) at (8.5, 1.75) {};
		\node [style=none] (5) at (8.5, -0.25) {};
		\node [style=none] (6) at (3, 3.5) {};
		\node [style=none] (7) at (3, 1.5) {};
		\node [style=none] (8) at (8.5, 5) {};
		\node [style=none] (9) at (17, 2.5) {};
		\node [style=none] (10) at (24.75, 5) {};
		\node [style=none] (11) at (21.75, 2.25) {\footnotesize$\barS$};
		\node [style=none] (12) at (0.5, 3.5) {};
		\node [style=none] (13) at (0.5, 1.5) {};
		\node [style=none] (14) at (8.5, 1.75) {};
		\node [style=none] (15) at (8.5, -0.25) {};
		\node [style=none] (18) at (14, 1.5) {};
		\node [style=none] (19) at (16.5, 1.5) {};
		\node [style=none] (20) at (14, 3.5) {};
		\node [style=none] (21) at (16.5, 3.5) {};
		\node [style=none] (22) at (8.5, 1.75) {};
		\node [style=none] (23) at (8.5, -0.25) {};
		\node [style=none] (24) at (14, 15.75) {};
		\node [style=none] (25) at (14, 13.75) {};
		\node [style=none] (26) at (6.75, 14.5) {};
		\node [style=none] (27) at (10.25, 14.5) {};
		\node [style=none] (28) at (6, 15.25) {};
		\node [style=none] (29) at (11, 15.25) {};
		\node [style=none] (30) at (8.5, 14) {};
		\node [style=none] (31) at (8.5, 12) {};
		\node [style=none] (32) at (3, 15.75) {};
		\node [style=none] (33) at (3, 13.75) {};
		\node [style=none] (34) at (8.5, 17.25) {};
		\node [style=none] (35) at (0.5, 15.75) {};
		\node [style=none] (36) at (0.5, 13.75) {};
		\node [style=none] (37) at (8.5, 14) {};
		\node [style=none] (38) at (8.5, 12) {};
		\node [style=none] (39) at (14, 13.75) {};
		\node [style=none] (40) at (16.5, 13.75) {};
		\node [style=none] (41) at (14, 15.75) {};
		\node [style=none] (42) at (16.5, 15.75) {};
		\node [style=none] (43) at (23, 8.25) {};
		\node [style=none] (44) at (26.5, 8.25) {};
		\node [style=none] (45) at (22.25, 9) {};
		\node [style=none] (46) at (27.25, 9) {};
		\node [style=none] (47) at (24.75, 7.75) {};
		\node [style=none] (48) at (24.75, 5.75) {};
		\node [style=none] (49) at (19.25, 9.5) {};
		\node [style=none] (50) at (19.25, 7.5) {};
		\node [style=none] (51) at (24.75, 11) {};
		\node [style=none] (52) at (16.75, 9.5) {};
		\node [style=none] (53) at (16.75, 7.5) {};
		\node [style=none] (54) at (24.75, 7.75) {};
		\node [style=none] (55) at (24.75, 5.75) {};
		\node [style=none] (56) at (30.25, 7.5) {};
		\node [style=none] (57) at (32.75, 7.5) {};
		\node [style=none] (58) at (30.25, 9.5) {};
		\node [style=none] (59) at (32.75, 9.5) {};
		\node [style=none] (60) at (-9.25, 8.25) {};
		\node [style=none] (61) at (-5.75, 8.25) {};
		\node [style=none] (62) at (-10, 9) {};
		\node [style=none] (63) at (-5, 9) {};
		\node [style=none] (64) at (-7.5, 7.75) {};
		\node [style=none] (65) at (-7.5, 5.75) {};
		\node [style=none] (66) at (-13, 9.5) {};
		\node [style=none] (67) at (-13, 7.5) {};
		\node [style=none] (68) at (-7.5, 11) {};
		\node [style=none] (69) at (-15.5, 9.5) {};
		\node [style=none] (70) at (-15.5, 7.5) {};
		\node [style=none] (71) at (-7.5, 7.75) {};
		\node [style=none] (72) at (-7.5, 5.75) {};
		\node [style=none] (73) at (-2, 7.5) {};
		\node [style=none] (74) at (0.5, 7.5) {};
		\node [style=none] (75) at (-2, 9.5) {};
		\node [style=none] (76) at (0.5, 9.5) {};
		\node [style=none] (77) at (8.5, 17.25) {};
		\node [style=none] (78) at (8.5, 15.25) {};
		\node [style=none] (79) at (-13, 9.5) {};
		\node [style=none] (80) at (-13, 7.5) {};
		\node [style=none] (81) at (3, 15.75) {};
		\node [style=none] (82) at (3, 13.75) {};
		\node [style=none] (83) at (30.25, 9.5) {};
		\node [style=none] (84) at (30.25, 7.5) {};
		\node [style=none] (85) at (14, 3.5) {};
		\node [style=none] (86) at (14, 1.5) {};
		\node [style=none] (87) at (-2, 9.5) {};
		\node [style=none] (88) at (-2, 7.5) {};
		\node [style=none] (97) at (-7.5, 11) {};
		\node [style=none] (98) at (-7.5, 9) {};
		\node [style=none] (110) at (-7.5, 11) {};
		\node [style=none] (111) at (-7.5, 9) {};
		\node [style=none] (130) at (-7.5, 7.75) {};
		\node [style=none] (131) at (-7.5, 5.75) {};
		\node [style=none] (132) at (-7.5, 7.75) {};
		\node [style=none] (133) at (-7.5, 5.75) {};
		\node [style=none] (134) at (19.25, 9.5) {};
		\node [style=none] (135) at (19.25, 7.5) {};
		\node [style=none] (136) at (19.25, 9.5) {};
		\node [style=none] (137) at (19.25, 7.5) {};
		\node [style=none] (148) at (24.75, 6.75) {};
		\node [style=none] (149) at (24.75, 10.25) {};
		\node [style=none] (154) at (3, 3.5) {};
		\node [style=none] (155) at (3, 1.5) {};
		\node [style=none] (156) at (3, 3.5) {};
		\node [style=none] (157) at (3, 1.5) {};
		\node [style=none] (158) at (-7.5, 5) {};
		\node [style=none] (159) at (0, 2.5) {};
		\node [style=none] (160) at (-4.25, 2.25) {\footnotesize$\barF$};
		\node [style=none] (161) at (-7.5, 11.75) {};
		\node [style=none] (162) at (0, 14.75) {};
		\node [style=none] (163) at (-4.5, 14.5) {\footnotesize$\barF$};
		\node [style=none] (164) at (17, 14.75) {};
		\node [style=none] (165) at (24.75, 12) {};
		\node [style=none] (166) at (21.5, 15) {\footnotesize$\barS$};
		\draw [bend left=45] (0.center) to (1.center);
		\draw [bend right=60] (2.center) to (3.center);
		\draw [in=0, out=180] (5.center) to (7.center);
		\draw [in=0, out=-180] (8.center) to (6.center);
		\draw [->] (9.center) to (10.center);
		\draw (6.center) to (12.center);
		\draw (7.center) to (13.center);
		\draw [in=-180, out=0] (15.center) to (18.center);
		\draw (18.center) to (19.center);
		\draw (20.center) to (21.center);
		\draw [in=180, out=0] (8.center) to (20.center);
		\draw [bend left=90,dotted] (22.center) to (23.center);
		\draw [bend right=90] (22.center) to (23.center);
		\draw [red, bend left=90,dotted] (24.center) to (25.center);
		\draw [red, bend right=90] (24.center) to (25.center);
		\draw [bend left=45] (26.center) to (27.center);
		\draw [bend right=60] (28.center) to (29.center);
		\draw [in=0, out=180] (31.center) to (33.center);
		\draw [in=0, out=-180] (34.center) to (32.center);
		\draw (32.center) to (35.center);
		\draw (33.center) to (36.center);
		\draw [in=-180, out=0] (38.center) to (39.center);
		\draw (39.center) to (40.center);
		\draw (41.center) to (42.center);
		\draw [in=180, out=0] (34.center) to (41.center);
		\draw [bend left=45] (43.center) to (44.center);
		\draw [bend right=60] (45.center) to (46.center);
		\draw [in=0, out=180] (48.center) to (50.center);
		\draw [in=0, out=-180] (51.center) to (49.center);
		\draw (49.center) to (52.center);
		\draw (50.center) to (53.center);
		\draw [in=-180, out=0] (55.center) to (56.center);
		\draw (56.center) to (57.center);
		\draw (58.center) to (59.center);
		\draw [in=180, out=0] (51.center) to (58.center);
		\draw [bend left=45] (60.center) to (61.center);
		\draw [bend right=60] (62.center) to (63.center);
		\draw [in=0, out=180] (65.center) to (67.center);
		\draw [in=0, out=-180] (68.center) to (66.center);
		\draw (66.center) to (69.center);
		\draw (67.center) to (70.center);
		\draw [in=-180, out=0] (72.center) to (73.center);
		\draw (73.center) to (74.center);
		\draw (75.center) to (76.center);
		\draw [in=180, out=0] (68.center) to (75.center);
		\draw [dotted, bend left=90] (77.center) to (78.center);
		\draw [bend right=90] (77.center) to (78.center);
		\draw [dotted, bend left=90] (79.center) to (80.center);
		\draw [bend right=90] (79.center) to (80.center);
		\draw [dotted, bend left=90] (81.center) to (82.center);
		\draw [bend right=90] (81.center) to (82.center);
		\draw [red, bend left=90,dotted] (83.center) to (84.center);
		\draw [red, bend right=90] (83.center) to (84.center);
		\draw [red, bend left=90,dotted] (85.center) to (86.center);
		\draw [red, bend right=90] (85.center) to (86.center);
		\draw [red, bend left=90,dotted] (87.center) to (88.center);
		\draw [red, bend right=90] (87.center) to (88.center);
		\draw [dotted, bend left=90] (110.center) to (111.center);
		\draw [bend right=90] (110.center) to (111.center);
		\draw [dotted, bend left=90] (132.center) to (133.center);
		\draw [bend right=90] (132.center) to (133.center);
		\draw [dotted, bend left=90] (136.center) to (137.center);
		\draw [bend right=90] (136.center) to (137.center);
		\draw [in=0, out=0, looseness=3.00] (149.center) to (148.center);
		\draw [in=180, out=-180, looseness=3.00] (149.center) to (148.center);
		\draw [dotted, bend left=90] (156.center) to (157.center);
		\draw [bend right=90] (156.center) to (157.center);
		\draw [->] (158.center) to (159.center);
		\draw [->] (161.center) to (162.center);
		\draw [->] (164.center) to (165.center);
		\end{tikzpicture}
		\caption{Pictorial presentation of a colored version of the compatibility of $\barF$ and $\barS$. We would obtain another relation by coloring also the leftmost cut or only the leftmost cut. In any case, the cuts in the middle cannot be colored.}
		\label{figrelSF}
	\end{figure}
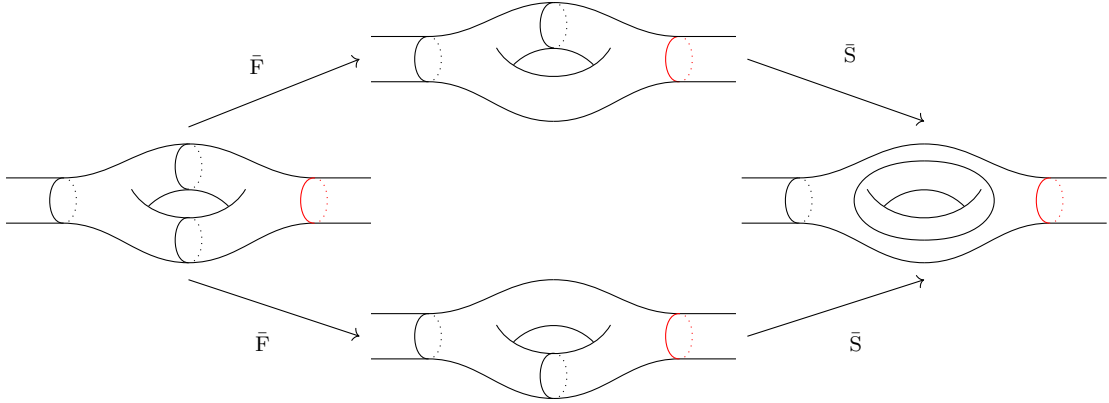

	Verifying that $\alpha_{\barF}$ and $\alpha_{\barS} $ as defined above in \eqref{alphaFeqn} and \eqref{alphaSeqn} satisfy the relation from Figure~\ref{figrelSF} now amounts to a statement about the F-isomorphism and the S-isomorphism for a vector space valued marked block for a torus with two holes. The latter can be extracted from \cite[Section~3.2]{jfcs}, where it is shown that vector space valued marked blocks yields a vector space valued functor defined on the groupoid of (fine) markings. 
	
This concludes the proof that we obtain a functor \eqref{eqnintfunctors}.	The statement that the functor 	sends all morphisms to equivalences is only non-trivial for the uncolorings. In this case, it follows from Corollary~\ref{coruncoloringmaps}.
	\end{proof}

In the next step, we prove that the constructions from Theorem~\ref{thmfunctormhat} are natural (in the appropriate sense) in the labeled extended surface.
In order to make this precise, we define 
for any modular category $\cat{C}$ the category $ \CMhat$
 as the Grothendieck construction
\begin{align}\CMhat:=\int \left(   \cat{C}\text{-}\Surf \to \Surf \xrightarrow{\ \Mhat\ } \Cat   \right) \ \label{eqncmcat} \end{align}
(this definition could  be made for any label set, of course). 
The category $\CMhat$ should be interpreted as the result of categorically gluing together the categories of colored markings over \emph{varying} $\cat{C}$-labeled surfaces.
We denote by
\begin{align} \Pi : \CMhat \to \cat{C}\text{-}\Surf   \label{eqnprojectionfunctor}\end{align}
the projection. Both categories inherit a symmetric monoidal structure from disjoint union such that $\Pi$ is a symmetric monoidal functor.

\begin{proposition}\label{propnatgdc}
	For any anomaly-free modular category $\cat{C}$,
	the functors from Theorem~\ref{thmfunctormhat} induce a symmetric monoidal functor
	\begin{align}
\mfcm \ :\ 	\CMhat \to \Ch \ . 
	\end{align}
	\end{proposition}

\begin{proof}
	By the definition of the Grothendieck construction  a functor  $\mfcm : 	\CMhat \to \Ch$ amounts to functors
	\begin{align}
	\Mhat(\Sigma) \to \Ch \label{eqnneedefunctorssurf}
	\end{align} 
	for each extended surface $\Sigma$ with projective boundary label $\underline{X}$
	and a consistent system of natural transformations
	(that we will elaborate on in a moment; afterwards, we also comment on the compatibility with the monoidal structure).
	For the functors \eqref{eqnneedefunctorssurf}, we take the functors $\mfcm(\Sigma,\underline{X},-) : \Mhat(\Sigma)\to\Ch$
	provided by Theorem~\ref{thmfunctormhat}.

	Additionally, we need to specify for any morphism $f: (\Sigma,\underline{X})\to(\Sigma',\underline{X'})$ in $\cat{C}\text{-}\Surf$
	a natural transformation
		\begin{equation}\label{eqnspecifytransf}
	\begin{tikzcd}
	\Mhat(\Sigma) \ar{rrd}{    \mfcm(\Sigma,   \underline{X} , -) } \ar[swap]{dd}{\Mhat(f)}   \ar[Rightarrow,shorten <= 0.5cm, shorten >= 0.8cm]{ddr}{\xi_f}        & &      \\
	& & \Ch \\
	\Mhat(\Sigma') \ar[swap]{rru}{     \mfcm(\Sigma',   \underline{X'}  , -)      }  & \phantom{x} & 
	\end{tikzcd} 
	\end{equation}
	such that the transformations $\xi_f$ respect the composition
	in $\cat{C}\text{-}\Surf$.
	We specify these transformations separately for sewings and mapping classes:

	\begin{itemize}
		
		\item
		Let 
		$s : (\Sigma,(\underline{X},P,P)) \to (\Sigma',\underline{X})$ be a sewing morphism in $\cat{C}\text{-}\Surf$ that glues an ingoing to an outgoing boundary component which are both labeled with $P$ (without loss of generality, it suffices to consider sewings of this form). 
		From \eqref{eqndefinitionalequality}, we can now read off that there is a canonical map
		\begin{align}
		 \mfcm(\Sigma,  ( \underline{X},P,P) , \Lambda) \to  \mfcm(\Sigma',   \underline{X} , \Mhat(s)(\Lambda))
		\end{align}
		for any colored marking $\Lambda$ in $\Sigma$
		(coming just from the definition of the homotopy coend;
		 compare also to the sewing maps \eqref{eqnsewingmaps}).
		This map can be easily seen to be natural in $\Lambda$, i.e.\ we get a natural transformation 
		\begin{equation}
		\begin{tikzcd}
		\Mhat(\Sigma) \ar{rrd}{  \mfcm(\Sigma,  ( \underline{X},P,P) , -)      } \ar[swap]{dd}{\Mhat(s)}   \ar[Rightarrow,shorten <= 0.5cm, shorten >= 0.8cm]{ddr}{\xi_s}        & &      \\
		& & \Ch \\
		\Mhat(\Sigma') \ar[swap]{rru}{ \mfcm(\Sigma,   \underline{X} , -)    }  & \phantom{x} & 
		\end{tikzcd} 
		\end{equation}
		These transformations preserve the composition of sewings strictly.

		\item It is an important observation that the definition of marked blocks for a marked surface (or their generalizations to colored markings) just depends on the incidences of (colored) cuts and markings (the relative location of these objects to each other, see also the explanations on page~\pageref{pageincidences}) --- and these incidences do not change when we act with a mapping class. 
		As a consequence of this observation, 
		for any mapping class $\phi : \Sigma \to \Sigma'$ seen as morphism $(\Sigma,\underline{X})\to (\Sigma',\underline{X'})$ in $\cat{C}\text{-}\Surf$
		the triangle 
			\begin{equation}\label{eqnstandardsurface0}
		\begin{tikzcd}
		\Mhat(\Sigma) \ar{rrd}{  \mfcm(\Sigma,   \underline{X} , -)    } \ar[swap]{dd}{\Mhat(\phi)}  & 
		& &      \\
		& & \Ch \\
		\Mhat(\Sigma') \ar[swap]{rru}{  \mfcm(\Sigma,   \underline{X'} , -)    }  & \phantom{x} & 
		\end{tikzcd} 
		\end{equation}
		commutes, so we can actually use the identity transformation to fill this triangle.

		\end{itemize}
	Since the transformations corresponding to  sewings respect composition and since the transformations corresponding to mapping classes are  identities, we conclude that the functors $  \mfcm(\Sigma,   \underline{X} , -)     : \Mhat(\Sigma)\to\Ch$ induce a functor $\CMhat \to \Ch$. 
	
	Moreover, the functor $\cat{C}\text{-}\Surf\to\Surf\ra{\Mhat}\Cat$ is symmetric monoidal (Proposition~\ref{propnatdec}),
	and the functors \eqref{eqnneedefunctorssurf} as well as the transformations \eqref{eqnspecifytransf} are compatible with this monoidal structure.
	Therefore, the functor 
	$\CMhat \to \Ch$ just constructed is also  symmetric monoidal.
	\end{proof}

\subsection{Homotopy left Kan extension\label{secleftkan}}
The functor
$\mfcm\ :\ 	\CMhat \to \Ch$
from Proposition~\ref{propnatgdc} is defined on a category of labeled extended surfaces equipped with a colored marking. In order to obtain a functor defined directly on the  category of $\cat{C}$-labeled surfaces, we use a homotopy left Kan extension along the functor 
$\Pi : \CMhat \to \cat{C}\text{-}\Surf$ from   \eqref{eqnprojectionfunctor}.

Recall that for any  
 category $\cat{S}$,
 the category $\Ch^\cat{S}$ of functors $\cat{S}\to\Ch$ can be equipped with the projective model structure. For any functor $\Phi : \cat{S}\to\cat{T}$, we obtain a Quillen pair
\begin{align}\label{eqn:introadjunction}
\xymatrix{
	\Phi_! \,:\, \Ch^{\cat{S}}        ~\ar@<0.5ex>[r]&\ar@<0.5ex>[l]  ~\Ch^{\cat{T}}  \,:\, \Phi^* \  
}
\end{align} by left Kan extension. We denote the homotopy left Kan extension, i.e.\ the left derivative of $\Phi_!$, by $\mathbb{L} \Phi_!$.

\begin{definition}\label{defdmfconstruc}
	For any anomaly-free modular category $\cat{C}$,
	we define the functor 
	\begin{align} \mfc:=\mathbb{L}\Pi_! \mfcm  \ :\ \cat{C}\text{-}\Surf \to \Ch \label{defmfc}
	\end{align} as the homotopy left Kan extension 
	of the functor $\mfcm  :   \CMhat  \to \Ch$ from Proposition~\ref{propnatgdc}
	 along the functor $\Pi : \CMhat \to \cat{C}\text{-}\Surf$ from   \eqref{eqnprojectionfunctor}.
\end{definition}

\begin{remark}\label{remjfcs}
	In \cite[Section~3.3]{jfcs} a right Kan extension along an unmarking functor $U:\mSurf\to\Surf$ from a category $\mSurf$ of marked surfaces to the category of surfaces is used for the construction of a so-called \emph{pinned block functor}.
	The use of the Kan extension in Definition~\ref{defdmfconstruc} seems similar, but the resemblance is on a purely formal level. The category $\mSurf$ in \cite{jfcs} is not equivalent to  $\CMhat$ (because $\mSurf$ actually has no non-trivial automorphisms),
	 and the unmarking functor $U$ in \cite{jfcs} is not  a projection functor like $\Pi$. It serves the entirely different purpose to translate  moves to mapping classes.
	\end{remark}

The notation $\mfc(\Sigma,\underline{X})$ suggests a relation to the complexes $\mfcm(\Sigma,\underline{X},\Lambda)$ from Theorem~\ref{thmfunctormhat} that additionally depended on a colored marking on $\Sigma$. This notation is justified by the next result:

\begin{proposition}\label{propcomputeleftkan}
	For any anomaly-free modular category $\cat{C}$ and any extended surface $\Sigma$ with projective boundary label $\underline{X}$, 
	there is a canonical equivalence
	\begin{align}
	\label{eqnpropcomputeleftkan1}	  \hocolimsub{\Lambda \in \Mhat(\Sigma)} \mfcm(\Sigma,\underline{X},\Lambda) \ra{\simeq} \mfc(\Sigma,\underline{X}) \ . 
	\end{align}
	After the choice of a colored marking $\Lambda$ on $\Sigma$, there is a canonical equivalence
	\begin{align}
	\label{eqnpropcomputeleftkan2}	\mfcm(\Sigma,\underline{X},\Lambda) \ra{\simeq} \mfc(\Sigma,\underline{X}) \ . 
	\end{align}
\end{proposition}

The proof of Proposition~\ref{propcomputeleftkan} will need a standard Lemma. First we establish some notation and terminology: 
For a functor $L:\cat{A}\to\cat{B}$ and $b\in \cat{B}$, we denote by $L/b$ the slice category of pairs $(a,f)$ of $a\in \cat{A}$ and a morphism $f:L(a)\to b$. 
A morphism $(a,f)\to (a',f')$ in $L/b$ is a morphism $g: a \to a'$ such that $f'L(g)=f$. Dually, we can define the slice category $b/L$.
A functor $L:\cat{A}\to\cat{B}$ is called \emph{homotopy final} if for each $b\in \cat{B}$ the slice category $b/L$ is contractible in the sense that $|N(b/L)|$ is equivalent to a point.

\begin{lemma}\label{lemmaUfinal}
	For any functor $F:\cat{B}\to\Cat$, the forgetful functor $\pi:\int F \to \cat{B}$ has the property that the natural functor $K_b : F(b) \to \pi/b$ for any $b\in \cat{B}$ is homotopy final.
\end{lemma}

A proof can be deduced from the more general statement \cite[Proposition~4.3.3.10]{htt} in the context of $\infty$-categories.

\begin{proof}[{\slshape Proof~of~Proposition~\ref{propcomputeleftkan}}]
	By \eqref{defmfc} and  the homotopy colimit formula for the homotopy left Kan extension we arrive at
	\begin{align}
	\mfc(\Sigma,\underline{X}) = \hocolim \left(   \Pi / (\Sigma,\underline{X}) \to \CMhat \ra{   \mfcm    }  \Ch      \right) \ , 
	\end{align}
	where $\Pi / (\Sigma,\underline{X}) \to \CMhat$ is the forgetful functor. 
	The natural functor $K_{\Sigma,\underline{X}} : \Mhat(\Sigma) \to \Pi / (\Sigma,\underline{X}) $
	(that appears for an arbitrary Grothendieck construction in Lemma~\ref{lemmaUfinal}) induces a map
	\begin{align}
 \hocolimsub{\Lambda \in \Mhat(\Sigma)}\mfcm(\Sigma,\underline{X},\Lambda)=	& \hocolim \left( \Mhat(\Sigma) \ra{K_{\Sigma,\underline{X}}}  \Pi / (\Sigma,\underline{X}) \to \CMhat \ra{   \mfcm    }  \Ch      \right)\\ 
	 \to & \hocolim \left(   \Pi / (\Sigma,\underline{X}) \to \CMhat \ra{   \mfcm    }  \Ch      \right)  =\mfc(\Sigma,\underline{X})
	\end{align}
	(the equality in the first line holds by definition of $ \mfcm $ in Proposition~\ref{propnatgdc}).
	For the proof of \eqref{eqnpropcomputeleftkan1}, it remains to prove that this map is an equivalence,
	but this follows from Lemma~\ref{lemmaUfinal} which states $K_{\Sigma,\underline{X}}$ is homotopy final, which implies that the map induced between the homotopy colimits is an equivalence (see e.g.\ \cite[Theorem~II.8.5.6]{riehl} for this standard result).
	
For the proof of \eqref{eqnpropcomputeleftkan2}, it suffices to prove that the canonical map
	\begin{align}
	\mfcm(\Sigma,\underline{X},\Lambda) \to \hocolimsub{\Lambda \in \Mhat(\Sigma)} \mfcm(\Sigma,\underline{X},\Lambda)\label{eqnrhshocolim}
	\end{align}
	is an equivalence (because then we can compose with \eqref{eqnpropcomputeleftkan1}).
	This can be concluded from the contractibility of the $\infty$-groupoid $\K(\Sigma)$ obtained by $\infty$-localization of $\Mhat(\Sigma)$ at all uncolorings:
	The right hand side of \eqref{eqnrhshocolim}
	is the homotopy colimit of the functor $\Mhat(\Sigma)\to\Ch$ from Theorem~\ref{thmfunctormhat} which, additionally, has the property that it sends all uncolorings to equivalences and hence descends to $\K(\Sigma)$ without changing the homotopy colimit (because $\infty$-localizations are homotopy final \cite[Proposition 7.1.10]{cisinski}).
	It suffices now to prove that the map $\star \to \K(\Sigma)$ selecting $\Lambda$ is homotopy final, but this follows from \cite[Corollary~4.1.2.6]{htt} because $\K(\Sigma)$ is a contractible Kan complex by  Theorem~\ref{thmcolorlego}.
\end{proof}

\begin{corollary}\label{corblockmfc}
	Let $\cat{C}$ an anomaly-free modular category. Then for any extended surface $\Sigma$ with projective boundary label $\underline{X}$ and any marking $\Gamma$ on $\Sigma$, there is a canonical equivalence
	\begin{align}
\block_\cat{C}^{\Sigma,\Gamma}(\underline{X}) \ra{\simeq}
\mfc(\Sigma,\underline{X}) \ .  \label{eqnagreementconfblock}
\end{align}
	\end{corollary}

	\begin{proof}
		We observe that by coloring all cuts of $\Gamma$ we obtain a colored marking $\Gamma^\text{c}$ such that $\mfcm(\Sigma,\underline{X},\Gamma^\text{c}) =\block_\cat{C}^{\Sigma,\Gamma}(\underline{X})$ holds by definition. Now we use \eqref{eqnpropcomputeleftkan2} from Proposition~\ref{propcomputeleftkan}. 
		\end{proof}

\begin{theorem}\label{thmalmostmain}
	Let $\cat{C}$ be an anomaly-free modular category.
	Then the functor \begin{align}
	\mfc\  :\   \cat{C}\text{-}\Surf   \to \Ch \label{eqnthedmf}
	\end{align} from Definition~\ref{defdmfconstruc} is a   modular functor with values in chain complexes
	for the category $\cat{C}$ in the sense of Definition~\ref{defdmf}.
\end{theorem}

\begin{proof} Through the construction leading to Definition~\ref{defdmfconstruc}, we have established that
	$\mfc$ is a functor  $\cat{C}\text{-}\Surf\to\Ch$.
	Moreover, for labeled extended surfaces $(\Sigma,\underline{X})$ and $(\Sigma',\underline{X'})$, we have
	$\Pi / (  \Sigma\sqcup \Sigma',\underline{X}\sqcup\underline{X'}    )\cong \Pi/(\Sigma,\underline{X})\times\Pi/(\Sigma',\underline{X'})$. 
	Since  $\mfcm :\CMhat \to \Ch$ is symmetric monoidal, we can now conclude that 
 $\mfc$ is symmetric monoidal (depending on how we model the homotopy colimits, the structure maps will just be equivalences).

	From Corollary~\ref{corblockmfc} we may conclude directly 
	 that the cylinder category of $\mfc$ is equivalent to $\Proj \cat{C}$ because marked blocks on decorated cylinders are given by the morphism spaces of $\cat{C}$ by \eqref{eqndmfoncyl}.
	
	It remains to prove  that \eqref{eqnthedmf} satisfies excision: 
	Let 
	$s : (\Sigma,(\underline{X},P,P)) \to (\Sigma',\underline{X})$ be a sewing morphism in $\cat{C}\text{-}\Surf$ that glues an ingoing to an outgoing boundary component which are both labeled with $P$. Any  fixed  marking $\Gamma$ on $\Sigma$ induces a marking $\Gamma'$ on $\Sigma'$.
	Now the naturality of the maps from Corollary~\ref{corblockmfc} gives us the commuting square (where they induce the vertical equivalences)
	\begin{equation}
	\begin{tikzcd}
	\lint^{P\in\Proj \cat{C}}   \block^{\Sigma,\Gamma}_\cat{C}(\underline{X},P,P)    \ar[swap]{dd}{  \simeq   } \ar[]{rrrrrr}{
\substack{\text{equivalence from Proposition~\ref{propexcision}} \\   \text{for excision with marking}   }     }	
 &&&& &&  \block^{\Sigma', \Gamma'}_\cat{C}(\underline{X})   \ar{dd}{\simeq}   \\
	& & \\
	\lint^{P\in\Proj \cat{C}}     \mfc(\Sigma,(\underline{X},P,P)) \ar{rrrrrr}{ \text{induced by evaluation of $\mfc$ on $s$} }  &&& &&&  \mfc(\Sigma,\underline{X})  \ . 
	\end{tikzcd} 
	\end{equation}
	The square commutes because the sewing transformations $\xi_s$ from the proof of Proposition~\ref{propnatgdc} generalize the sewing maps from Proposition~\ref{propexcision}.
	It follows that the lower horizontal map is an equivalence which proves excision.	
\end{proof}

\begin{remark}[The anomalous case]\label{anomalouscase}
	Theorem~\ref{thmalmostmain} provides --- at least in the anomaly-free case --- the modular functor needed for the   Main~Theorem~\ref{thmmain},
	and Corollary~\ref{corblockmfc} gives us the concrete 
	prescription how to compute it in terms of a marking. 
The restriction to anomaly-free modular categories
	throughout Section~\ref{secconstruc} was made for presentation purposes because the modifications needed to deal with anomalous case are analogous to the ones needed for vector space valued modular functors:
	Let us recall that for  marked blocks with values in vector spaces, the construction  in the anomalous case proceeds  precisely as in the anomaly-free case, but with the groupoid $\M(\Sigma)$ of markings on an extended surface $\Sigma$ replaced with a groupoid $\M^\C(\Sigma)$ which contains  additional central generators \cite[Section~3.2]{jfcs} and is no longer contractible. In fact, for a connected surface $\Sigma$, we have $\M^\C(\Sigma)\simeq \star // \mathbb{Z}$ by a non-canonical equivalence.
	The groupoid $\M^\C(\Sigma)$ comes with a functor $\M^\C(\Sigma)\to\M(\Sigma)$ sending the central generators to  identities. 
	We may see $\M^\C(\Sigma)$ as a central extension of $\M(\Sigma)$. 
	Now the vector space valued marked blocks will be defined on $\M^\C(\Sigma)$ instead of $\M(\Sigma)$. The central generators  will be sent to a scalar multiple of this identity. This scalar is given by $\zeta^g$, where $\zeta\in k^\times$ is the framing anomaly (Remark~\ref{remframinganomaly}) and $g$ is the genus of $\Sigma$. This allows us to interpret a functor out of $\M^\C(\Sigma)$ as a (certain type of) projective functor out of $\M(\Sigma)$. 
	
	In the differential graded setting, it is straightforward to take these central extensions into account as well: The central extension
	$\M^\C(\Sigma)\to\M(\Sigma)$ induces a central extension $\Mhat^\C(\Sigma)\to \Mhat(\Sigma)$
	of the category $\Mhat(\Sigma)$ of colored markings. As in Theorem~\ref{thmfunctormhat}, we will obtain
	for any projective boundary label $\underline{X}$ of $\Sigma$ a functor
	\begin{align}
	\mfcm( \Sigma,\underline{X},-)  : \Mhat^\C(\Sigma) \to \Ch  \label{eqnfunctorsonmarkingsext}
	\end{align}
	that sends all morphisms to equivalences.
	By the same arguments as for Proposition~\ref{propnatgdc}, it induces a symmetric monoidal functor $\CMhat^\C \to \Ch$ on the Grothendieck construction. Via homotopy left Kan extension along $\CMhat^\C \to \CMC$ induced by the functors $\Mhat^\C(\Sigma)\to\M^\C(\Sigma)$, 
	we obtain a symmetric monoidal functor $\CMC \to \Ch$. 	But since $\M(\Sigma)\simeq \star$ for every extended surface, 
	we have $\CM\simeq \cat{C}\text{-}\Surf$. Similarly, \begin{align}\label{eqnmodelcentralextension} \CMC \simeq \cat{C}\text{-}\Surfc\end{align} (this can be interpreted in the sense that $\CMC$ provides a model for the central extension $\cat{C}\text{-}\Surfc$; in fact, we could just \emph{define} $\cat{C}\text{-}\Surfc$ as $\CMC$). This allows us to see the functor $\CMC \to \Ch$ as a functor $\cat{C}\text{-}\Surfc \to \Ch$
	--- and this will give us the modular functor in the anomalous case.
	It is however not clear that its values are actually equivalent to the marked blocks. This, however, can be seen with arguments analogous to those in the proof of Proposition~\ref{propcomputeleftkan}.
	This completes the proof of the Main~Theorem~\ref{thmmain} in the general case.
\end{remark}

\begin{remark}[Relation to Lyubashenko's mapping class group representations and to the Reshetikhin-Turaev construction]\label{rellyu} Let $\cat{C}$ be a modular category.
	By construction, the zeroth homology $H_0 \mfc$ of the modular functor $\mfc:\cat{C}\text{-}\Surfc\to\Ch$ is a modular functor with values in vector spaces.
	This modular functor is (up to some technical subtleties that we will explain now) built from Lyubashenko's mapping class group representations \cite{lubacmp,luba,lubalex}: 
For an extended surface $\Sigma$ with projective boundary label $\underline{X}$, we have 
a canonical isomorphism
\begin{align}
H_0 \block_\cat{C}^{\Sigma,\Gamma}(\underline{X}) \cong \left(    \ordblock_\cat{C}^{\Sigma,\Gamma}(\underline{X}^\vee)       \right)^* \ , \label{eqnH0iso}
\end{align}
where $\block_\cat{C}^{\Sigma,\Gamma}$ denotes marked blocks (Section~\ref{derivedconformalblockssec})
and $\ordblock_\cat{C}^{\Sigma,\Gamma}$ vector space valued marked blocks (Section~\ref{secaugmentationfibration}). 
For the specific marking in Example~\ref{exblock}, this follows from \eqref{eqnexblock2} and the self-duality $\mathbb{F}^\vee\cong\mathbb{F}$ (that we also used in Remark~\ref{remsvea}), and  the general case can be played back to this special case with arguments similar to those in the proof of Proposition~\ref{proptrivfibbdy}.
If $\Sigma$ has at least one boundary component per connected component, the isomorphism \eqref{eqnH0iso} is compatible with the isomorphism $H_0 \block_\cat{C}^{\Sigma,\Gamma}(\underline{X}) \cong \ordblock_\cat{C}^{\Sigma,\Gamma}(\underline{X})$ induced by the augmentation fibration \eqref{eqnaugmentionfibrationeqn} in the sense that the triangle of isomorphisms
\begin{equation}
\begin{tikzcd}
H_0 \block_\cat{C}^{\Sigma,\Gamma}(\underline{X})  \ar{rrd}{\text{\eqref{eqnH0iso}}} \ar[swap]{dd}{  \substack{   \text{induced by}\\   \text{augmentation fibration}}}  & 
& &      \\
& & \left(    \ordblock_\cat{C}^{\Sigma,\Gamma}(\underline{X}^\vee)       \right)^* \\
\ordblock_\cat{C}^{\Sigma,\Gamma}(\underline{X}) \ar[swap]{rru}{  \text{\eqref{eqndualhom} and $\mathbb{F}^\vee\cong\mathbb{F}$}    }  & \phantom{x} & 
\end{tikzcd} 
\end{equation}
commutes.
The left hand side of \eqref{eqnH0iso} is functorial in $\Mhat^\C(\Sigma)$, but will descend to $\M^\C(\Sigma)$ such that \eqref{eqnH0iso} is a natural isomorphism of functors defined on $\M^\C(\Sigma)$.
Therefore, \eqref{eqnH0iso} will induce an isomorphism of mapping class group representations. 
Since the mapping class group representations in \cite{jfcs} are equivalent to  Lyubashenko's mapping class group representations, we conclude that in zeroth homology, $H_0\mfc$ recovers Lyubashenko's mapping class group representations.

The comparison just discussed automatically  implies the comparison to the Reshetikhin-Turaev modular functor stated in the introduction for the case the $\cat{C}$ is semisimple: In the semisimple case, $H_0 \mfc$ agrees with the dual of the Reshetikhin-Turaev modular functor (because the Lyubashenko construction generalizes the latter), and, in fact, the homology of $\mfc$ is concentrated in degree zero because in the semisimple case the statement from Proposition~\ref{proptrivfibbdy} will hold also for closed surfaces. This follows since in the semisimple case \emph{all} objects in $\cat{C}$ are projective.
\end{remark}

\small 
\spaceplease

\end{document}